\documentclass{amsart}

\usepackage{amssymb}
\usepackage{amsmath}
\usepackage{amsfonts}
\usepackage{amsthm}
\usepackage{stmaryrd}
\usepackage[all]{xy}
\usepackage{mathrsfs}
\usepackage{graphicx}
\usepackage{hyperref}
\usepackage{color}
\usepackage{multirow}
\usepackage{xcolor}
\usepackage{cleveref}

\usepackage{comment}
\usepackage[shortlabels]{enumitem}
\usepackage[mathscr]{eucal}
\usepackage{scalerel}

\usepackage{stackengine}

\usepackage{mathtools}
\usepackage{bbm}
\DeclareMathAlphabet{\mathpzc}{OT1}{pzc}{m}{it}

\usepackage{relsize}
\usepackage[bbgreekl]{mathbbol}
\usepackage{amsfonts}
\DeclareSymbolFontAlphabet{\mathbb}{AMSb} %to ensure that the meaning of \mathbb does not change
\DeclareSymbolFontAlphabet{\mathbbl}{bbold}
\newcommand{\Prism}{{\mathlarger{\mathbbl{\Delta}}}}

\numberwithin{equation}{subsection}
\newtheorem{theorem}[subsubsection]{Theorem}

\newtheorem{corollary}[subsubsection]{Corollary}
\newtheorem{lemma}[subsubsection]{Lemma}
\newtheorem{lemmadef}[subsubsection]{Definition/Lemma}
\newtheorem{proposition}[subsubsection]{Proposition}
\newtheorem{definition}[subsubsection]{Definition}
\newtheorem{claim}[subsubsection]{Claim}

\theoremstyle{definition}

\newtheorem{construction}[subsubsection]{Construction}
\newtheorem{convention}[subsubsection]{Convention}

\newtheorem{example}[subsubsection]{Example}

\newtheorem{question}[subsubsection]{Question}
\newtheorem{remark}[subsubsection]{Remark}

\newtheorem{assumption}[subsubsection]{Assumption}

\newcommand{\ra}{\rightarrow}

\def\BB{\mathbb{B}}

\def\NN{\mathbb{N}}

\def\calF{\mathcal{F}}

\def\calO{\mathcal{O}}

\def\scrD{\mathscr{D}}

\def\scrX{\mathscr{X}}

\def\ul{\underline}
\def\wt{\widetilde}
\def\wh{\widehat}

\DeclareMathOperator{\Hom}{Hom}
\DeclareMathOperator{\Ext}{Ext}

\DeclareMathOperator{\id}{id}

\DeclareMathOperator{\Spec}{Spec}
\DeclareMathOperator{\Spf}{Spf}

\DeclareMathOperator*{\motimes}{\text{\raisebox{0.25ex}{\scalebox{0.8}{$\bigotimes$}}}}
\DeclareMathOperator*{\moplus}{\text{\raisebox{0.25ex}{\scalebox{0.8}{$\bigoplus$}}}}

\newcommand{\Sh}{\mathrm{Sh}}

\newcommand{\Sym}{\mathrm{Sym}}

\newcommand{\rra}{\longrightarrow}
\newcommand{\lmt}{\longmapsto}

\newcommand{\ol}{\overline}

\newcommand{\frakp}{\mathfrak{p}}

\newcommand{\Fun}{\textrm{Fun}}

\newcommand{\et}{{\mathrm{\acute{e}t}}}
\newcommand{\pe}{{\mathrm{pro\acute{e}t}}}

\newcommand{\Spa}{\mathrm{Spa}}

\newcommand{\ct}{{\mathrm{cont}}}

\newcommand{\Ainf}{{\mathrm{A_{inf}}}}

\newcommand{\Bdr}{{\mathrm{B_{dR}^+}}}

\newcommand{\an}{{\mathrm{an}}}
\newcommand{\colim}{{\mathrm{colim}}}

\newcommand{\Coh}{{\mathrm{Coh}}}

\newcommand{\gr}{{\mathrm{gr}}}
\newcommand{\Fil}{{\mathrm{Fil}}}
\newcommand{\dR}{{\wh{\mathrm{dR}}^\an}}
\newcommand{\op}{{\mathrm{op}}}

\newcommand{\Ainfe}{{\mathrm{A_{inf,e}}}}
\newcommand{\Ainfn}{{\mathrm{A_{inf,n}}}}
\newcommand{\Bdre}{{\mathrm{B^+_{dR,e}}}}
\newcommand{\Bdrn}{{\mathrm{B^+_{dR,n}}}}
\newcommand{\Bdrm}{{\mathrm{B^+_{dR,m}}}}

\newcommand{\Mod}{{\mathrm{Mod}}}

\newcommand{\red}{{\mathrm{red}}}
\newcommand{\BBdr}{{\mathbb{B}_{\mathrm{dR}}}}
\newcommand{\BBdrp}{{\mathbb{B}^+_{\mathrm{dR}}}}

\newcommand{\DF}{{\mathscr{DF}}}

\newcommand{\tf}{{\mathrm{tf}}}

\newcommand{\Cd}{{\mathrm{Cond(Ab)}}}
\newcommand{\Sd}{{\mathrm{Solid}}}
\newcommand{\bs}{{\blacksquare}}
\begin{document}
	\title{Prismatic cohomology of rigid analytic spaces over de Rham period ring}
	\date{}\maketitle
	
	\centerline{Haoyang Guo}

	\begin{abstract}
		Inspired by Bhatt-Scholze \cite{BS19}, in this article, we introduce prismatic cohomology for rigid analytic spaces with l.c.i singularities, with coefficients over Fontaine's de Rham period ring $\Bdr$.
		
		%	In particular, we prove the comparison of it with the analytic derived de Rham complex, the \'eh cohomology, and the pro-\'etale cohomology.
	\end{abstract}
	
	\tableofcontents

\section{Introduction}
\subsection{Background and main theorems}
Let $k$ be a complete discretely valued $p$-adic field that has a perfect residue field, and let $K$ be its completed algebraic closure.
In the seminal work of Bhatt and Scholze \cite{BS19}, they introduce the notion of prisms and prismatic cohomology for a $p$-adic smooth formal scheme $X$ over $\mathcal{O}_k$.
A prism can be defined as a diagram of algebras
\[
(A \rra A/I \longleftarrow R),
\]
where $\Spf(R)$ is an open subset of $X$, and $A$ is equipped with an additional structure of a lift of Frobenius operator that satisfies various conditions (\cite[Section 2, 3]{BS19}).
The resulted cohomology theory, namely prismatic cohomology, turns out to specialize to many other important $p$-adic cohomology in the $p$-adic Hodge theory, and thus provides a powerful tool to study the relation of different cohomology associated to the given formal scheme $X$.

Let $X$ now be a rigid space over $K$, and let $\xi$ be a fixed generator of the kernel ideal for the  natural surjection $\theta:\Bdr \ra K$,
where $\Bdr$ is Fontaine's de Rham period ring (\cite{Fon94}).
In this article, we study a rational analogue of  prismatic cohomology for rigid spaces over $K$, with coefficients in $\Bdr$.
We start by defining the notion of a prism in our setting.
\begin{definition}\label{intro def}
	A $\Bdr$-prism over $X$ is defined as a diagram of $\Bdr$-algebras as below
	\[
	(B \rra B/\xi \longleftarrow R),
	\]
	where $\Spa(R)$ is an affinoid open subset of $X$, and $B$ is a $\xi$-adic complete, $\xi$-torsionfree algebra over $\Bdr$ satisfies the finite type condition as in \Cref{pm site} (i).
\end{definition}
One essential difference between our notion of prisms and that of Bhatt-Scholze in \cite{BS19} is that there is \emph{no} Frobenius structure in our definition.
Nonetheless, the resulted cohomology theory over the opposite category of prisms still enjoys many good properties when $X$ has mild singularities.

Now we state our main results.
\begin{theorem}\label{main sm}
	There is a $\Bdr$-linear cohomology theory on the category of $K$-rigid spaces 
	\[
	X \longmapsto R\Gamma_\Prism(X/\Bdr):=R\Gamma(X/\Bdr_\Prism,\mathcal{O}_\Prism),
	\] such that for an affinoid rigid space $X$ that has l.c.i singularities, it satisfies the following property (as $\mathbb{E}_\infty$-algebras over $\Bdr$):
	\begin{enumerate}[(i)]
		\item(Pro-\'etale comparison) There is a natural map 
		\[
		R\Gamma_\Prism(X/\Bdr) \rra R\Gamma_\pe(X,\BBdrp),
		\]
		where the right hand side is pro-\'etale cohomology of the deRham period sheaf as in \cite{Sch13}.
		\footnote{First introduced by Brinon  in \cite{Bri08}.} The map is an isomorphism when $X$ is smooth.
		\item(infinitesimal comparison) There is a natural map
		\[ 
		R\Gamma_{\inf}(X/\Bdr) \rra R\Gamma_\Prism(X/\Bdr),
		\]
		which factors through the subcomplex $L\eta_\xi R\Gamma_\Prism(X/\Bdr)$.
		Here the left hand side is infinitesimal cohomology of $X/\Bdr$ as in \cite[Section 7]{Guo21} (see also \Cref{subs prism to inf}).
		
		Moreover, if $X$ is smooth, the above induces an isomorphism $R\Gamma_{\inf}(X/\Bdr)\ra L\eta_\xi R\Gamma_\Prism(X/\Bdr)$.
		\item(Hodge-Tate filtration) There is a natural $\mathbb{N}^\op$-indexed increasing exhaustive filtration over the reduction $R\Gamma_\Prism(X/\Bdr)\otimes^L_\Bdr K$, whose $i$-th graded piece is $R\Gamma(X,L\wedge^i \mathbb{L}^\an_{X/K}(-i)[-i])$.
		\item(Galois invariant) Assume $X=Y_K$ for some affinoid rigid space $Y$ over $k$.
		Then there is a natural condensed algebra $\ul{R\Gamma}_\Prism(X/\Bdr)$ underlying $R\Gamma_\Prism(X/\Bdr)$, satisfying
		\[
		R\Gamma(G_k, \ul{R\Gamma}_\Prism(X/\Bdr)) \cong R\Gamma(G_k, \Bdr)\otimes^L_k R\Gamma_{\inf}(Y/k),
		\]
		where $R\Gamma(G_k, -)$ is defined as the condensed group cohomology as in \cite[Appendix 2]{Bos21}.
	\end{enumerate}
\end{theorem}

We give several comments on the result above.
\begin{remark}[Cohomology of $\BBdrp$]\label{rmk integral proetale comparison}
	Let $X$ be a smooth rigid space over $K$ of dimension $n$. 
	In this case, infinitesimal cohomology $R\Gamma_{\inf}(X/\Bdr)$ was first introduced in \cite[Section 13]{BMS} (denoted there as $R\Gamma_\mathrm{crys}(X/\Bdr)$ and is called \emph{crystalline cohomology}).
	It is locally computed by the formal completion of the de Rham complex of a smooth ambient space with respect to a closed immersion (\cite[Section 13.1]{BMS}, \cite[Theorem 4.1.1]{Guo21}).
	Shown in \cite[Theorem 13.1]{BMS}, infinitesimal cohomology admits a map to pro-\'etale cohomology $R\Gamma_\pe(X,\BBdrp)$, which is an isomorphism after tensoring with $\mathrm{B_{dR}}$.
	This in particular allows us to compute pro-\'etale cohomology of $\BBdr=\BBdrp[\frac{1}{\xi}]$ using the de Rham complex.
	It is then a natural question if one can also compute pro-\'etale cohomology of $\BBdrp$ using a de Rham style cohomology and within the category of noetherian algebras.
	
	Our \Cref{main sm}.(i) then gives an answer to the question.
	In fact, if $X=\Spa(R)$ admits a compatible system of smooth $\Bdr/\xi^e$-algebras $\wt R_e$ (which always exists locally), and let $\wt R$ be their inverse limit.
	Then $R\Gamma_\Prism(R/\Bdr)$ is isomorphic to the following $\xi$-divided de Rham complex
	\[
	\xymatrix{\wt R \ar[r]^{d~~~~~~~} & \frac{\wt \Omega^1}{\xi} \ar[r]^d & \cdots \ar[r] & \frac{\wt\Omega^{n}}{\xi^{n}}}.
	\]
	Here $\wt\Omega^1 :=\varprojlim_e \Omega_{(\wt R/\xi^e)/\Bdre}^1$ is  the inverse limit of analytic differentials of affinoid algebras $\wt R_e$ over $\Bdre$, and the map $d$ is defined as the composition 
	\[
	\wt R \rra \wt\Omega^1 \rra  \frac{\wt\Omega^1}{\xi},
	\]
	where the first map is the continuous differential over $\Bdr$, and the second map is the natural inclusion.
\end{remark}
\begin{remark}[Prismatic $\BBdrp$]
	In the recent work \cite[Construction 6.4]{BS21} Bhatt-Scholze introduce the notion of \emph{de Rham period sheaf} $\Prism_\bullet[1/p]^\wedge_I$ over the category of quasi-regular semiperfectoid algebras over a smooth $p$-complete $\mathcal{O}_K$-algebra $R$ (with quasi-syntomic topology).
	Using $\Ainf$-comparison in \cite[Theorem 17.2]{BS19} together with pro-\'etale comparison in \Cref{main sm} (i), one can show that two prismatic cohomology over $\Bdr$ are compatible, namely cohomology of $\Prism_\bullet[1/p]^\wedge_I$ of a smooth formal scheme of $\scrX$ is isomorphic to $R\Gamma_\Prism(\scrX_\eta/\Bdr)$ (after replacing the ideal $I=(\phi(\xi))$ by $(\xi)$).
	Thus our prismatic cohomology coincides with that of \cite{BS21} for rigid spaces that have good reduction.
\end{remark}
\begin{remark}[Globalization]\label{rmk globalize}
\Cref{main sm} % (i, iii, iv), namely the pro-\'etale comparison, Hodge-Tate filtration, and the calculation of Galois invariant can be globalized 
naturally extends to proper rigid spaces that has l.c.i singularities (with $L\eta_\xi R\Gamma_\Prism(X/\Bdr)$ in \Cref{main sm} (ii) replaced by a sheaf version).

When $X$ is proper smooth, by Hodge-Tate filtration and the finiteness of coherent cohomology, each $\mathrm{H}^i_\Prism(X/\Bdr)$ is finite generated over $\Bdr$ and lives in cohomological degree $0\leq i\leq 2\dim(X)$.
In fact, using the primitive comparison theorem (\cite{Sch12}) one can show that each $\mathrm{H}^i_\Prism(X/\Bdr)$ is finite free over $\Bdr$.

However, if $X$ is not smooth, the Hodge-Tate filtration in \Cref{main sm} (iii) has infinite graded pieces. 
In the special case when $X=Y_K$ for a proper rigid space $Y$ over a discretely valued subfield $k$, the Hodge-Tate filtration splits into an infinite direct sum  $\moplus_{i\in \mathbb{N}}R\Gamma(X, L\wedge^i\mathbb{L}^\an_{X/K}(-i)[-i])$, which is in general infinite dimensional for proper rigid spaces that have l.c.i singularities.
Thus one loses the finiteness if we go beyond the smooth proper assumption.

On the other hand, combine \Cref{main sm} (i) and (ii), one obtain the following composition
\[
R\Gamma_{\inf}(X/\Bdr) \rra R\Gamma_\Prism(X/\Bdr) \rra R\Gamma_\pe(X,\BBdrp).
\]
It can be shown that the composition is an isomorphism after inverting $\xi$ (\cite[Theoren 7.3.2]{Guo21}), and thus  $R\Gamma_{\inf}(X/\Bdr)[\frac{1}{\xi}]$ is a direct summand of $R\Gamma_\Prism(X/\Bdr)[\frac{1}{\xi}]$, and contains all Galois invariants when $X=Y_K$ is defined over $k$ (notice the latter is finite when $Y$ is proper).
One can then ask if the complement of $R\Gamma_{\inf}(X/\Bdr)$ in $R\Gamma_\Prism(X/\Bdr)$  has bounded torsion, which would imply the following:
\begin{question}\label{ques iso invert xi}
	Are the maps in \Cref{main sm} (i) and (ii) inducing isomorphisms after inverting $\xi$?
\end{question}
In the case when $X=Y_K$ is defined over the discretely valued subfield $k$, it can be shown that the boundedness of torsion in $R\Gamma_\Prism(X/\Bdr)$ follows from the following question on analytic derived de Rham cohomology of $Y/k$.
\begin{question}
	Let $Y$ be a rigid space over $k$ that has l.c.i singularities.
	Is there a positive integer $N$, such that the following map of cohomology is always zero for any $n\in \mathbb{N}$ and $i\in \mathbb{Z}$
	\[
	\mathrm{H}^{i}(Y, \Fil^{n+N} \dR_{Y/k}) \rra \mathrm{H}^i(Y, \Fil^n \dR_{Y/k}).
	\]
	Here $\Fil^\bullet$ is the natural filtration on the analytic derived de Rham complex.
	
\end{question}
\begin{remark}
	The question has a more classical variant in the complex algebraic geometry.
	Namely for a complex algebraic variety $Y$ that has l.c.i singularities, can we find a positive integer $N$ so that the following map is always zero:
	\[
	\mathrm{H}^{i}(Y, \Fil^{n+N} \wh{\mathrm{dR}}_{Y/\mathbb{C}}) \rra \mathrm{H}^i(Y, \Fil^n \wh{\mathrm{dR}}_{Y/\mathbb{C}}),
	\]
	where $\Fil^\bullet$ is the natural filtration on the Hodge-completed derived de Rham complex $\wh{\mathrm{dR}}_{Y/\mathbb{C}}$.
	
	If $Y$ is of dimension zero, one can show the above by hand, using explicit calculation for a closed immersion into $\mathbb{A}^1$.
	On the other hand, it is shown in \cite[Corollary 5.4]{Bha12} that the map of cohomology $\Fil^n\wh{\mathrm{dR}}_{Y/\mathbb{C}} \ra \wh{\mathrm{dR}}_{Y/\mathbb{C}}$ is the zero map for $n>>0$.
\end{remark}
%However, as prismatic cohomology is derived $\xi$-complete and denominator of  $\Fil^n\dR_{Y/k}$ goes to infinity, it is not enough

%However, it is known that the decalage functor does not commute with taking global sections.
%We give an example in \Cref{eg decalage proper} illustrating the discrepancy between infinitesimal cohomology and decalaged prismatic cohomology for projective curves.
\end{remark}

\begin{remark}(Condensed mathematics)
	Due to the fact that there is no derived category of topological groups, we instead use the notion of condensed mathematics of Clausen-Scholze to study Galois invariant of prismatic cohomology  in \Cref{main sm} (iv).
	As byproducts, we also compute Galois invariant of infinitesmal cohomology, and give a condensed base extension formula of infinitesimal cohomology, as in \Cref{Galois of inf} and \Cref{Galois of inf} (see also the filtered versions as in \Cref{Galois filtered} and \Cref{Galois tensor product}).
%	In fact, one may even define the category of $\Bdr$-prisms completely within the condensed formalism using solid $\Bdr$-algebras.
\end{remark}

We now briefly discuss the proof in next two subsections.
\subsection{Simpson's functor}

Recall from \cite{Simp} that Simpson introduced an equivalence between the category of $\mathbb{Z}^\op$-filtered vector spaces over a field and the category of modules over $\mathbb{A}^1/\mathbb{G}_m$, using a Rees algebra construction.
The functor is defined as follows
\[
(V,\Fil^\bullet) \lmt \underset{i\in\mathbb{Z}}{\colim} \frac{\Fil^iV}{t^i},
\]
where $t$ is the coordinate of $\mathbb{A}^1$, and  admits a natural action by $\mathbb{G}_m$.
The functor can be extended to an equivalence in the derived level (\cite{Mou19}), and its pullback along $\{t=0\}\ra \mathbb{A}^1$ is the direct sum of graded pieces for $(V,\Fil^\bullet)$.

Analogous to the above construction, we introduce a variant of Simpson's functor, sending a $\mathbb{N}^\op$-filtered $\Bdr$-complex onto the $\Bdr$-complex as below
\[
\Psi:(M, \Fil^\bullet) \lmt (\underset{i\in\mathbb{N}}{\colim} \frac{\Fil^iM}{\xi^i})^\wedge,
\]
where $(-)^\wedge$ is the derived $\xi$-completion.
One of our main observations is that prismatic cohomology over $\Bdr$ can be computed by applying Simpson's functor at infinitesimal cohomology of $X/\Bdr$ with a lifted infinitesimal filtration.
\begin{theorem}\label{main tech}
	Let $X$ be an affinoid rigid space over $K$ that has l.c.i singularities.
	Then we have the following formula
	\[
	\Psi(R\Gamma_{\inf}(X/\Bdr), \wt\Fil^\bullet) \cong R\Gamma_\Prism(X/\Bdr).
	\]
\end{theorem}
Here $\wt\Fil^\bullet$ is a lift of the infinitesimal filtration of $X/K$ along $\Bdr\ra K$, and is not canonical in general (see discussion around \Cref{lifted inf}).
We use the following two examples to illustrate what the formula looks like.
\begin{example}
	Assume $X=\Spa(R)$ is smooth affinoid and admits a compatible system of lifts $\Spa(\wt R_e)$ as in \Cref{rmk integral proetale comparison}.
	Then $R\Gamma_{\inf}(X/\Bdr)$ is isomorphic to $\wt \Omega^\bullet$, and the lifted infinitesimal filtration is defined as the Hodge filtration 
	\[
	\wt\Fil^i \wt \Omega^\bullet = (\wt\Omega^i \rra \cdots \rra \wt\Omega^n)[-i].
	\]
	In particular, applying Simpson's functor at this filtered object, we get 
	\[
	\Psi(R\Gamma_{\inf}(X/\Bdr), \wt\Fil^\bullet) = \wt \Omega^\bullet \cup \frac{\wt\Omega^{\geq 1}}{\xi} \cup \cdots \cup \frac{\wt\Omega^d}{\xi^d},
	\]
	which is exactly the $\xi$-divided de Rham complex as in \Cref{rmk integral proetale comparison}.
	Here the union is taken inside of $\wt \Omega^\bullet[\frac{1}{\xi}]$.
\end{example}
\begin{example}
	Another special case is when $X=Y_K$ for an affinoid rigid space $Y$ over the discretely valued subfield $k$.
	Under the assumption, we can take the completed base extension of the infinitesimal filtration of $Y/k$ along a fixed embedding $k\ra \Bdr$ 
	\[
	\wt\Fil^\bullet := \Fil^\bullet_{\inf}(Y/k) \wh\otimes_k \Bdr,
	\]
	which in particular admits a $G_k$-action.
	We warn the reader that this is \emph{not} the same as the infinitesimal filtration for $R\Gamma_{\inf}(Y_K/\Bdr)$.
\end{example}

\subsection{Simplicial resolution}
Another ingredient of the proof is the simplicial resolution and the left Kan extension for affinoid algebras.
Not as in the algebraic setting, where one can use polynomial rings of infinite variables to resolve any algebra; in the rigid geometry a typical algebraic object is topologically generated by finite variables.
It is then natural to ask if we can resolve an affinoid algebra by convergent power series rings of finite variables, thus within the rigid geometry.
We give a positive answer to this question, and show in \Cref{subset simplicial res} that an affinoid algebra admits a simplicial resolution of Tate algebras with finite variables.
Moreover, we use this to show that the analytic cotangent complex $\mathbb{L}^\an$ and the analytic derived de Rham complex $\dR$ for an affinoid algebra can be computed using the aforementioned resolution (\Cref{left kan of ddR}).

With the help of the simplicial resolution, we can apply the left Kan extension at prismatic cohomology of smooth affinoid algebras.
This allows us to define a notion of \emph{derived prismatic cohomology} $L\Prism_{R/\Bdr}$ for an arbitrary affinoid algebra $R$ over $K$ (see \Cref{left Kan HT}), analogous to \cite[Section 7]{BS19}.
As expected, derived prismatic cohomology and usual prismatic cohomology coincide for l.c.i singularities.
\begin{theorem}\label{main derived}
	Let $X$ be a rigid space over $K$ that has l.c.i singularities.
	There is a natural isomorphism of $\Bdr$-$\mathbb{E}_\infty$-algebras
	\[
	L\Prism_{R/\Bdr} \rra R\Gamma_\Prism(X/\Bdr).
	\]
\end{theorem}
This in particular allows us to define the Hodge-Tate filtration on reduced prismatic cohomology $R\Gamma_\Prism(X/\Bdr)\otimes^L_\Bdr K$ for rigid spaces with l.c.i singularities, which is \Cref{main sm} (iii).

\begin{remark}
	For our application, we consider a slightly more general situation.
	Namely we show the existence of the finite type resolutions for a pair $(B,I)$, consisting of a topologically finite type algebra $B$ over $\Ainf/\xi^e$ (resp. $\Bdr/\xi^e$, or $\Bdr$) together with an ideal $I$.
	This roughly means that any closed immersion of topologically finite type affine formal schemes (resp. affinoid rigid spaces) can be simplicially resolved by regular closed immersions.
	In the special case when $I$ is the zero ideal, we get the aforementioned simplicial resolution for a topological algebra itself.
\end{remark}
\begin{remark}
	As both derived prismatic cohomology and the usual prismatic cohomology are well-defined for general rigid spaces, it is a natural question if \Cref{main derived} holds without the assumption on $X$, or if the cone of the comparing map always has bounded $\xi$-torsion.
	We hope to investigate the question in the future.
\end{remark}
\begin{comment}
	In \cite[Section 7]{BS19}, they introduce the notion of derived prismatic cohomology for arbitrary $p$-complete rings by Left Kan extension.
	On the other hand, for $p$-complete topologically finite type rings that have l.c.i singularities integrally, one can compute the usual site theoretic prismatic cohomology as in \cite[Section 4]{BS19}.
	One can then check that those two cohomology coincide.
\end{comment}

\subsection{Other comments}
\begin{remark}
	We want to mention that one can define an analogous notion of prismatic cohomology for complex algebraic varieties, replacing $\Bdr$ by $\mathbb{C}[[t]]$.
	In the smooth case, this was studied independently in an unpublished article by Nils Waßmuth (\cite{Was19}), where he proved the de Rham comparison and the Hodge-Tate comparison for affine smooth varieties over $\mathbb{C}$.
	In general, similar to what we do in this article, one can use Simpson's equivalence to study prismatic cohomology for complex varieties that has l.c.i singularities, using derived de Rham cohomology (\cite{Ill71}, \cite{Bha12}) and Hartshorne's de Rham cohomology over $\mathbb{C}$ (\cite{Har75}).
\end{remark}
\begin{remark}(Crystal)
	As in \cite{BS21}, our prismatic site $X/\Bdr_\Prism$ admits a natural notion of \emph{crystals}, which we study in \Cref{sec cry}.
	Analogous to the crystal-connection translation in the crystalline theory, a prismatic crystal corresponds to a quasi-nilpotent integrable connection that have log poles at $\{\xi=0\}$ (Definition \ref{pm, log conn}, \Cref{pm cry via envelope}).
	In fact, important resources of prismatic crystals are those coming from the infinitesimal site $X/\Bdr_{\inf}$, whose associated log-connections are integral.
	Moreover, our \Cref{main tech} can be extended to prismatic crystals that come from the infinitesimal site.
\end{remark}
\begin{remark}(Non-noetherian enlargement)
	As mentioned in the beginning, our prismatic site (\Cref{intro def}) consists only of those $\Bdr$-algebras that are noetherian.
	However, one can still ask if we could allow other interesting but non-noetherian rings.
	In fact, to compare prismatic cohomology with pro-\'etale cohomology as in \Cref{main sm} (i), we consider an enlarged version of the prismatic site (with the same indiscrete topology) in \Cref{subsec proetale and enlarged}, adding algebras coming from perfectoid rings.
	There we prove that cohomology of the enlarged site coincides with the smaller one in \Cref{intro def}.
	Similarly, adding $\BBdrp(R)$ for quasi-regular semiperfectoid rings $R$ over $\mathcal{O}_K$  would not change the cohomology.
\end{remark}

\paragraph{\textbf{Leitfaden of the paper}}
The article is structured as follows.
We start with  \Cref{sec site} about the basics on the prismatic site over $\Bdr$. 
Then we study the notion of prismatic crystals in \Cref{sec cry}, where we relate it to log connections, and compute their cohomology using \v{C}ech-Alexander complex.
In \Cref{sec inf to prism}, we give a local calculation of cohomology for an affinoid rigid space that admit a regular immersion. 
Here we first introduce the $\Bdr$-variant of Simpson's functor in \Cref{subsec Simpson}. 
Then in \Cref{subs prism to inf} we construct a natural functor from infinitesimal crystals to prismatic crystals.
In \Cref{subsec loc cal on lci} we compute prismatic cohomology using regular immersions and Simpson's functor, and in particular prove \Cref{main tech}  (in \Cref{prism coh}).
Next in \Cref{sec HT}, we prove the Hodge-Tate filtration theorem (i.e. \Cref{main sm} (iii)) in \Cref{HT comp} and \Cref{HT fil}.
We start by proving the version for smooth rigid spaces in \Cref{HT comp}.
Then we prove the existence of simplicial resolution for affinoid algebras paired with ideals in \Cref{subset simplicial res}.
With the help of the simplicial resolution, we are able to define the notion of derived prismatic cohomology and its Hodge-Tate filtration (\Cref{left Kan HT}).
Finally we show derived prismatic cohomology and usual prismatic cohomology coincide for l.c.i singularities (\Cref{left kan of coh}), thus finishes the proof of \Cref{main tech} (iii).
In \Cref{sec inf compa}, we prove the infinitesimal comparison theorem (i.e. \Cref{main sm} (ii)) in \Cref{decalage comp}.
In \Cref{sec proetale compa}, we prove the pro-\'etale comparison theorem (i.e. \Cref{main sm} (i)) in \Cref{proet comp}.
As a byproduct we also give a K\"unneth formula for reduced prismatic cohomology in \Cref{Kunneth reduced}.
Finally we prove the formula on Galois invariant (i.e. \Cref{main sm} (iv)) in \Cref{sec Galois inv}.
We start with a brief discussion in \Cref{subsec condensed} recalling basic results we need from condensed mathematics. 
Then we calculate Galois invariant of infinitesimal cohomology and prismatic cohomology in \Cref{Galois of inf}, \Cref{Galois filtered}, and \Cref{Galois of prism}.\\

\paragraph{\textbf{Conventions and notations}}
We fix a complete perfectoid field extension $K/\mathbb{Q}_p$ such that $K$ contains all $p^\infty$ roots of unity.
We also fix a generator $\xi=\frac{[\epsilon]-1}{\varphi^{-1}[\epsilon] -1}$ for the canonical surjection $\theta:\Bdr\ra K$, following \cite{Fon94}. 
For a non-negative integer $e$, we denote $\Ainfe$ to be the quotient ring $\Ainf/\xi^e$, and $\Bdre$ to be $\Bdr/\xi^e$ with the $p$-adic rational topology.
We assume the basics of quasi-compact quasi-separated  topologically of finite type adic spaces over a $p$-adic field or $\Bdre$ (which we call $K$-rigid space or $\Bdre$-rigid spaces separately), and refer the reader to \cite{Hu96} for a detailed foundational study.
We also use the notation $\Bdr\langle T_i\rangle=\Bdr\langle T_1,\ldots, T_l\rangle$ throughout the article to denote the inverse limit of affinoid rings $\Bdre\langle T_1,\ldots, T_l\rangle$ with respect $e$.
For a $\Bdr$-module $M$, we use $M_\tf$ to denote its $\xi$-torsionfree quotient $M/(\xi^\infty\text{-torsion})$.

We use mildly the language of infinity category and $\mathbb{E}_\infty$-algebras throughout the article, and refer the reader to \cite{Lu09} and \cite{Lu17} for the foundations.\\

\subsection{Acknowledgements}
The project was initiated during a conversation with Arthur-C\'esar Le Bras.
I thank Le Bras heartily for encouraging me to work on this topic at the very beginning.
It is obvious that the article is heavily influenced by the foundational work of Bhatt-Scholze \cite{BS19}.
I am grateful for various helpful discussions with Bhargav Bhatt on construction of prismatic sites.
Many thanks go to Sebastian Bartling and Shizhang Li for communications in the early stage of the work, and to Peter Scholze for explaining condensed group cohomology.
I also thank Dmitry Kubrak and Emanuel Reinecke for asking questions in an informal seminar, which motivates me to consider the left Kan extension for rigid spaces; and thank Guido Bosco for useful comments on the draft.
At last, during the writing process I am funded by National Science Foundation FRG grant no. DMS-1952399, and by Max Planck Institute for Mathematics.

\section{Prismatic site over $\Bdr$}\label{sec site}
In this section, we define the prismatic site over $\Bdr$ for an affinoid rigid space over $K$.

\subsection{Algebraic preliminary}
We start with an algebraic preliminary on power-bounded functions of a finite type scheme over a rigid space.
The result here will be used later to show the weak finality of the prismatic envelope (\Cref{pm uni}).

\begin{proposition}\label{power-bounded on relative}
	Let $R$ be a topologically finite type algebra over $K$, and let $S$ be a finite type algebra over $R$.
	Assume there exists a $K$-linear homomorphism $\phi:K\langle T\rangle \ra S$.
	Then the image $\phi(T)$ is finite over $R$.
\end{proposition}
When $R$ is the $p$-adic field $K$, this is proved in \cite[Lemma 5.5]{HL20}.
\begin{proof}
	$~$
	\begin{itemize}
		\item[Step 1]
		We first reduce to the case when $S$ is normal.
		Let $S_\red$ be the reduced quotient of the noetherian ring $S$, and let $S_\mathrm{nor}$ be the normalization of $S_\red$.
		We then have the diagram
		\[
		\xymatrix{K\langle{T} \rangle  \ar[r]^\phi& S \ar[r]& S_\red \ar[r] & S_\mathrm{nor}\\
			&R \ar[u] &&.}
		\]
		Denote $t$, $t_1$, and $t_2$ to be the image of $T$ in $S$, $S_\red$ and $S_\mathrm{nor}$ separately.
		Assume there exists a monic polynomial $g(X)\in R[X]$ over $R$ such that $g(t_2)=0$.
		As the ring $S_\red$ is noetherian, the normalization $S_\red \ra S_\mathrm{nor}$ is injective.
		In particular, $g(t_2)=0$ implies $g(t_1)$ is zero, which then implies that $g(t)$ is nilpotent in $S$.
		Thus by taking $h(X):=g(X)^N\in R[X]$ for some large $N$, we can find a monic polynomial $h(X)$ over $R$ with $h(t)=0$.
		So it suffices to show that by passing to the normalization of $S$, the image of $T\in K\langle T\rangle$ is finite over $R$.
		
		\item[Step 2]
		Denote by $X$ the scheme $\Spec(S)$, and $X^\an$ the relative analytification of $S$ over the affinoid ring $R$, as a rigid space over $\Spa(R)$ (c.f \cite{Con06}).
		Here we note that $X^\an$ is also a normal rigid space, as the normality is checked around each $K$-valued point of the scheme $X$ (\cite[Theorem 2.2.1]{Berk93}).
		For each affinoid open subset $U=\Spa(A)$ of $X^\an$, the composition of rings $K\langle T\rangle \ra S\ra A$ is automatically continuous by \cite[6.1.3/1]{BGR}.
		In particular, this induces a map of rigid spaces $X^\an \ra \BB^1$, where $\BB^1=\Spa(K\langle T\rangle)$ is the unit disc over $K$.
		Here we note that the ring $S$ is contained in the global section $\calO_{X^\an}(X^\an)$ of the structure sheaf.
		Moreover, to give  a map from the rigid space $Y=X^\an$ to $\BB^1$ is equivalent to give a power-bounded global analytic function over $Y$, i.e. an element in $\calO_Y^+(Y)$.
		\footnote{To see this, it suffices to use the criterion of a power-bounded function over a rigid space: an element $f\in \calO_Y(Y)$ is power-bounded if and only if $|f|_y\leq 1$ for any rigid point $y\in Y$ (\cite[6.2.3/1]{BGR}).}
		So the image $t=\phi(T)$ is a power-bounded elements in $\calO_{X^\an}(X^\an)$.
		
		\item[Step 3]
		As the affine scheme $X=\Spec(S)$ is finite type over the noetherian scheme $\Spec(R)$, by Nagata's compactification theorem we could find a schematic open immersion $X\ra \ol X$ with dense image, where $\ol X$ is a proper $\Spec(R)$-scheme.
		In particular, by taking the relative analytification $\ol X^\an$ of $\ol X$, the induced morphism $X^\an \ra \ol X^\an$ is an open immersion of rigid spaces over $\Spa(R)$ whose complement is Zariski closed of codimension $\geq 1$ in $\ol X$.
		Moreover, notice that since $X^\an$ is normal, it is also an open dense subspace in the normalization of $\ol X^\an$, so we can thus assume $\ol X^\an$ is normal (\cite[Theorem 2.1.2]{Con99}).
		We apply the first Riemann extension theorem (\cite[$\S$3, Theorem]{Bar76}) for the open immersion $X^\an \ra \ol X^\an$ into the proper normal rigid space, to get the isomorphism 
		\[
		\calO^+_{X^\an}(X^\an)= \calO^+_{\ol X^\an}(\ol X^\an).
		\]
		But since $\ol X^\an$ is proper over $\Spa(R)$, the global section of the structure sheaf $\calO_{\ol X^\an}(\ol X^\an)$ is a finite $R$-algebra.
		In this way, the element $t=\phi(T)\in \calO^+_{X^\an}(X^\an)= \calO^+_{\ol X^\an}(\ol X^\an)\subset \calO_{\ol X^\an}(\ol X^\an)$ is also finite over $R$, and we are done.
		
	\end{itemize}
\end{proof}
The proof above also shows the following result.
\begin{corollary}
	Let $Y$ be the relative analytification of a finite type scheme over an affinoid rigid space $\Spa(R)$.
	Then any power-bounded function of $Y$ is finite over $R$.
\end{corollary}
The analogous result of \Cref{power-bounded on relative} holds for $\Bdrm$ as well.
\begin{corollary}\label{power-bounded on relative Bdr}
	Let $R$ be a topologically finite type algebra over $\Bdrm$, and let $S$ be a finite type algebra over $R$.
	Assume there exists a $K$-linear homomorphism $\phi:\Bdrm\langle T\rangle \ra S$.
	Then the image $\phi(T)$ is finite over $R$.
\end{corollary}
\begin{proof}
	By \Cref{power-bounded on relative}, the image $\phi(T)$ mod $\xi$ is finite over $R/\xi$.
	In particular, there exists a polynomial $g(X)\in \Bdrm[X]$, such that $g(\phi(T))=0$ mod $\xi$.
	Thus by taking a power of $g(X)$ we get the result.
\end{proof}

\subsection{Prismatic site}
In this subsection, we introduce the basics on the prismatic site over $\Bdr$.

We start by defining the prismatic site.
\begin{definition}\label{pm site}
	Let $X=\Spa(R)$ be an affinoid rigid space over $K$.
	\begin{enumerate}[(i)]
		\item An $\Bdr$-\emph{prism} (or in short a \emph{prism}) $(B,\xi B)$ over $X$ is defined by the following diagram of maps
		\[
		\xymatrix{B \ar[r] & B/\xi & R \ar[l]_f,}
		\]
		where  $B$ is a noetherian $\Bdr$-algebra satisfying the conditions below:
		\begin{itemize}
			\item it is $\xi$-torsion free and $\xi$-adic complete;
			\item each quotient $B/\xi^m$ is a finite type algebra over some $\Bdrm\langle T_i\rangle$.
		\end{itemize}
		The map $B\ra B/\xi$ is the natural surjection, and $f:R \ra B/\xi$ is a map of $K$-algebras. 
		\item A map of prisms $(B,\xi B)\ra (B',\xi B')$ over $R$ is defined as a commutative diagram between the data above.
		\item The \emph{prismatic site of $X=\Spa(R)$ over $\Bdr$}, denoted as $X/\Bdr_\Prism$, is defined on the opposite category of all $\Bdr$-prisms over $X$ and is equipped with the indiscrete topology. %analytic topology: namely a family of morphisms $\{(B,\xi B) \ra (B_i,f_i) \}$ is a covering if the induced map of $K$-rigid spaces $\{ \Spa(B_i/\xi) \ra \Spa(B/\xi)\}$ is an open covering. \footnote{As $B/\xi$ and $B_i/\xi$ are finite type algebras over some affinoid $K$-algebras, it makes sense to consider their associated rigid space structure, as the relative analytifications in the sense of \cite{Con06}. }
	\end{enumerate}
\end{definition}
\begin{remark}
	By definition, given a prism $(B,\xi B)$ over $R$, the only covering of the prism $(B,\xi B)$ is the identity map from itself.
\end{remark}
\begin{remark}
	Except for everything is in the rational level (namely over $\mathbb{Q}$), the most obvious simplification from \cite{BS19} is that a prism does not come with a delta structure defined integrally.
	Namely there is no Frobenius lift structure.
\end{remark}
\begin{remark}
	To save us from complicated notations, we will abbreviate notation of a prism $(B,\xi B)$ simply into $B$ from time to time, when the meaning is clear.
\end{remark}

Next we define the notion of the prismatic envelope.
\begin{definition}\label{pm envelope}
	Let $B\ra B'$ be a surjection of complete noetherian $\Bdr$-algebras, and let $I$ be the kernel ideal.
	The \emph{prismatic envelope} of $B'$ in $B$, is defined as the $\Bdr$-algebra
	\[
	B[\frac{I}{\xi}]^\wedge_\tf:= \left(B[x_1, \ldots, x_m]/(\xi x_i-f_i, 1\leq i\leq m)\right)^\wedge/(\xi^\infty\text{-}torsion).
	\]
	where $(f_i)$ is a finite set of generators for the ideal $I$, and $(-)^\wedge$ is the classical $\xi$-adic completion.
\end{definition}
To make sense of the notation that is independent of the choice of generators, we have the following easy property.
\begin{lemma}
	Let $I$ be an ideal of a noetherian $\Bdr$-algebra $B$, and let $(f_i)\subset (f_i,g)$ be two finite sets of generators of $I$.
	Then the natural map of rings below is an isomorphism
	\[
	\left(B[x_i]/(\xi x_i-f_i)\right)^\wedge/(\xi^\infty\text{-}torsion) \rra 
	\left(B[x_i,y]/(\xi x_i-f_i; \xi y-g)\right)^\wedge/(\xi^\infty\text{-}torsion).
	\]
	In particular, the ring $B[\frac{I}{\xi}]^\wedge_\tf/(\xi^\infty\text{-}torsion)$ in \Cref{pm envelope} is independent of the choice of generators of $I$.
\end{lemma}
\begin{proof}
	As a finite type algebra over $B$ is noetherian, its formal completion along $\xi=0$ identifies the $\xi^\infty$-elements before and after the completion (\cite[Tag 05E9]{Sta}).
	In particular, we can switch the order of the $\xi$-completion and the $\xi$-torsionfree quotient, and it suffices to show the isomorphism of the map between incomplete algebras
	\[
	\left( B[x_i]/(\xi x_i-f_i)\right) /(\xi^\infty\text{-}torsion) \rra 
	\left( B[x_i,y]/(\xi x_i-f_i; \xi y-g)\right) /(\xi^\infty\text{-}torsion).
	\]
	By assumption, since $(f_i)$ generates the ideal, we can write $g=\sum a_jf_j$.
	So we get the formula in the second ring above
	\[
	\xi \cdot (y -\sum a_j x_j)=0.
	\]
	Thus by the $\xi$-torsion freeness of the target ring, the map above is surjective.
	Moreover, we can define a section of the above map by mapping $y$ onto $\sum a_jx_j$, which is possible by the $\xi$-torsion freeness of the quotient.
	Hence we get the isomorphism.
\end{proof}
\begin{remark}\label{pm envelope alt}
	In the above setting, we note that as a $B$-algebra, the envelope is automatically complete $I$-adically, by the $\xi$-adic completeness and the $\xi$-divisibility of the image of $I$ in the envelope.
	In particular, the map $B\ra B[\frac{I}{\xi}]^\wedge_\tf$ factors through the $I$-adic completion of $B$.
	As a consequence, mod any power of $\xi$, there is a natural isomorphism of $B$-algebras 
	\[
	B[\frac{I}{\xi}]^\wedge_\tf\cong D[\frac{I}{\xi}]^\wedge_\tf,
	\]
	where $D$ is the formal completion of $B$ along the ideal $I$.
\end{remark}

\begin{remark}\label{pm envelope rep}
	Let $B$ and $I$ be as in Definition \ref{pm envelope}, and let $R=B/(I,\xi)$.
	As the image of $(I,\xi)$ under the composition
	\[
	B \ra B[\frac{I}{\xi}]^\wedge_\tf \ra B[\frac{I}{\xi}]^\wedge_\tf/\xi
	\]
	is zero, it naturally induces a continuous map $f:R=B/(I,\xi) \ra  B[\frac{I}{\xi}]^\wedge_\tf/ \xi$.
	This in particular produces a prism $B[\frac{I}{\xi}]^\wedge_\tf $ over $R$.
	Furthermore, for a prism $(B',\xi B')$ in $X/\Bdr_\Prism$, to give a continuous $\Bdr$-map from the complete $\Bdr$-algebra $B$ to $B'$ such that the composition with $B'\ra B'/\xi$ factors through the surjection $B\ra R$, it is equivalent to give a map of prisms from $(B[\frac{I}{\xi}]^\wedge_\tf ),  B[\frac{I}{\xi}]^\wedge_\tf / \xi)$ to $(B', \xi B')$, by the $\xi$-torsion free assumption on $B'$.
\end{remark}
\begin{definition}
	We define the \emph{prism $D_\Prism$ associated to the envelope} for the surjection $B\ra I$ to be the one associated to the ring $B[\frac{I}{\xi}]^\wedge_\tf $.
\end{definition}
To simplify the notation, we also use the same symbol $D_\Prism$ to denote the ring $B[\frac{I}{\xi}]^\wedge_\tf $ itself.
\begin{example}\label{pm envelope eg}
	Let $B$ be the limit ring $\Bdr\langle T,S\rangle = \varprojlim_n\Bdr/\xi^n\langle T,S\rangle$, % be the inverse limit algebra of convergent power series rings over $\Bdrn$,
	and assume $R$ is the quotient ring $K\langle T\rangle$ of $B$.
	Then the envelope of the surjection $B\ra \Bdr\langle T\rangle $ is given by the following formula
	\[
	\Bdr\langle T\rangle[\frac{S}{\xi}]^\wedge=\varprojlim_n \left( \Bdrn\langle T\rangle[\frac{S}{\xi}] \right).
	\]
\end{example}
Thanks to \Cref{power-bounded on relative}, the envelope of a $\Bdr$-lift of $R$ into a formally smooth topological algebra over $\Bdr$ covers the final object in the prismatic topos.
\begin{proposition}\label{pm uni}
	Let $P$ be the ring $\Bdr\langle T_i\rangle$ and $I$ be an ideal of $P$, 
	and $f_0:P \ra R:=P/(I,\xi)$ be the quotient.
	Then the prism $D_\Prism$ associated to the envelope for $P\ra R$ covers the final object of the topos $\Sh(X/\Bdr_\Prism)$.
\end{proposition}
\begin{proof}
	Let $(B,\xi B)$ be a prism over $R$.
	We need to show that it admits a map from the prism associated to the envelope $P[\frac{I}{\xi}]^\wedge_\tf $.
	We first notice the following commutative diagram
	\[
	\xymatrix{P=\Bdr\langle T_i\rangle \ar[d] \ar[rd]^{f_0} &\\
		B/\xi & R ~~.\ar[l]}
	\]
	Then we show the diagram lifts to a homomorphism $P\ra B$ over $\Bdr$.
	By the assumption of a prism in Definition \ref{pm site} (ii), the ring $B/\xi^m$ admits a surjection from $\Bdrm\langle a_{j_m}\rangle [b_{l_m}]$ for finite sets of $j_m$ and $l_m$, compatible with $m\in \NN$.
	Moreover, by enlarging the set of $j_m$ and $l_m$ inductively for each $m$, we may assume a compatible system of surjections as below
	\[
	\xymatrix{ \mathrm{B_{dR,m+1}^+} \langle a_{j_{m+1}}\rangle [b_{l_{m+1}}] \ar[d] \ar@{>>}[r] & B/\xi^{m+1} \ar[d]\\
		\Bdrm \langle a_{j_m}\rangle [b_{l_m}]  \ar@{>>}[r] & B/\xi^m.}
	\]
	In particular, when $m=1$, the quotient ring $B/\xi$ is a finite type algebra over an affinoid ring $K\langle a_{j_1}\rangle$, and it admits a map from $P/\xi=K\langle T_i\rangle$ by the diagram above.
	Note that by Proposition \ref{power-bounded on relative} the image $t_i$ of each $T_i$ in $B/\xi$ is integral over $K\langle a_{j_1}\rangle$.
	Moreover, by \Cref{power-bounded on relative Bdr} any lift $\wt t_i$ of $t_i$ in $B/\xi^m$ is integral over $\Bdrm\langle a_{j_m}\rangle$.
%	\footnote{We may pick a $K\langle a_j\rangle$-monic polynomial $P_0(X)$ such that $P_0(t_i)=0$, and let $P(X)\in \Bdrm\langle a_j\rangle[X]$ be any lift of $P_0(X)$.
%		Then $P(\wt t_i)$ belongs to the ideal $\xi B/\xi^m$, where the latter is nilpotent in $B/\xi^m$.
%		So by taking a power of $P(X)$ we see $\wt t_i$ is integral over $\Bdrm\langle a_j\rangle$.}
	So by choosing compatible lifts $\wt t_i$ of $t_i$ in $B/\xi^m$ for each $m\in \NN$, we could extend the map $\Bdrm\langle a_{j_m}\rangle \ra B/\xi$ to $\Bdrm\langle a_{j_m} ; c_i\rangle \ra B/\xi^m$, where $c_i$ maps onto lifts $\wt t_i$ of $t_i\in B/\xi$.
	As a consequence, 	we can enlarge variables $\{a_{j_m}\}$ into $\{a_{j_m}, c_i\}$ to get a surjection $\Bdrm\langle a_{j_m, c_i}\rangle[b_{l_m}]\ra B/\xi^m\ra B/\xi$.
	In particular, we may assume the map $P=\Bdr\langle T_i\rangle \ra B/\xi$ above factors through the surjections  as below:
	\[
	\xymatrix{
		&& P \ar[d] \ar[rd] \ar[lld]_{T_i\mapsto c_i} &\\
		\Bdrm\langle a_{j_m}, c_i \rangle[b_{l_m}] \ar@{>>}[r] &B/\xi^m \ar@{>>}[r] & B/\xi & R ~~, \ar[l]
	}\]
which is compatible with $m\in \mathbb{N}$.
	So by taking an inverse limit with respect to $m$, we get the commutative diagram
	\[
	\xymatrix{
		&P \ar[ld] \ar[rd] \ar[d]&\\
		B \ar[r]& B/\xi& R~~. \ar[l]
	}\]
	
	At last, since the image of the ideal $I\subset P$ in $B$ is killed by $\xi$ and the ring $B$ is $\xi$-torsion free, the homomorphism $P\ra B$ induces a natural commutative diagram as below
	\[
	\xymatrix{
		P[\frac{I}{\xi}]^\wedge_\tf \ar[d] \ar[r]&P[\frac{I}{\xi}]^\wedge_\tf/\xi=R[\frac{I}{\xi}] \ar[d] & ~~R ~~\ar[l]\ar@{=}[d]\\
		B \ar[r] &B/\xi & ~~R~~,\ar[l]}\]
	where $P[\frac{I}{\xi}]^\wedge_\tf$ is the $\xi$-adic completion of the $\Bdr$-algebra $P[\frac{I}{\xi}]$.
	Hence the ring $P[\frac{I}{\xi}]^\wedge_\tf$ together with the $R$-algebra structure of its quotient forms a prism $D_\Prism$, and we get a map of prisms $(D_\Prism, \xi D_\Prism) \ra (B,\xi B)$.
	
\end{proof}

\section{Prismatic crystal and log connection}\label{sec cry}
There are two natural \emph{structure sheaves} over $X/\Bdr_\Prism$, defined as follows:
\begin{align*}
	\calO_\Prism&: (B,\xi B) \lmt B;\\
	\ol\calO_\Prism&: (B,\xi B) \lmt B/\xi.
\end{align*}
This allows us to define the notion of the crystal over the prismatic site, analogous to the infinitesimal story.
\begin{definition}\label{pm crys}
	A (coherent) \emph{prismatic crystal} over $X$ is defined as a sheaf $\calF$ of $\mathcal{O}_\Prism$-modules over $X/\Bdr_\Prism$, satisfying 
	\begin{enumerate}[(a)]
		\item each $\calF(B,\xi B)$ is a finite module over $B$, for every prism $(B,\xi B)\in X/\Bdr_\Prism$;
		\item for every map of prisms $(B,\xi B)\ra (B',\xi B')$, the natural base change map below is an isomorphism of $B'$-modules
		\[
		\calF(B,\xi B)\otimes_B B' \rra \calF(B',\xi B').
		\]
	\end{enumerate} 
\end{definition}
Notice that given any prismatic crystal $\mathcal{F}$, there is a way to associate a \emph{reduction} $\ol{\mathcal{F}}$ to it, by assigning
\[
(B,\xi B) \lmt \mathcal{F}(B,\xi B) \otimes_B B/\xi.
\]
This in particular gives a prismatic crystal over the reduced structure sheaf $\ol{\mathcal{O}}_\Prism=\mathcal{O}_\Prism\otimes_\Bdr K$.
\begin{definition}
A prismatic crystal $\mathcal{F}$ is called \emph{flat} if each $B$-module $\mathcal{F}(B,\xi B)$ is locally free over $B$.
It is called \emph{reduced} if each $B$-module $\mathcal{F}(B,\xi B)$ is locally free over $B/\xi$.
\end{definition}

Analogous to the infinitesimal theory as in \cite[Theorem 3.3.1]{Guo21}, a crystal can be described as an \emph{integrable log  connection} over the envelope.

To start, let $B=\varprojlim_m B_m$ be an inverse limit of flat noetherian $\Bdrm$-algebras $B_m$, such that the transition maps induce isomorphisms $B_m\otimes_{\Bdrm} \mathrm{B_{dR,m-1}^+} \ra B_{m-1}$, and $B_1$ is isomorphic to a formal completion of some Tate algebra $K\langle T_1,\ldots, T_l\rangle$.
Let $I$ be an ideal of $B$.
Denote $D_\Prism$ to be $\xi$-torsion free quotient of the classical $\xi$-completion of the noetherian ring $B[\frac{I}{\xi}]$.

Fix a map $P:=\Bdr\langle T_i\rangle \ra B$ lifting  $K\langle T_i\rangle \ra B_1$ as above.
We first notice that there is a natural diagram of $\Bdr$-linear continuous differentials
\[
\xymatrix{
	P   \ar[r]^{d~~~~~~~~~} \ar[d]&  \Omega_{P/\Bdr}^1 \ar[d]\\
	D_\Prism \ar[r]^{d~~~~~~~~~~~~~~~~~~~~~~~~~~~~~~~~~~~} & D_\Prism \otimes_{P} \frac{\Omega_{P/\Bdr}^1}{\xi}.}
\]
Here $\Omega_{P/\Bdr}^1$ is the inverse limit of the $p$-adic continuous K\"ahler differential of $P_m/\Bdrm$, and horizontal maps are continuous with respect to $\xi$-adic topology, such that their mod $\xi$-reduction are continuous differentials with respect to $K \langle T_i\rangle$ and $B/\xi$.
On the other hand, the vertical maps are natural maps induced from the inclusions
\[
\Omega_{P/\Bdr}^1\subset \frac{\Omega_{P/\Bdr}^1}{\xi} \subset \Omega_{P/\Bdr}^1\otimes_\Bdr \mathrm{B_{dR}}.
\]
In particular, for an element $\frac{f_1\cdots f_i}{\xi^i}$ in $D_\Prism=P[\frac{I}{\xi}]^\wedge_\tf $ with $f_1,\ldots, f_i\in I$, we have
\[
d(\frac{f_1\cdots f_i}{\xi^i})= \frac{\sum_{j=1}^i f_1\cdots \hat{f_j} \cdots f_i\cdot df_j}{\xi^i}=\sum_{j=1}^i\frac{ f_1\cdots \hat{f_j} \cdots f_i}{\xi^{i-1}}\cdot \frac{df_j}{\xi} \in D_\Prism\otimes_P \frac{\Omega_{P/\Bdr}^1}{\xi}.
\]

Now we can consider a generalization to the coefficient theory.
\begin{definition}\label{pm, log conn}
	For a finite generated module $M$ over $D_\Prism$, a \emph{log connection} of $M$ is defined as an $\Bdr$-linear morphism
	\[
	\nabla: M \rra M\otimes_{D_\Prism} \frac{\Omega^1_{D_\Prism}}{\xi}:=M\otimes_P \frac{\Omega_{P/\Bdr}^1}{\xi},
	\]
	such that for $x\in M$ and $f\in D_\Prism$ we have
	\[
	\nabla(fx)=f\nabla(x)+x\otimes df,
	\]
	where $df$ is the continuous relative differential of $f\in D_\Prism$ over $\Bdr$ constructed above.
\end{definition}
The connection is called \emph{integrable} if the induced composition below is zero
\[
\nabla^1\circ \nabla: M \rra M\otimes_{D_\Prism} \frac{\Omega^2_{D_\Prism}}{\xi^2}:=M\otimes_P \frac{\Omega_{P/\Bdr}^2}{\xi^2}.
\]
Here the map $\nabla^1: M\otimes_{D_\Prism} \frac{\Omega^1_{D_\Prism}}{\xi} \rra M\otimes_{D_\Prism} \frac{\Omega^2_{D_\Prism}}{\xi^2}$ is given by the formula
\[
x \otimes \frac{\omega}{\xi} \lmt \nabla(x)\wedge \frac{\omega}{\xi} +x\otimes \frac{d\omega}{\xi}.
\]

\begin{remark}
	When $B$ is the $J$-adic completion of $P=\Bdr\langle T_i\rangle$ for some ideal $J$ in $P$, and $I$ is the zero ideal, the above definition extends the notion of connections in the infinitesimal theory defined in \cite[Definition 3.1.8]{Guo21}, allowing a log pole at $\{\xi=0\}$.
\end{remark}
\begin{remark}\label{pm cry integral def}
	In the special case when $\nabla:M \ra M\otimes_{D_\Prism} \frac{\Omega_{P/\Bdr}^1}{\xi}$ factors through $ M\otimes_{D_\Prism} \Omega_{P/\Bdr}^1$, we call the connection \emph{integral}.
	We will see soon that log connections coming from the infinitesimal site are all integral.
\end{remark}

Different from the infinitesimal theory, not all of the integrable log connections come from crystal.
This is more similar to the schematic theory in mixed characteristic, where \emph{quasi-nilpotence} condition is needed in order to make the transition power series associated to a connection to be convergent (c.f. \cite{BO78}).
We now define the $\Bdr$-analogue of the quasi-nilpotence as below.
\begin{definition}\label{q-nil}
	Let $B, ~I, ~P=\Bdr\langle T_1,\ldots ,T_l\rangle$ be as above, and $(M,\nabla)$ be a log connection over $D_\Prism=B[\frac{I}{\xi}]^\wedge_\tf$.
	We say $(M,\nabla)$ is \emph{quasi-nilpotent} if for any element $x\in M$, the following power series are convergent in the module $M\otimes_{D_\Prism} B[\frac{I}{\xi},\wt \delta_1,\ldots, \wt \delta_l]^\wedge$
	\[
	\sum_{E=(e_i)} \prod_{i=1}^l   \wt\nabla_i^{e_i}(x)\otimes \frac{\wt \delta_i^{e_i}}{e_i!}.
	\]
	
\end{definition}
Here $\wt\nabla_i:M\ra M$ is defined as the composition 
\[
\xymatrix{ M\ar[r]^{\nabla~~~~~~~~} & M\otimes_P \frac{\Omega_{P/\Bdr}^1}{\xi} \ar[rr]^{~~~~~~~x\otimes \frac{dT_i}{\xi} \mapsto x}_{~~~~~~~~x\otimes \frac{dT_j}{\xi} \mapsto 0} && M,}
\]
where $j$ above is not equals to $i$.
\begin{remark}
	As an exercise of Taylor expansion in calculus, it can be shown that in order to show the quasi-nilpotence, it is enough to check the convergence for a set of generators in \Cref{q-nil}.
\end{remark}
\begin{remark}\label{log conn, rmk on conn}
	The map $\wt \nabla_i$ should be thought as the multiplication of $\xi$ with the classically defined map $\nabla_i$, where the latter is the composition of $\nabla$ and the paring with $\frac{\partial}{\partial T_i}$.
	More explicitly the pairing map sends $\Omega_{P/\Bdr}^1$ to $P$ by the formula
	\[
	dT_i \lmt 1;~dT_j \lmt 0,~\forall j\neq i.
	\]
	This could be made precise either by inverting $\xi$ or when $M$ has no $\xi$-torsion.
\end{remark}
\begin{example}
	As a simple example of a quasi-nilpotent log connection, let $R=K \langle T\rangle,$ $P=\Bdr \langle T\rangle$ and $I=(0)$, so that $D_\Prism=P$.
	Let $(M,\nabla)=(\Bdr \langle T\rangle,d)$ be the natural continuous differential operator on the structure sheaf with $df(T)=f'(T)dT$.
	Then for an element $f(T)\in \Bdr \langle T\rangle$, the power series in Definition \ref{q-nil} is 
	\[
	\sum_e f^{(e)}(T)\cdot \xi^e\otimes \frac{\wt\delta^e}{e!}=f(T+\xi\wt\delta),
	\]
	which is convergent in the ring $P[{\wt\delta}]^\wedge=\Bdr \langle T\rangle[\wt\delta]^\wedge$, by checking mod $\xi^n$ for every $n\in \NN$.
	Thus the canonical continuous differential on $P$ is quasi-nilpotent.	
	
	In fact, as we will see, this is the (log) connection corresponds to the prismatic structure sheaf $\mathcal{O}_\Prism$ over $R/\Bdr_\Prism$.
\end{example}
\begin{example}\label{pm cry integral}
	If a log connection is integral, then it is quasi-nilpotent.
	To show this, it suffices to notice that by assumption in \Cref{pm cry integral def} when $(M,\nabla)$ is integral, each $\wt\nabla_i$(x) is divisible by $\xi$.
\end{example}
\begin{example}\label{log conn, non-eg}
	To give an example that is not quasi-nilpotent, still consider  the special setting as above when $R=K \langle T \rangle$, $P=\Bdr \langle T \rangle$ and $I$ is the zero ideal, in which case the envelope $D_\Prism$ is equal to the ring $P$.
	Let $M$ be $D_\Prism$ itself, and let $\nabla$ be the twisted differential map
	\[
	\nabla=\frac{d}{\xi}: \Bdr \langle  T\rangle \rra \Bdr \langle T\rangle\otimes \frac{dT}{\xi};~f(T)\lmt f'(T)\otimes\frac{dT}{\xi}.
	\]
	Then this is an integrable log connection of $M$ over $D_\Prism$ but is not nilpotent.
	To see the latter, it suffices to notice that given $f(T)\in \Bdr \langle  T\rangle$, the formal power series in Definition \ref{q-nil} is equal to
	\[
	\sum_e f^{(e)}(T)\otimes \frac{\wt \delta^e}{e!}=f(T+\wt \delta),
	\]
	which after mod $\xi$ is not convergent in $K\langle  T\rangle[\wt \delta]$ whenever $f$ mod $\xi$ is an infinite series.
\end{example}
\begin{proposition}\label{pm cry via envelope}
	Let $P=\Bdr\langle T_i\rangle$, let $I$ be an ideal of $P$, let $R$ be the  topologically finite type algebra $P/(I,\xi)$ and let $D_\Prism=P[\frac{I}{\xi}]^\wedge_\tf $ be the associated envelope.
	Then the following two categories are equivalent:
	\begin{align*}
		\{coherent~prismatic~crystals~over~R\} &\longrightarrow \{(M,\nabla)~|~M\in \Coh(D_\Prism),~\nabla~integrable~quasi\text{-}nilpotent~log~connection\}\\
		\calF & \lmt (\calF(D_\Prism), \nabla_{D_\Prism}).
	\end{align*}
\end{proposition}
\begin{proof}
	The proof is similar to that of \cite[Theorem 3.3.1]{Guo21}.
	We first notice that  by \Cref{pm uni} the prism $D_\Prism$ associated to the prismatic envelope for the surjection $P\ra P/I$ covers the final object of the topos.
	Let $\calF$ be a coherent prismatic crystal over $R$.
	By Taylor expansion, the formula in \Cref{q-nil} is exactly the transition isomorphism of the pullbacks along the two projections
	\[
	\phi:\mathrm{pr}_1^*\calF(D_\Prism) \ra \mathrm{pr}_0^*\calF(D_\Prism),
	\]
	where $\mathrm{pr}_i$ are two canonical maps of rings $D(1)_\Prism=P[\frac{I,\Delta(1)}{\xi}]^\wedge_\tf  \leftleftarrows P[\frac{I}{\xi}]$ induced by $P\langle \Delta(1)\rangle \leftleftarrows P$.
	Here $\Delta(1)$ is the kernel ideal for the diagonal surjection $P^{\wh\otimes_{\Bdr} 2}\ra P$.
	To define the log connection on $D_\Prism$, we note that the following map has the image in $\mathcal{F}(D_\Prism)\otimes_P \frac{\Delta(1)/\Delta(1)^2}{\xi}$ (which is isomorphic to $\cong \mathcal{F}(D_\Prism)\otimes_P \frac{\Omega_{P/\Bdr}^1}{\xi}$)
	\[
	\id \otimes 1 - \phi\circ (1\otimes \id): \mathcal{F}(D_\Prism) \rra \mathrm{pr}_0^*\mathcal{F}(D_\Prism)/(\frac{\Delta(1)^2}{\xi^2}).
	\]
	This allows us to define the log connection as 
	\[
	\nabla: \mathcal{F}(D_\Prism) \rra \mathcal{F}(D_\Prism)\otimes_P \frac{\Omega_{P/\Bdr}^1}{\xi}.
	\]
	Then as a consequence of the crystal condition, both $\mathrm{pr}_i^*\calF(D_\Prism)$ are isomorphic to $\calF(D(1)_\Prism)$, and the existence of the isomorphism implies the convergence of the power series.
	Moreover, integrability for the log connection coming from the crystal follows from the crystal conditions again, and can be proved as in \cite{Sta}[07J6].
	
	Conversely, given an integrable, quasi-nilpotent log connection $(M,\nabla)$ over an envelope, Lemma \ref{pm uni} allows us to define a finite module over any prism $(B,\xi B)$ via the pullback along a map $D_\Prism \ra B$.
	Then the rest is to show the independence of the choice of maps, which follows from the quasi-nilpotence and the integrability, and can be checked as in the proof of \cite[Theorem 3.3.1]{Guo21} (especially \cite[Claim 3.3.4]{Guo21}).
	Here we note that the quasi-nilpotence assumption is used to show the formal power series associated to the transition morphism converges, thus exists.
\end{proof}

We also compute the \v{C}ech-Alexander complex under the prismatic setting.
As a preparation, we introduce the following notations.
As before, let $P=\Bdr\langle T_1,\ldots, T_m\rangle\ra P/I \ra R=P/(I,\xi)$  be surjections, and let $D_\Prism$ be the associated envelope.
Denote $P(l)$ to be the inverse limit 
\[
P(l):=\varprojlim_n \left( \Bdrn\langle T_i\rangle^{\wh\otimes_{\Bdrn} (l+1)} \right)\cong \varprojlim_n \Bdrn \langle T_i,\delta_{j,i} \rangle,
\]
where $\delta_{j,i}$ for $1\leq j\leq l,~1\leq i\leq m$ correspond to natural generators of the kernel ideal $\Delta(l)$ for the diagonal surjection $P(l) \ra P$,
and the complete tensor product $\wh\otimes_{\Bdrn}$ above is defined as the integrally $p$-complete tensor product.
\begin{proposition}\label{pm product}
	Let $P, R$ and $P(l)$ be as above.
	The $(l+1)$-th self fiber  product of the prism $D_\Prism$ over the final object in the prismatic topos is representable by the prism associated to the prismatic envelope $D(l)_\Prism$ for $P(l)\ra P(l)/(\Delta(l),I)=P/I$.
\end{proposition}
\begin{proof}
	Let $(B,\xi B)$ be a prism over $R$.
	We first notice that by Remark \ref{pm envelope rep}, to give a map of prisms from $(D_\Prism,\xi D_\Prism)$ to $(B,\xi B)$, it is equivalent to give a continuous $\Bdr$-linear morphism $g:P=\Bdr\langle T_i\rangle \ra B$, satisfying the following commutative diagram
	\[
	\xymatrix{P \ar[r] \ar[d]_g &R \ar[d]^f \\ 
		B \ar[r] & B/\xi.}\tag{$\ast$}
	\]
	This implies that a map from $(l+1)$-th self product of the prism $(D_\Prism, \xi D_\Prism)$ to $(B,\xi B)$ is equivalent to $(l+1)$ continuous $\Bdr$-linear morphisms $g_0,\ldots, g_l:P \ra B$, such that the induced compositions $P \ra B \ra B/\xi$ all factor through the  morphism $f:R \ra B/\xi$.
	
	%We then recall that the topology on $P$ is the inverse limit topology induced by the $p$-adic topology on each $\Bdrn\langle T_i\rangle$, and the ring $R$ is only equipped with the $p$-adic topology.
	We then consider the integral level.
	To give $(l+1)$ continuous maps from $P$ to $B$, it suffices to give a sequence of $\Ainfn$-linear (for $\Ainfn=\Ainf/\xi^n$) continuous maps $g_{0,n},\ldots, g_{l,n}:\Ainfn\langle T_i\rangle \ra B/\xi^n$, compatible with $n\in \NN$ via natural surjections, and with the commutative diagrams $(\ast)$ above (after inverting $p$).
	When $n=1$, the maps $g_{j,n}$ all factors through the same map $\calO_K\langle T_i\rangle \ra R_0 \ra B/\xi$, where $R_0$ is a ring of definition of $R$ and $R_0\ra B/\xi$ is induced from $f:R \ra B/\xi$.
	This is equivalent to give a map $\calO_K\langle T_i\rangle^{\otimes_{\calO_K} (l+1)} \ra B/\xi$ that factors through the diagonal and the above surjection  $\calO_K\langle T_i\rangle \ra R_0 \ra B/\xi$.
	Notice that the tensor product is a priori only the ordinary tensor product of $\calO_K$-rings, but since $R_0$ is $p$-complete, the above induces uniquely a continuous map $\calO_K\langle T_i\rangle^{\wh\otimes_{\calO_K} (l+1)} \ra B/\xi$ for the $p$-adic completed tensor product.
	For general $n\in \NN$, the choice of $\{g_{0,n},\ldots, g_{l,n}\}$ amounts to give a map from $\Ainfn\langle T_i\rangle^{\otimes_{\Ainfn} (l+1)} \ra B/\xi^n$, whose composition with $B/\xi^n\ra B/\xi$ factors through the diagonal surjection of $\Ainfn\langle T_i\rangle^{\otimes_{\Ainfn} (l+1)}$ and $\Ainfn\langle T_i\rangle \ra R_0 \ra B/\xi$.
	Using \Cref{power-bounded on relative Bdr}, the images of incomplete tensor product in $B/\xi^n$ are within an affinoid algebra over $\Bdrn$, so the map uniquely factors through a map $\Ainfn\langle T_i\rangle^{\wh\otimes_{\Ainfn} (l+1)} \ra B/\xi^n$, where the tensor product of the source is $p$-completed.
	In this way, by inverting $p$ and taking the inverse limit with respect to $n$, we see to give $(l+1)$ continuous $\Bdr$-linear map from $P \ra B$ satisfying the diagram $(\ast)$, it is the same as giving the following commutative diagram
	\[
	\xymatrix{ P(l) \ar@{>>}[r] \ar[d] & P=\Bdr\langle T_i\rangle \ar@{>>}[r] & R \ar[ld]^f\\
		B \ar[r] & B/\xi &.}
	\]
	In this way, by the construction of the prismatic envelope in Definition \ref{pm envelope} and Remark \ref{pm envelope rep}, we see this is equivalent to give a map of prisms from $(D(l)_\Prism, \xi D(l)_\Prism)$ to $(B,\xi B)$, where $D(l)_\Prism$ is the prismatic envelope.
	
\end{proof}
\begin{corollary}\label{pm Cech}
	Let $P=\Bdr\langle T_i\rangle\ra P/I\ra R=P/(I,\xi)$ be surjections with $R$ being topologically finite type over $K$.
	Let $\calF$ be a sheaf over $X/\Bdr_\Prism$, and let $D(l)_\Prism$ be the prismatic envelope as in Lemma \ref{pm product}.
	Then the following natural map from the homotopy limit over the simplicial diagram $\Delta^\op$ is a quasi-isomorphism
	\[
	R\varprojlim_{[l]\in \Delta^\op} \calF(D(l)_\Prism,\xi D(l)_\Prism) \rra R\Gamma(X/\Bdr_\Prism,\calF).
	\]
\end{corollary}
\begin{proof}
	This follows from the general category theory and Proposition \ref{pm product}, as the the simplicial diagram of prisms $[l]\mapsto (D(l)_\Prism,\xi D(l)_\Prism)$ in the prismatic site $X/\Bdr_\Prism$ forms the \v{C}ech nerve of the covering map $(D_\Prism, \xi D_\Prism) \ra 1_{X/\Bdr_\Prism}$ (Lemma \ref{pm uni}), where $1_{X/\Bdr_\Prism}$ is the final object in the prismatic topos $\Sh(X/\Bdr_\Prism)$.
\end{proof}
As an application, we obtain a tensor product formula relating the cohomology of a crystal and its reduction.
\begin{corollary}\label{coh of red}
	Let $\mathcal{F}$ be a flat prismatic crystal over $X$, and let $\ol{\mathcal{F}}=\mathcal{F}\otimes_\Bdr K$ be its reduction.
	Then the natural map below is an isomorphism
	\[
	R\Gamma(X/\Bdr_\Prism, \mathcal{F})\otimes^L_\Bdr K \rra R\Gamma(X/\Bdr_\Prism, \ol{\mathcal{F}}).
	\]
\end{corollary}
\begin{proof}
	First notice the map of cohomology above is induced by the reduction map of crystals $\mathcal{F}\ra \ol{\mathcal{F}}$.
	So it suffices to check that the map of cohomology is an isomorphism, assuming $R$ admits a surjection as in \Cref{pm Cech}.
	Then we notice that by the local freeness of each $\mathcal{F}(D(l)_\Prism, \xi D(l)_\Prism)$ over $D(l)$ and the $\xi$-torsionfreeness of $D(l)$ over $\Bdr$, the derived tensor product of the cosimplicial diagram $[l]\mapsto  \mathcal{F}(D(l)_\Prism, \xi D(l)_\Prism)$ with $K$ is equal to 
	\[
	[l] \lmt \ol{\mathcal{F}}(D(l)_\Prism,\xi D(l)_\Prism)= \mathcal{F}(D(l)_\Prism ,\xi D(l)_\Prism) /\xi.
	\]
	Thus the rest follows from \Cref{pm Cech}.
\end{proof}

\section{From infinitesimal to prismatic}\label{sec inf to prism}
In this section, we show how to compute cohomology of a prismatic crystal locally by using infinitesimal cohomology, for rigid spaces that has l.c.i. singularity.

\subsection{A variant of Simpson's functor}\label{subsec Simpson}
We start this subsection with introducing a variant of Simpson's functor in \cite[Lemma 19]{Simp}, over the ($\infty$) category of filtered objects (rings, modules, or complexes) over a $\Bdr$-algebra.
This will be used to connect the infinitesimal theory to the prismatic theory.
\begin{construction}\label{Simp constr}
	Let $(M,\Fil^\bullet M)$ be a $\mathbb{N}^\op$-indexed descending filtered complex of $\Bdr$-modules.
	Namely $\Fil^\bullet M$ is an object in $\DF(\Mod_{\Bdr})=\Fun(\mathbb{N}^\op, \Mod_{\Bdr})$ as in the beginning of \cite[Section 5]{Guo21}.
	Then we define the functor $\Psi(M,\Fil^\bullet M)$ (or $\Psi(M)$ in short if there is no confusion) as an object in $\scrD(\Bdr)$ by the following formula
	\[
	\Psi(M,\Fil^\bullet M):= \left( \colim_{i\in\NN} \frac{Fil^iM}{\xi^i} \right)^\wedge.
	\]
	Here $(-)^\wedge$ above is the derived $\xi$-adic completion, and precisely the colimit is taken over the following diagram of $\Bdr$-complexes
	\[
	\xymatrix{ \Fil^0 M & \Fil^1 M \ar[l] \ar[d]^{\cdot \xi} & & \\
		& \Fil^1 M  & \Fil^2 M \ar[d]^{\cdot \xi} \ar[l] &\\
		& & \Fil^2 M & \Fil^3 M \ar[l] \ar[d]^{\cdot \xi} \\
		& & & \vdots ,}
	\]
	where all horizontal arrows are defining morphisms of the filtration.
\end{construction}
\begin{remark}\label{Simp diagram}
	The idea of the diagram in Construction \ref{Simp constr} is that we want to add the denominator $\xi^i$ to the $i$-th filtration.
	To see this, we change the notation and denote $\frac{\Fil^iM}{\xi^i}$ to be the $\Bdr$-complex $\Fil^i M$ but formally adjoining the denominator $\frac{1}{\xi^i}$.
	Then the diagram in \Cref{Simp constr} can be rewritten as the following
	\[
	\xymatrix{ \Fil^0 M & \Fil^1 M \ar[l] \ar[d] & & \\
		& \frac{\Fil^1 M}{\xi}  & \frac{\Fil^2 M}{\xi} \ar[d] \ar[l]&\\
		& & \frac{\Fil^2 M}{\xi^2}  & \frac{\Fil^3M}{\xi^2} \ar[l] \ar[d] \\
		& & & \vdots ,}
	\]
	where each vertical map identifies $\frac{Fil^i M}{\xi^{i-1}}$ with $\xi\cdot \frac{Fil^i M}{\xi^i}$ via the multiplication by $\xi$.
	Note that in the special case when each $\Fil^iM$ is a complex of $\xi$-torsionfree $\Bdr$-modules, the vertical maps are all inclusion maps.
	
\end{remark}
\begin{remark}\label{Simp CAlg}
	When $(M,\Fil^\bullet M)=(A,I^\bullet)$ is a noetherian $\xi$-torsionfree $\Bdr$-algebra $A$ together with a filtration given by a power of its ideal $I$, the module $\Psi(A,I^\bullet)$ is also equipped with a natural $\Bdr$-algebra structure, by
	\[
	\frac{I^i}{\xi^i}\otimes \frac{I^j}{\xi^j} \rra \frac{I^{i+j}}{\xi^{i+j}}.
	\]
	More generally, under the Day convolution on $\DF(\Mod_{\Bdr})$ (see for example \cite{GP18}), the functor $\Psi:\DF(\Mod_{\Bdr}) \ra 
	\scrD(\Bdr)$ is lax-symmetric monoidal, thus inducing a functor on $\mathbb{E}_\infty$:
	\[
	\Psi:\mathrm{CAlg}(\DF(\Mod_{\Bdr})) \rra \mathrm{CAlg}(\scrD(\Mod_{\Bdr})).
	\]
	Intuitively, using the notation in the above remark, this is to say the following multiplicative structure in $\Psi(M,\Fil^\bullet M)$
	\[
	\frac{\Fil^i M}{\xi^i} \otimes \frac{\Fil^j M}{\xi^j} \rra \frac{\Fil^{i+j} M}{\xi^{i+j}}.
	\]
\end{remark}
\begin{remark}\label{Simp A-linear}
	Moreover, when $(M,\Fil^\bullet M)$ is a filtered complex of $A$-modules for a $\xi$-adic complete $\Bdr$-algebra $A$, as the arrows in the diagram of \Cref{Simp constr} above are all $A$-linear, the object $\Psi(M,\Fil^\bullet M)$ is also a complex of $A$-modules.
\end{remark}
In the following, for a $K$-module $M$, we denote $M\otimes_K \xi^i/\xi^{i+1}$ as $M(i)$, called the \emph{$i$-th Tate twist} of $M$.
This would not change the underlying $K$-module structure, but its Galois structure is replaced by the $i$-th Tate twist when $M$ itself is equipped with one. 
\begin{remark}\label{Simp gr}
	Given a filtered $\Bdr$-complex $(M,\Fil^\bullet M)$ such that $\Fil^\bullet M$ is uniformly bounded to the right, we note that the derived mod $\xi$ reduction  has the following formula
	\[
	\Psi(M,\Fil^\bullet M)\otimes^L_{\Bdr} K \cong \bigoplus_{n\in \NN} \gr^n (M\otimes^L_\Bdr K)(-n),
	\]
	where $M\otimes^L_\Bdr K$ is the filtered complex of $K$-modules endowed with filtration $\Fil^iM\otimes^L_\Bdr K$.
	This follows from \cite[Tag 0EEV]{Sta}, the commutativity of colimit and derived tensor product, and the diagram in Construction \ref{Simp constr}.
\end{remark}

\subsection{Prisms associated to infinitesimal thickenings}\label{subs prism to inf}
We now construct a functor from the infinitesimal site to the prismatic site, over the de Rham period ring $\Bdr$.

To start, we consider the following pro-version analogue of the infinitesimal site considered in \cite{Guo21}.
\begin{definition}
	Let $X$ be a rigid space over $K$.
	The \emph{pro-infinitesimal site} of $X/\Bdr$, denoted as $X/\Bdr_\mathrm{pinf}$, is defined as the category of $(U,T_\ast):=(U,T_e)$ for $e\in \NN$, such that 
	\begin{itemize}
		\item each $(U,T_e)$ is an infinitesimal thickening of affinoid rigid spaces over $\Bdre$ (in the sense of \cite[Section 2.1]{Guo21}, with $T_e$ flat over $\Bdre$ and $U$ an affinoid open subspace of $X$;
		\item  transition maps $T_e \ra T_{e+1}$ are also infiniteisimal thickenings that are compatible with maps from $U$, and there is a compatible system of closed immersions $T_e\ra \Spa(\Bdre \langle x_1,\ldots, x_l\rangle)$ for some fixed $l\in \mathbb{N}$.
		%\item a collection of maps $\{(U_i,T_{i,e}) \ra (U,T_e)\}$ is a \emph{covering} if for each $e\in \NN$, the collection $\{(U_i,T_{i,e}) \ra (U,T_e)\}$ is an analytic covering of pairs.
	\end{itemize}
The category is equipped with indiscrete topology.
\end{definition}
\begin{remark}\label{inf final}
	Let $R$ be a topologically of finite type $K$-algebra, and let $P\ra R$ be a surjection for a smooth $\Bdr$-algebra $P$ (in the sense that each $P/\xi^e$ is smooth over $\Bdre$ as an adic space).
	Then the induced pro-infinitesimal thickening $(U=\Spa(R), T_e=\Spa(P/I^e))$ is weakly final over $\Spa(R)/\Bdr_\mathrm{inf}$, where $I=\ker(P\ra R)$.
\end{remark}
\begin{remark}
	Analogous to the classical infinitesimal theory, we can define the infinitesimal structure sheaf $\mathcal{O}_{X/\Bdr}$ and the notion of crystals over $X/\Bdr_\mathrm{pinf}$, as in \cite{Guo21}.
	Here for a pro-infinitesimal thickening $(U,T_e=\Spa(B_e))$ we have 
	\[
	\mathcal{O}_{X/\Bdr}(U,T_e)=\varprojlim_e B_e.
	\]
	
	Moreover, by \cite[Theorem 7.2.3, Theorem 3.3.1]{Guo21}, the category of crystals over $X/\Bdr_{\mathrm{inf}}$ in the sense of \cite{Guo21} is naturally equivalent to that over $X/\Bdr_\mathrm{pinf}$, inducing natural isomorphisms between their cohomology.
\end{remark}
Next we define a natural functor from the infinitesimal site to the prismatic site.
Let $R$ be a fixed topologically finite type algebra over $K$, and let $X=\Spa(R)$.
\begin{definition}\label{ass prism}
	Let $(U=\Spa(A),T_e=\Spa(B_e))\in X/\Bdr_\mathrm{pinf}$ be a pro-infinitesimal thickening over the rigid space $X=\Spa(R)$ over $K$, and let $B=\varprojlim_e B_e$, with the ideal $I=\ker(B\ra A)$.
	The \emph{associated prism} of $(U,T_e)$ is defined as 
	\[
	\left(B' \ra B'/\xi \leftarrow R\right),
	\]
	where $B'=B[\frac{I}{\xi}]^\wedge_\tf $, the symbol $(-)^\wedge$ is the classical $\xi$-adic completion, and the arrow on the right is the natural composition $R\ra A \ra B[\frac{I}{\xi}]^\wedge_\tf/\xi$.
\end{definition}
\begin{remark}\label{pm final}
	When $(U,T_e)$ is a pro-infinitesimal thickening as in \Cref{inf final}, by the construction above together with \Cref{pm uni}, the associated prism covers the final object of the prismatic site.
\end{remark}
By construction, this defines a map of ringed sites
\[
\pi:(X/\Bdr_\Prism,\mathcal{O}_\Prism) \rra (X/\Bdr_\mathrm{pinf}, \mathcal{O}_{X/\Bdr}).
\]
Moreover, by taking the coherent pullback, we can produce a prismatic crystal from an infinitesimal crystal.
\begin{lemmadef}\label{functor crys}
	There is a natural functor between categories of crystals 
	\[
	\pi^*:\mathrm{Cry}_{X/\Bdr_\mathrm{pinf}} \ra \mathrm{Cry}_{X/\Bdr_\Prism},
	\]
	 such that for an infinitesimal thickening $(U=\Spa(A),T_e=\Spa(B_e))$ with $(B'\ra B'/\xi\leftarrow R)$ being the associated prism as in \Cref{ass prism}, we have
	\[
	\pi^*\mathcal{F}(B')=\mathcal{F}(U,T_e)\otimes_B B'.
	\]
\end{lemmadef}
\begin{proof}
	For each $(U=\Spa(A), T_e=\Spa(B_e))$, we construct a module over $B'=B[\frac{I}{\xi}]^\wedge_\tf $ by taking the tensor product as in the formula above.
	Then by the crystal condition of $\mathcal{F}\in \mathrm{Cry}_{X/\Bdr_\mathrm{inf}}$, this satisfies the crystal condition with respect to prisms coming from infinitesimal thickenings.
	Then it suffices to notice that by \Cref{pm final} the category of associated prisms covers the final objects of the prismatic topos.
\end{proof}

\subsection{Local calculation for l.c.i singularities}\label{subsec loc cal on lci}
In this subsection, we give some concrete local formulae computing prismatic cohomology of crystals using Simpson's functor, for affinoid rigid space that has l.c.i singularities (in the sense of \cite[Appendix]{GL20}).
Precisely, we assume the following throughout the subsection.
\begin{assumption}\label{Assump}
	We let $R$ be a fixed topologically finite type $K$-algebra, and let $P$ be the ring $\Bdr \langle T_i\rangle$.
	Assume $I$ is a Koszul-regular ideal of $P$ (in the sense of \cite[Tag 07CU]{Sta}), such that $P/I$ is flat over $\Bdr$, and $P/(I,\xi)$ is isomorphic to $R$.
\end{assumption}
Slightly abuse the language, we call any such surjection $P \ra P/I$ as above a \emph{regular closed immersion}.
\begin{example}\label{pm envelope eg2}
	Assume $R$ is a topologically finite type algebra over $K$ that has l.c.i singularities.
	By definition in \cite[Proposition 5.3]{GL20}, the kernel ideal $\ol I$ of any surjection from $K \langle T_i\rangle$ onto $R$ is a Koszul-regular ideal in $K\langle T_i\rangle$.
	In particular, by \cite[Tag 0669]{Sta}, as the union of $\xi$ with a lift of any (Zariski) local Koszul-regular sequence of $\ol I$ is a Koszul-regular sequence, the ideal $\ker(P\ra R)$ is also a Koszul-regular ideal of $P$.
	So by taking any lift of Koszul-regular sequence of $\ol I$, we see Zariski (thus analytic) locally the affinoid algebra $R$ always admits a setup as in \Cref{Assump}.

\end{example}
\begin{remark}\label{sift}
	It is worth mentioning the category of (regular) closed immersions are sifted.
	This is because given any two surjections $P_i \ra P_i/I_i \ra R$ for $i=1,2$ and $P_i$ being of the form $\Bdr \langle T_j\rangle$, we can form the product 
	\[
	P:=\varprojlim_e \left( (P_1\otimes_\Bdr \Bdre)\otimes_\Bdre (P_2\otimes_\Bdr \Bdre) \right),
	\]
	which naturally admits a surjection to $R$ together with maps from $P_i$.
	Moreover, by taking the ideal $I\subset P$ generated by the image of $I_1$ and $I_2$, the quotient ring $P/I$ is a flat $\Bdr$-lift of $R$ along $\Bdr\ra K$.
	At last, when both surjections $P_i \ra R$ are locally complete intersection, we can use the analytic cotangent complex to check that $P \ra R$ is also regular (\cite[Appendix]{GL20}).
\end{remark}

We start by computing the prismatic envelope using Simpson's functor.
\begin{lemma}\label{prism envelope}
	Let $P$, $I$, and $R$ be as in \Cref{Assump}.
	Then the prismatic envelope $P[\frac{I}{\xi}]^\wedge_\tf $ for $P\ra P/I$ is naturally isomorphic to $\Psi(P,I^\bullet)$.
	In particular, $\Psi(P,I^\bullet)$ lives in cohomological degree zero and is a $\xi$-adically complete, $\xi$-torsionfree noetherian algebra over $P$.
\end{lemma}
\begin{proof}
	We first assume that $I$ admits a finite set of generators $(f_i)$ such that $(\xi, f_i)$ is a Koszul-regular sequence, which is always true Zariski locally around $\Spec(R)$ in $\Spec(P)$ (c.f. \Cref{pm envelope eg2}).
	Then there exists a natural map from the polynomial algebra $P[x_i]$ to $\Psi(P,I^\bullet)$, sending $x_i$ onto $\frac{f_i}{\xi}\in \frac{I}{\xi}$.
	This induces a natural homomorphism from the noetherian ring $P[\frac{f_i}{\xi}]=P[x_i]/(\xi\cdot x_i -f_i)$ to $\Psi(P,I^\bullet)$, and factors through the derived $\xi$-completion of $P[\frac{f_i}{\xi}]$.
	Note by \cite[Tag 0A06]{Sta} the classical $\xi$-completion and the derived $\xi$-completion of a noetherian $\Bdr$-algebra are the same.
	Thus the above leads to a natural map from the actual noetherian algebra $P[\frac{f_i}{\xi}]^\wedge$ to $\Psi(P,I^\bullet)$, where the former is classically $\xi$-completed.
	Moreover, as $(\xi, f_i)$ forms a regular sequence of the ring $P$, the map of rings $\Bdr[y_i] \ra P$, $y_i\mapsto f_i$ is flat after localizing at the ideal $(\xi)$ (c.f. \cite[Proposition 1]{Har66}).
	As a consequence, the ring $P[\frac{f_i}{\xi}]$ has no $\xi$-torsion, and we get a map from $P[\frac{I}{\xi}]^\wedge_\tf$ to $\Psi(P, I^\bullet)$.

	To show the isomorphism, by the derived Nakayama lemma (\cite[Tag 0G1U]{Sta}) for $\Bdr$, it suffices to check the above is an isomorphism after a derived tensor product with $K$.
	On the one hand, by \Cref{Simp gr} we have
	\[
	\Psi(P,I^\bullet)\otimes^L_\Bdr K \cong \bigoplus_{n\in \NN} \gr^n (P\otimes^L_\Bdr K).
	\]
	By the assumption of $P$, each $I^n/I^{n+1}$ is Tor-independent of $K$ over $\Bdr$.
	So the above tensor product is equal to the discrete module as below
	\[
	\bigoplus_{n\in \NN} (\ol I)^n/(\ol I)^{n+1},
	\]
	where $\ol I$ is the tensor product $I\otimes_\Bdr K$.
	On the other hand, the assumption of $P$ and $I=(f_i)$ implies the  ring $P[x_1, \ldots, x_m]/(\xi x_i-f_i, 1\leq i\leq m)$ is flat over $\Bdr$.
	In particular, we have
	\begin{align*}
		P[\frac{f_i}{\xi}]\otimes^L_\Bdr K & \cong P[\frac{f_i}{\xi}]\otimes_\Bdr K\\
		&\cong P/(\xi ,f_i)[x_i],
	\end{align*}
	which by assumption is exactly the direct sum $\bigoplus_{n\in \NN} (\ol I)^n/(\ol I)^{n+1}$.
	In this way, as the above two reductions are abstractly isomorphic to each other, to finish the proof, it suffices to notice that the map we constructed identifies the generators of each direct summand.
	
	In general, one can still form the map from $P[\frac{I}{\xi}]^\wedge_\tf $ to $\Psi(P,I^\bullet)$, replacing $(f_i)$ above by any finite set of generators of $I$.
	As the ring $P$ is noetherian, we can switch the order of the torsionfree quotient and the $\xi$-completion to get 
	\[
	P[\frac{I}{\xi}]^\wedge_\tf  = \left( P[\frac{f_i}{\xi}]/(\xi^\infty\text{-torsion}) \right)^\wedge.
	\]
	So to show the isomorphism, it suffices to take the derived reduction mod $\xi$ and consider the following
	\[
	\left( P[\frac{f_i}{\xi}]/(\xi^\infty\text{-torsion}) \right) \otimes^L_\Bdr K \ra \Psi(P, I^\bullet)\otimes^L_\Bdr K.
	\]
	But notice that the map commutes with any Zariski localization of $P$, and for each $\frakp\in \Spec(R)$ there is  a Zariski open neighborhood of $\frakp$ in $\Spec(P)$ such that $I$ admits a set of Koszul-regular generators and the above is an isomorphism.
	Thus a \v{C}ech complex argument for an open covering finishes the proof.
	
\end{proof}

From now on towards the end of the subsection, we use $D_{\inf}$ to denote the infinitesimal envelope $\varprojlim P/I^m$ for the surjection $P\ra R$ (c.f. \cite[Section 2.2]{Guo21}), and use $D_\Prism=P[\frac{I}{\xi}]^\wedge_\tf$ for the prismatic envelope, which by \Cref{prism envelope} is isomorphic to $\Psi(P,I^\bullet)$.

Next we use Simpson's functor to relate a connection over $D_{\inf}$ to a log connection over $D_\Prism$.
\begin{theorem}\label{inf to pm conn}
	Let $P,I, R$ and $D_{\inf}, D_\Prism$ be as above, let $\mathcal{F}$ be a flat infinitesimal crystal over $R/\Bdr_{\mathrm{pinf}}$, and let $(M, \nabla)$ the be the integrable connection of $\mathcal{F}$ at $D_{\inf}$ (as in \cite[Section 3]{Guo21}).
	Then the log connection of $\pi^*\mathcal{F}$ at $D_\Prism$ is integral, integrable, and is naturally isomorphic to the unique extension of $\nabla$ on $M\otimes_{D_{\inf}} D_\Prism$ by the Leibniz rule.
\end{theorem}
By \Cref{functor crys}, we can denote the log connection as $(\pi^*M, \nabla)$, with $\pi^*M\cong M\otimes_{D_{\inf}} D_\Prism$.
Note that by \Cref{pm cry integral}, the log connection $(\pi^*M, \nabla)$ is in particular quasi-nilpotent.

\begin{comment}
	Let $P,I, R$ and $D_{\inf}, D_\Prism$ be as above.
	The Simpson functor induces a faithful embedding of categories as below
	\begin{align*}
		\{\text{locally free connections over }D_{\inf}\} & \rra \{\text{quasi-nilpotent locally free log connections over }D_\Prism\};\\
		(M,\nabla) & \lmt (\Psi(M,I^\bullet),\Psi((\nabla)),
	\end{align*}
	such that there is a natural isomorphism $\Psi(M,I^\bullet)\cong \pi^*M=M\otimes_{D_{\inf}} D_\Prism$, and $\Psi(\nabla)$ is the unique extension of $\nabla$ using Leibniz rule.
	Here the essential image of the functor is the collection of those that are integral.
	Moreover, the functor preserves the integrability.
\end{comment}
\begin{proof}
	By proof of \Cref{pm cry via envelope}, to describe the log connection structure on $\pi^*M=M\otimes_{D_{\inf}} D_\Prism$ (see \Cref{functor crys}), it suffices to consider the difference map
	\[
	\phi_\Prism:\mathrm{pr}_1^* (\pi^*\mathcal{F})(D_\Prism) \rra \mathrm{pr}_0^*(\pi^*\mathcal{F})(D_\Prism),
	\]
	where $\phi_\Prism$ is the isomorphism of $D(1)_\Prism$-modules induced from the crystal condition.
	Here by \Cref{functor crys}, the isomorphism $\phi_\Prism$ is naturally isomorphic to $\phi_{\inf}\otimes_{D(1)_{\inf}} D(1)_\Prism$, with 
	\[
	\phi_{\inf}= \left( \xymatrix{\mathrm{pr}_1^* \mathcal{F}(D_{\inf}) \ar[r]^\sim & \mathrm{pr}_0^*\mathcal{F}(D_{\inf})} \right).
	\]
	On the other hand, notice that the isomorphism $\phi_{\inf}$ is filtered under the $(I,\Delta(1))$-adic filtration, and in particular by applying $\Psi(-, (I,\Delta(1))^\bullet)$ and \Cref{prism envelope} we have
	\[
	\phi_\Prism=\Psi(\phi_{\inf}, (I,\Delta(1))^\bullet).
	\]
	Thus by taking the difference between $\id\otimes 1$ and $\phi_\Prism$ and mod the ideal $(\frac{\Delta(1)}{\xi})^2$, we get a natural log connection on $\pi^*M=M\otimes_{D_{\inf}} D_\Prism$.
	
	\begin{comment}
	For the finite free connection $(M,\nabla)$ over $D_{\inf}$, we define a filtration on $M\otimes_{D_{\inf}} \Omega_{D_{\inf}}^1=M\otimes_P \Omega_{P/\Bdr}^1$, by assigning $I^iM\otimes_P \Omega_{P/\Bdr}^1$ as the $i+1$-th filtration for $i\geq 0$, and $M\otimes_P \Omega_{P/\Bdr}^1$ the $0$-th filtration.
	Then the differential $\nabla$ is filtered under the $I$-adic filtration on $M$ and the filtration on $M\otimes_{D_{\inf}} \Omega_{D_{\inf}}^1$.
	In particular, we can apply Simpson's functor to get a natural $\Bdr$-linear map
	\[
	\Psi(\nabla): \Psi(M,I^\bullet) \rra \Psi(M\otimes_P \Omega_{P/\Bdr}^1,\Fil^\bullet),
	\] 
	where both terms admit natural action by $\Psi(P,I^\bullet)=D_\Prism$.
	This could be improved to all higher degrees $\nabla^j:M\otimes_P\Omega_{P/\Bdr}^j \ra M\otimes_P \Omega_{P/\Bdr}^{j+1}$ and is compatible with compositions up to homotopies.
	Thus by the definitions of the integrability in both cases (c.f. \cite[Section 3.1]{Guo21} and the discussion after Definition \ref{pm, log conn}), we see it is preserved under Simpson's functor.
	
	Moreover, since $M$ is finite free over $D_{\inf}$ and thus $\xi$-torsion free, by checking graded pieces (\Cref{Simp gr}) we have the natural isomorphisms
	\begin{align*}
		\Psi(M,I^\bullet)&\cong M\otimes_{D_{\inf}} D_{\Prism},\\
		\Psi(M\otimes_P \Omega_{P/\Bdr}^1,\Fil^\bullet) &\cong M\otimes_{D_{\inf}}D_\Prism \otimes_P \frac{\Omega_{P/\Bdr}^1}{\xi},
	\end{align*}
	which  implies $\Psi(M,I^\bullet)$ is discrete and finite free over $D_{\Prism}$.
	\end{comment}
	Moreover, by the property of colimits, there is a natural commutative diagram sending the connection $(M,\nabla)$ to the log connection on $\pi^*M$, compatible with the natural structure map $D_{\inf} \ra D_\Prism$:
	\[
	\xymatrix{
		M \ar[r]^\nabla \ar[d]& M\otimes_P \Omega_{P/\Bdr}^1 \ar[d]\\
		M\otimes_{D_{\inf}} D_\Prism \ar[r] & (M\otimes_{D_{\inf}}D_\Prism) \otimes_P \frac{\Omega_{P/\Bdr}^1}{\xi},}
	\]
	where the vertical maps are natural maps compatible with the inclusion $\Omega_{P/\Bdr}^1\ra \frac{\Omega_{P/\Bdr}^1}{\xi}$ over the ring $P$.
	From the diagram, we see the log connection on $\pi^*M$ coincides with $\nabla$ when restricted to $M$, and is uniquely extended by $\nabla$ using the Leibniz rule of the log connection.
	
	At last, the integrability and the quasi-nilpotence either follow from the general result in \Cref{pm cry via envelope}, or the above commutative diagram and Leibniz rule, as $\wt \nabla_i(x)$ is $\xi$-divisible for any $i$ and $x\in M$, and thus by Definition \ref{q-nil} the connection $\Psi(\nabla)$ is integral and thus quasi-nilpotent.

\end{proof}
The following lemma is useful when relating the log connection $\pi^*M$ with the filtration.
\begin{lemma}\label{prism envelope 2}
	Assume the same as in \Cref{inf to pm conn}.
	Then $\pi^*M$ is naturally isomorphic to $\Psi(M,I^\bullet M)$.
\end{lemma}
\begin{proof}
	There is a natural map below
	\[
	M\otimes_{D_{\inf}} D_\Prism = M\otimes_{D_{\inf}}^L \Psi(D_{\inf}, I^\bullet) \rra \Psi(M,I^\bullet M).
	\]
	By the commutativity between the derived tensor product functor and the colimit, the above is an isomorphism before the derived $\xi$-adic completion for $\Psi(-,-)$ as in \Cref{Simp constr}.
	So thanks to the derived $\xi$-completeness of $D_\Prism=\Psi(D_{\inf}, I^\bullet)$ and the finiteness of $M$ over $D_{\inf}$, we see the left hand side above is automatically $\xi$-complete, and hence isomorphic to the right.
\end{proof}
In the l.c.i case, we also get the faithfulness of the functor from infinitesimal crystals to prismatic crystals.
\begin{proposition}
	Let $X$ be an affinoid rigid space over $K$ that is a local complete intersection.
	The functor $\pi^*$ in \Cref{functor crys} from the category of flat infinitesimal crystals over $X/\Bdr_{\mathrm{pinf}}$ to the category of flat prismatic crystals over $X/\Bdr_\Prism$ is faithful.
\end{proposition}
\begin{proof}
	By the tensor product formula of $\pi^*M$ and the compatiblility between connections and log connections, it suffices to show that under Assumption \ref{Assump} for the ring $P=\Bdr \langle T_i\rangle$, the ideal $I$ and the ring $R=P/(I,\xi)$, the natural map $D_{\inf}\ra D_\Prism=P[\frac{I}{\xi}]^\wedge_\tf$ is an inclusion.

	By definition and \Cref{prism envelope}, we have $D_{\inf}$ is the $I$-adic completion of $P$, and $D_\Prism$ is the $\xi$-adic completion of $P[\frac{I}{\xi}]_\tf$.
	So thanks to the left exactness of the inverse limit, it suffices to show the injectivity for each $n\in \NN$ below
	\[
	P/(I^n,\xi^n) \rra P[\frac{I}{\xi}]_\tf/\xi^n, \tag{$\ast$}
	\]
	where we use the equality $\varprojlim P/(I^n,\xi^n) =\varprojlim P/I^n$ by the $\xi$-adic completeness of $P$.
	This can be checked Zariski locally, where we can assume $I$ admits a finite set of regular generators $(f_1,\ldots ,f_l)$.
	In this case, the natural map below is flat
	\[
	\xymatrix{ \Bdr[y_j] \ar[r]^{y_j\mapsto f_j~~~~} & P=\Bdr \langle T_i\rangle}
	\]
	Thus the injectivity of $(\ast)$ for each $n$ follows from that of 
	\[
	\Bdrn[y_j] \ra \Bdrn[x_j]/(\xi\cdot x_j-y_j),
	\]
	by base change along the flat map $\Bdr[y_j] \ra P$.

\end{proof}

\begin{remark}
	From \Cref{inf to pm conn}, the log connection associated to the image of $\pi^*$ is integral and integrable.
	One might also be able to use the explicit formula of the connection to describe the essential image of the functor $\pi^*$.
\end{remark}

We then give a formula computing prismatic cohomology of a crystal that comes from the infinitesimal site, using Simpson's functor and the de Rham complex.

Assume there is a system of affinoid rigid spaces  $\{X_e\}$, with each $X_e$ flat over $\Bdre$, lifting a given $K$-rigid space $X$.
Let $\mathcal{F}$ be a crystal over $X/\Bdr_{\mathrm{pinf}}$, and let $\mathcal{F}_e$ be its restriction to $X_e$, as a crystal over the infinitesimal site $X_e/\Bdre_{\inf}$ (c.f. \cite[Section 2.4]{Guo21}).
Following \cite[Section 2]{Guo21}, the $I_{X_e/\Bdre}$-adic filtration on an infinitesimal crystal $\calF_e$ induces a natural filtration of $R\Gamma(X_e/\Bdre_{\inf}, \calF_e)$, compatible with the closed immersion $X_e/\Bdre_{\inf} \ra X_{e+1}/\mathrm{B^+_{dR,e+1}}_{\inf}$.
In particular, by the limit formula $R\Gamma(X/\Bdr_{\mathrm{pinf}}, \mathcal{F})=R\varprojlim_e R\Gamma(X_e/\Bdre_{\inf},\mathcal{F}_e)$, the above  defines a filtration on infinitesimal cohomology of $X/\Bdr$ as below
\begin{definition}\label{lifted inf}
	Let $X$ be a rigid space over $K$, and let $\{X_e\}$ be a compatible system of $\Bdre$-flat lifts of $X$.
	For a crystal $\mathcal{F}$ over $X/\Bdr_{\mathrm{pinf}}$, the $i$-th \emph{lifted (infinitesimal) filtration} on its cohomology is
	\[
	\wt\Fil^iR\Gamma(X/\Bdr_{\mathrm{pinf}}, \mathcal{F}) := \varprojlim_e R\Gamma(X_e/\Bdre_{\inf}, \mathcal{I}_{X_e/\Bdre}^i\mathcal{F}_e).
	\]
\end{definition}
Note that similar to \cite[Theorem 7.2.3 (iii)]{Guo21}, the derived tensor product of the lifted filtration with $\Bdre$ over $\Bdr$ is isomorphic to the filtration of $R\Gamma(X_e/\Bdre_{\inf},\calF_e)$.

\begin{remark}
	When $P/I$ is smooth over $\Bdr$ (in the sense that each reduction mod $\xi^e$ is a smooth affinoid algebra over $\Bdre$), the $i$-th lifted filtration coincides with the $i$-th Hodge filtration of the de Rham complex of $\mathcal{F}$ over $P/I$.
	In particular, this is the $\Bdr$-linear lift of $i$-th the Hodge filtration of the de Rham complex of $\mathcal{F}(R)$ over $K$, hence the name.
\end{remark}
\begin{remark}
	When $X$ is a rigid space defined over a discretely valued subfield $k$, it admits a canonical (depending on a choice embedding $k\ra \Bdr$) $G_k$-equivariant lifts $\{X_e\}$, where each $X_e$ is the complete base extension of the $k$-rigid space with $\Bdre$.
	In particular the lifted filtration admits a natural action by Galois group.
\end{remark}
\begin{theorem}\label{prism coh}
	Let $P,I,R$ be as in \Cref{Assump} with $X=\Spa(R)$, and let $\calF$ be a flat infinitesimal crystal over $X/\Bdr_{\mathrm{pinf}}$.
	There is a natural isomorphism, functorial in $\mathcal{F}$ and the surjection $P\ra P/I$, computing cohomology of the prismatic crystal $\pi^*\mathcal{F}$
	\[
	\Psi(R\Gamma(X/\Bdr_{\mathrm{pinf}}, \mathcal{F}), \wt \Fil^\bullet) \cong R\Gamma(X/\Bdr_\Prism, \pi^*\mathcal{F}).
	\]	
\end{theorem}
\begin{proof}
	The idea is to compare \v{C}ech-Alexander complexes of two sides.
	Denote $D(n)_{\inf}$ to be the formal completion of the complete tensor product $P(n)=P^{\wh\otimes_{\Bdr} n+1}:= \Bdr \langle T_i, \delta_{i,j}; 1\leq i\leq l, 1\leq j\leq n \rangle$ along the surjection onto $P/I$, and let $\Delta(n)$ be the ideal of $D(n)_{\inf}$ defined by $(\delta_{i,j})$.
	Let $D(n)_\Prism =P(n)[\frac{I, \Delta(n)}{\xi}]^\wedge_\tf\cong \Psi(P(n),(I,\Delta(n))^\bullet)$ be the prismatic envelope of $P/I$ inside of $P^{\wh\otimes_{\Bdr} n+1}$ (\Cref{prism envelope}).
	Note that by \Cref{pm envelope alt}, the ring $D(n)_\Prism$ is also naturally isomorphic to $D(n)_{\inf}[\frac{I,\Delta(n)}{\xi}]^\wedge_\tf$.
	
	We then turn to the \v{C}ech-Alexander complexes of $\calF$ and $\pi^*\mathcal{F}$.
	By \cite[Proposition 2.2.7]{Guo21} and its $\Bdr$-linear version \cite[Theorem 7.2.3]{Guo21}, we have a natural filtered isomorphism
	\[
	R\Gamma(X/\Bdr_{\mathrm{pinf}},\mathcal{F}) \cong R\varprojlim_{[n]\in \Delta^\op} \mathcal{F}(D(n)_{\inf}), \tag{1}
	\]
	where the left side is endowed with the lifted filtration, and each term of the right side is equipped with the $(I,\Delta(n))$-adic filtration (as $(I,\Delta(n))$ is the kernel ideal of the pro-infiniteisimal thickening $D(n)_{\inf} \ra P/I$).
	On the other hand, by Corollary \ref{pm Cech} we have the natural isomorphism
	\[
	R\Gamma(X/\Bdr_\Prism, \pi^*\mathcal{F}) \cong R\varprojlim_{[n]\in \Delta^\op} \mathcal{F}(D(n)_\Prism), \tag{2}
	\]
	Moreover, by the construction of $\pi^*\mathcal{F}$ in \Cref{inf to pm conn} and  \Cref{prism envelope 2} the right hand side above is equal to
	\[
	R\varprojlim_{[n]\in \Delta^\op} \Psi(\mathcal{F}(D(n)_{\inf}),(I,\Delta(n))^\bullet).
	\]
	In this way, by applying $\Psi$ at the sequence (1) using the lifted filtration, and notice that homotopy limit along $\Delta^\op$ commuts with $\Psi$, we get
	\begin{align*}
		R\Gamma(X/\Bdr_\Prism,\pi^*\mathcal{F}) &\cong  R\varprojlim_{[n]\in \Delta^\op} \Psi(\mathcal{F}(D(n)_\Prism), (I,\Delta(n))^\bullet)\\
		&\cong  \Psi(R\varprojlim_{[n]\in \Delta^\op} \mathcal{F}(D(n)_{\inf}), \wt\Fil^\bullet ) \\
		&\cong \Psi(R\Gamma(X/\Bdr_{\mathrm{pinf}},\mathcal{F}), \wt\Fil^\bullet).
	\end{align*}

\end{proof}
As a corollary, we have the following explicit formula computing prismatic cohomology.
\begin{corollary}\label{prism dR}
	Let $P$, $I$, $R$ be as in \Cref{Assump}.
	Let $\mathcal{F}$ be a flat infinitesimal crystal over $X/\Bdr_{\mathrm{pinf}}$, and let $M$ be its section at the pro infinitesimal thickening $D_{\inf}=\varprojlim_n P/I^n$.
	Then there  is an isomorphism, functorial in $\mathcal{F}$, as below
	\[
	R\Gamma(X/\Bdr_\Prism, \pi^*\mathcal{F}) \cong \left( \wt M \ra \wt M\underset{P}{\motimes}\frac{\Omega_{P/\Bdr}^1}{\xi} \ra \cdots \ra \wt M\underset{P}{\motimes} \frac{\Omega_{P/\Bdr}^d}{\xi^d}\right),
	\]
	where $\wt M=\pi^*\mathcal{F}(P[\frac{I}{\xi}]^\wedge_\tf)\cong M\otimes_{D_{\inf}} P[\frac{I}{\xi}]^\wedge_\tf$ is the section of $\pi^*\mathcal{F}$ at the prismatic envelope $P[\frac{I}{\xi}]^\wedge_\tf$, and the right hand side is the de Rham complex associated to the log connection of $\wt M$
\end{corollary}
\begin{proof}
	Using \Cref{prism coh}, cohomology of $\pi^*\mathcal{F}$ is isomorphic to $\Psi(R\Gamma(X/\Bdr_{\mathrm{pinf}}, \mathcal{F}), \wt \Fil^\bullet)$.
	Then by taking an inverse limit of \cite[Theorem 4.1.1]{Guo21} with respect to $\Spa(P/(I,\xi^e))$ over $\Bdre$, we have
	\[
	\Psi(R\Gamma(X/\Bdr_{\mathrm{pinf}}, \mathcal{F}), \wt \Fil^\bullet) \cong \Psi(M\underset{P}{\otimes} \Omega_{P/\Bdr}^\bullet, \wt\Fil^\bullet),
	\]
	with $\wt\Fil^\bullet$ is the lifted infinitesimal filtration on the de Rham complex.
	Thus by writing out the right hand side above explicitly, we get the conclusion.
\end{proof}
In the special case when $R$ is smooth over $K$, we can obtain a simpler formula of cohomology of $\pi^*\mathcal{F}$ using the de Rham complex.
\begin{corollary}\label{coh of envelope}
	Assume $X=\Spa(R)$ is smooth of dimension $d$ over $K$, and $\wt R=\varprojlim_e \wt R/\xi^e$ is a smooth lift over $\Bdr$.
	Let $\mathcal{F}$ be a flat infinitesimal crystal over $X/\Bdr_{\mathrm{pinf}}$, and let $M$ be its section at the pro infinitesimal thickening $\wt R \ra R$.
	Then we have
	\[
	R\Gamma(X/\Bdr_\Prism,\pi^*\mathcal{F}) \cong \left(M \ra M\underset{\wt R}{\otimes} \frac{\Omega_{\wt R/\Bdr}^1}{\xi} \ra \cdots \ra M\underset{\wt R}{\otimes} \frac{\Omega_{\wt R/\Bdr}^d}{\xi^d} \right).
	\]
\end{corollary}
\begin{proof}
	As a smooth lift $\wt R/\xi^e$ itself is an infinitesimal envelope of $R$ into a smooth rigid space over $\Bdre$, we can apply Theorem \ref{prism coh} at the filtered isomorphism $R\Gamma(X/\Bdre_{\inf}, \mathcal{F})\cong (M/\xi^e\otimes_{\wt R/\xi^e} \Omega_{(\wt R/\xi^e)/\Bdre}^\bullet)$ in \cite[Theorem 4.1.1]{Guo21} to get the result.
\end{proof}

Another special case is when $\mathcal{F}$ is the infinitesimal structure sheaf, in which case $\pi^*\mathcal{F}$ is the prismatic structure sheaf $\mathcal{O}_\Prism$.
In this case, we can compute prismatic cohomology using Simpson's functor and the analytic derived de Rham complex.
\begin{corollary}
	Let  $P$, $I$, $R$ be as in \Cref{Assump}.
	Then there is an isomorphism as below
	\[
	R\Gamma(X/\Bdr_\Prism, \mathcal{O}_\Prism) \cong \Psi(\dR_{(P/I)/\Bdr}).
	\]
\end{corollary}
\begin{proof}
	Recall the analytic derived de Rham complex $\dR_{(P/I)/\Bdr}$ over $\Bdr$ is defined as the filtered object
	\[
	R\varprojlim_e \dR_{(P/(I,\xi^e))/\Bdre},
	\]
	where $\dR_{(P/(I,\xi^e))/\Bdre}$ is constructed as in \cite[Section 5]{Guo21} for affinoid algebras over $\Bdre$.
	By \cite[Corollary 5.5.2]{Guo21}, assuming \Cref{Assump}, the derived de Rham complex $\dR_{(P/(I,\xi^e))/\Bdre}$ is filtered isomorphic to $R\Gamma(\Spa(P/(I,\xi^e))/\Bdre_{\inf}, \mathcal{O}_{e \inf})$.
	Thus the formula follows from Theorem \ref{prism coh}.
\end{proof}
\begin{corollary}\label{lift coh}
	Assume $X=\Spa(R)$ is smooth of dimension $d$ over $K$, and $\wt R=\varprojlim_e \wt R/\xi^e$ is a smooth lift over $\Bdr$.
	Then we have
	\[
	R\Gamma(X/\Bdr_\Prism, \mathcal{O}_\Prism) \cong \left( \wt R\ra \frac{\Omega_{\wt R/\Bdr}^1}{\xi} \ra \cdots \ra \frac{\Omega_{\wt R/\Bdr}^d}{\xi^d} \right).
	\]
\end{corollary}

To finish this section, we show that a choice of regular immersion induces a splitting of reduced prismatic cohomology.
\begin{proposition}\label{split of HT}
	Let $X$ be a rigid space over $K$, and let $\{Y_e\}_{e\in \mathbb{N}}$ be a system of smooth rigid spaces over $\Bdre$ with $Y_{e+1}\times_{\mathrm{B_{dR,e+1}^+}} \Bdre \cong Y_e$, and $X\ra Y_0$ is a regular closed immersion.
	Let $\mathcal{F}$ be a flat infinitesimal crystal over $X/\Bdr_{\mathrm{pinf}}$.
	Then the choice of $\{Y_e\}$ induces an isomorphism as below
	\[
	R\Gamma (X/\Bdr_\Prism, \ol{\pi^*\mathcal{F}}) \cong \bigoplus_n \gr^n R\Gamma(X/K_{\mathrm{inf}}, \ol{\mathcal{F}}),
	\]
	where the latter is the direct sum of graded pieces for the infinitesimal filtration over cohomology of the crystal $\ol{\mathcal{F}}=\mathcal{F}/\xi$ over $X/K_{\mathrm{pinf}}$.
\end{proposition}
\begin{proof}
	By the global version of Theorem \ref{prism coh} and \Cref{coh of red}, the regular immersion induces an isomorphism between the left hand side above and the following
	\[
	\Psi(R\Gamma(X/\Bdr_{\mathrm{pinf}}, \mathcal{F}), \wt \Fil^\bullet)\otimes^L_\Bdr K,
	\]
	where $\wt \Fil^\bullet$ is the lifted infinitesimal filtration.
	By \Cref{Simp gr} it is further equal to 
	\[
	\bigoplus_n \gr^n \left( R\Gamma(X/\Bdr_{\mathrm{pinf}}, \mathcal{F})\otimes^L_\Bdr K \right)(-n).
	\]
	Here the filtration is the infinitesimal filtration.
	At last, by \cite[Theorem 7.2.3]{Guo21}, the base change $R\Gamma(X/\Bdr_{\mathrm{pinf}}, \mathcal{F})\otimes^L_\Bdr K$ is isomorphic to infinitesimal cohomology $R\Gamma(X/K_{\mathrm{inf}}, \ol{\mathcal{F}})$ with the canonical infinitesimal filtration.
\end{proof}
Note that in the special case when $\mathcal{F}=\mathcal{O}_{X/\Bdr}$ is the infinitesimal structure sheaf, the $i$-th graded piece above is isomorphic to $R\Gamma(X, L\wedge^i\mathbb{L}^\an_{R/K})(-i)[-i]$.

\section{Hodge-Tate filtration}\label{sec HT}
Using the local computation from the last section and the simplicial resolutions, we are able to prove the Hodge-Tate filtration theorem.

\subsection{Smooth case}
In the smooth case, we first construct a natural map from Hodge-Tate cohomology to prismatic cohomology, using the Bockstein operator as in \cite{BS19}.
Let $R$ be a topological finite type algebra over $K$.
Consider the short exact sequence of $\Bdr$-modules
\[
\xi^{i+1}/\xi^{i+2} \rra \xi^i/\xi^{i+2} \rra \xi^i/\xi^{i+1}.
\]
We take the derived tensor product of the sequence with $R\Gamma(R/\Bdr_\Prism, \mathcal{O}_\Prism)$ over $\Bdr$, then the connecting morphism induces the following
\[
R\Gamma(R/\Bdr_\Prism, \mathcal{O}_\Prism) \otimes^L_\Bdr K\cdot \xi^i \rra R\Gamma(R/\Bdr_\Prism, \mathcal{O}_\Prism) \otimes^L_\Bdr K\cdot \xi^{i+1} [1].
\]
By the reduction isomorphism in \Cref{coh of red} we have
\[
R\Gamma(R/\Bdr_\Prism, \mathcal{O}_\Prism) \otimes^L_\Bdr K \cong R\Gamma(R/\Bdr_\Prism, \ol{\mathcal{O}}).
\]
Thus we can rewrite the connecting morphism as below
\[
\beta_\xi:R\Gamma(R/\Bdr_\Prism, \ol{\mathcal{O}}_\Prism)(i) \rra R\Gamma(R/\Bdr_\Prism, \ol{\mathcal{O}}_\Prism)(i+1)[1].
\]

As in \cite[Section 4]{BS19}, by evaluating at the $i$-th cohomology group, the above induces a $K$-linear cdga over $R$
\[
(\mathrm{H}^\bullet(R/\Bdr_\Prism,\ol{\mathcal{O}}_\Prism)(\bullet), \beta_\xi).
\]
So by the universal property of the algebraic de Rham complex $\Omega_{R/K}^{\bullet, alg}$ of $R/K$, we get a natural map of cdga
\[
(\Omega_{R/K}^{\bullet,alg}, d) \rra (\mathrm{H}^\bullet(R/\Bdr_\Prism,\ol{\mathcal{O}}_\Prism)(\bullet), \beta_\xi).
\]

The next result proves that the map induces an isomorphism of of Hodge-Tate cohomology with the analytic differentials when $R$ is smooth.
\begin{theorem}\label{HT comp}
	Assume $R$ is a smooth affinoid algebra over $K$.
	Then there is a natural isomorphism
	\[
	\mathrm{H}^i(R/\Bdr_\Prism, \ol{\mathcal{O}}_\Prism) \cong \Omega_{R/K}^i(-i),
	\]
	where $\Omega_{R/K}^i$ is the sheaf of $i$-th analytic differentials.
\end{theorem}
Assume the above case when $R$ is smooth and affinoid, we call the Postnikov filtration of $R\Gamma(R/\Bdr_\Prism, \ol{\mathcal{O}}_\Prism)$ the \emph{Hodge-Tate filtration}.
\begin{proof}
	As we have constructed a natural map from algebraic de Rham complex of $R/K$ to the cdga $(\mathrm{H}^\bullet(R/\Bdr_\Prism,\ol{\mathcal{O}}_\Prism)(\bullet), \beta_\xi)$, it suffices to show that the map factors through the analytic de Rham complex of $R/K$, and induces a termwise isomorphism with the target.
	
	Let us assume $R$ admits a surjection from $P=\Bdr \langle T_i\rangle$, together with an ideal $I\in P$ satisfying \Cref{Assump}.
	By \Cref{prism dR}, prismatic cohomology is isomorphic to $\xi$-divided de Rham cohomology
	\[
	C_{P}:=\left( D_\Prism \ra D_\Prism \underset{P}{\motimes}\frac{\Omega_{P/\Bdr}^1}{\xi} \ra \cdots \ra D_\Prism \underset{P}{\motimes} \frac{\Omega_{P/\Bdr}^d}{\xi^d} \right),
	\]
	where each differential map is the continuous differential with respect to $(\xi,p)$-adic topology.
	Denote its termwise reduction by $\ol{C}_{P}$, which by \Cref{prism coh} and \Cref{Simp gr} is isomorphic to 
	\[
	\moplus_{i\in \mathbb{N}} L\wedge^i\mathbb{L}_{R/K}(-i)[-i].
	\]
	As $R$ is smooth over $K$, each $L\wedge^i\mathbb{L}_{R/K}$ is isomorphic to $i$-th analytic differential $\Omega_{R/K}^i$.
	So by applying the Bockstein operator $\beta_\xi$ and its associated complex, we have
	\[
	(\mathrm{H}^\bullet(\ol{C}_{P})(\bullet), \beta_\xi) = (\Omega_{R/K}^\bullet,d),
	\]
	where $d$ is the $p$-adic continuous differential.
	%This in particular shows that the kernel of the analytic quotient $\Omega_{R/K}^{alg, i} \ra \Omega_{R/K}^i$ is killed under the map $\Omega_{R/K}^{alg, i} \ra 	\mathrm{H}^i(R/\Bdr_\Prism, \ol{\mathcal{O}}_\Prism)(i)$.
	Thus the universal map from algebraic de Rham complex $(\Omega_{R/K}^{\bullet,alg}, d)$ to $(\mathrm{H}^\bullet(\ol{C}_{P})(\bullet), \beta_\xi)$ identifies its analytic quotient with the target.
	At last, notice that since the category of regular surjections $P\ra R$ is sifted as in \Cref{sift} (c.f. \cite[Lemma 5.5.4]{Guo21}), the isomorphism above is independent of the choice of $P\ra R$.
	So we are done.
	
\end{proof}

\subsection{Simplicial resolution for affinoid algebras}\label{subset simplicial res}
To extend the Hodge-Tate filtration to non-smooth rigid spaces, we will need the simplicial resolution of affinoid algebras by smooth ones. 
Here we mention that different from the algebraic setting, we cannot use the standard polynomial resolution where variables are infinite.
To resolve topologically finite type algebras over a $p$-adic field, we will need a version of simplicial resolution such that each term  satisfies the finite generatedness condition, thus within the category of topologically finite type affinoid algebras.

We start by recalling the following general construction resolving a given class of a given simplicial algebra using finite type polynomials.
As a convention, for a simplicial abelian group $A=A_\bullet$ and a co-degeneracy map $s^j:[n+1]\ra [n]$, we use $s_j$ to denote the corresponding map $A_n \ra A_{n+1}$ for $A$.
Similarly for (co-) face maps.
\begin{construction}\label{Iye const}
	Let $R$ be a commutative ring, and $A$ be a simplicial $R$-algebra.
	Let $d\in \NN$ be an integer, and $\omega\in A_{d-1}$ be a cycle of degree $d-1$.
	
	For each $n\in \NN$, let $X_n$ be the finite set
	\[
	X_n:=\{x_t|t:[n]\ra [d]\text{ is surjective and monotone}\}.
	\]
	Then by [\cite{Iy07}, 4.10], there is a simplicial $A$-algebra structure on $B=A[\{x\}|\partial x=\omega]$. defined as the free simplicial $A$-algebra such that $B_n=A_n[X_n]$.
	
	The lemma below describes the properties we need for this resolution.
\begin{lemma}\label{Iye prop}
	Let $R$ be a ring, and $A$ be a simplicial $R$-algebra.
	Then the simplicial $A$-algebra $B=A[\{x\}|\partial x=\omega]$ satisfies the following:
	\begin{enumerate}[(i)]
		\item For for $n\leq d-1$, we have $A_n=B_n$.
		\item For each $n\in \NN$, $B_n$ is a polynomial of finite variables over $A_n$.
		\item The natural map $A\ra B$ induces a short exact sequence
		\[
		0\rra A_{d-1}\cdot \omega\rra \pi_{d-1}(A)\rra \pi_{d-1}(B)\rra 0.
		\]
	\end{enumerate}
\end{lemma}
In particular the cycle $\omega\in \pi_{d-1}(A)$ is killed in $B$.
\end{construction}

Now we apply the above construction to the $p$-adic setting, to show the existence of topologically finite type resolutions for an affinoid algebra.
\begin{theorem}\label{res of pair}
	 Let $(B,I)$ be a topologically finitely presented algebra over $\Ainfe$ with a finitely generated ideal, such that $B/I$ is $p$-torsion free.
	Then there exists a quasi-isomorphism from simplicial pairs of $\Ainfe$-algebras with ideals,
	\[
	(\Ainfe\langle X_\bullet,Y_\bullet\rangle, (Y_\bullet))\rra (B,I),
	\]
	where both $X_i$ and $Y_i$ are finite for each $i\in \NN$.
%	\item Let $(A_e,I_e)$ for $e\in \mathbb{N}$ be a compatible system of flat-$\Ainfe$ pairs as in (i), with the transition isomorphisms $(A_{e+1}, I_{e+1})\otimes_{\mathrm{A_{inf, e+1}}} \Ainfe \cong (A_e, I_e)$.
%	Then there exists a system of quasi-isomorphisms from simplicial pairs 
%	\[
%	(\Ainfe\langle X_\bullet,Y_\bullet\rangle, (Y_\bullet))\rra (A_e,I_e),
%	\]
%	such that both $X_i$ and $Y_i$ are finite for each $i\in \mathbb{N}$, and moreover
%	\[
%	(\mathrm{A_{inf, e+1}}\langle X_i,Y_i\rangle, (Y_i)) \otimes_{\mathrm{A_{inf, e+1}}} \Ainfe \cong (\Ainfe\langle X_i,Y_i\rangle, (Y_i)).
%	\]
\end{theorem}
\begin{proof}
	We proceed the construction by induction as below.
	\begin{itemize}
		\item[Step 1] 
		By assumption, $B$ is topologically finitely presented and $I$ is finitely generated.
		We pick two finite sets $X^{(0)}$ and $Y^{(0)}$ that corresponds to the generators of $B$ (over $\Ainfe$ as a topological algebra) and $I$ (over $B$ as a module) separately, with surjections $(f^{(0)},g^{(0)}):\Ainfe\langle X^{(0)}, Y^{(0)}\rangle\ra B$ and $(Y^{(0)})\ra I$.
		We denote $(\Ainfe\langle X^{(0)}_\bullet, Y^{(0)}_\bullet\rangle, (Y^{(0)}_\bullet))$  to be the associated constant simplicial algebra with the ideal.
		%Moreover, in the setup of (ii), as the element $\xi^e$ is nilpotent in $\mathrm{A_{inf, e+1}} \langle X^{(0)}, Y^{(0)} \rangle$, by Nakayama's lemma we can choose a compatible system of surjections onto the pair $(A_{e+1}, I_{e+1})$ lifting those over $\Ainfe$.		
		\item[Step 2]
		We now construct the simplcial pair $(\Ainfe\langle X^{(1)}_\bullet,Y^{(1)}_\bullet\rangle, (Y^{(1)}_\bullet))$ together with a map 
		\[
		f^{(1)}:(\Ainfe\langle X^{(0)}_\bullet,Y^{(0)}_\bullet\rangle, (Y^{(0)}_\bullet))\rra (\Ainfe\langle X^{(1)}_\bullet,Y^{(1)}_\bullet\rangle, (Y^{(1)}_\bullet)),
		\]
		such that the image of $(\ker(f^{(0)}),\ker(g^{(0)}))$ is killed.
		
		Both $\ker(f^{(0)})$ and $\ker(g^{(0)})$ are finitely generated $\Ainfe\langle X^{(0)}_\bullet,Y^{(0)}_\bullet \rangle$-modules (this follows from \cite[Corollary 5.1.3 (iii)]{Guo21}).
		So we can pick finite generators $u_i$ of $\ker(f^{(0)})$ and $v_j$ of $\ker(g^{(0)})$ separately.
		Using \Cref{Iye const} we form a free simplicial algebra 
		\[
		\Ainfe\langle X^{(0)}_\bullet,Y^{(0)}_\bullet\rangle [\partial U_i=u_i][\partial V_j=v_j]
		\] 
		over $\Ainfe\langle X^{(1)}_\bullet,Y^{(1)}_\bullet \rangle$, such that the image of the following two maps vanish (Lemma \ref{Iye prop} (iii)):
		\begin{align*}
			\ker(f^{(0)})&\rra \pi_0(\Ainfe\langle X^{(1)}_\bullet,Y^{(1)}_\bullet\rangle [\partial U_i=u_i][\partial V_j=v_j]),\\
			\ker(g^{(0)})&\rra \pi_0((Y^{(0)}_\bullet,V_\bullet)).
		\end{align*}
		We take the term-wise $p$-adic completion at the simplicial ring $\Ainfe\langle X^{(0)}_\bullet,Y^{(0)}_\bullet\rangle [\partial U_i=u_i][\partial V_j=v_j]$.
		Then by Lemma \ref{Iye prop} (ii), we get a new simplicial pair over $\Ainfe$: 
		\[
		(\Ainfe\langle X^{(1)}_\bullet,Y^{(1)}_\bullet\rangle, (Y^{(1)})),
		\]
		where each $X^{(1)}_n$ (resp. $Y^{(1)}_n$) is the finite set given by  the union of $X^{(0)}_n$ (resp. $Y^{(0)}_n$) with the newly added variables from $[\partial U_i=u_i]$ (resp. $[\partial V_j=v_j]$), as in \Cref{Iye const}.
		Moreover, the image of the pair of ideals $(\ker(f^{(0)}),\ker(g^{(0)})$ in the $0$-th fundamental group vanishes, by the composition below
		\[
		\xymatrix{
			(\ker(f^{(0)}),\ker(g^{(0)})) \ar[r] \ar[rd]&(\Ainfe\langle X^{(0)}_\bullet,Y^{(0)}_\bullet\rangle [\partial U_i=u_i][\partial V_j=v_j], (Y^{(0)}\cup \wt V_\bullet)) \ar[d]\\
			&(\Ainfe\langle X^{(1)}_\bullet,Y^{(1)}_\bullet\rangle, (Y^{(1)})).}
		\]
		(Here we temporarily use $\wt V_\bullet$ to denote the newly added variables in this step).
		Note that by Lemma \ref{Iye prop} (iii) the zero-th fundamental group of $(\Ainfe\langle X^{(1)}_\bullet,Y^{(1)}_\bullet\rangle, (Y^{(1)}))$ is exactly isomorphic to $(A_e,I_e)$.
			
		\item[Step 3] 
		As fundamental groups $\pi_i$ of $(B,I)$ are trivial for $i\geq 1$, we use the above agreement to kill the cycles of $(\Ainfe\langle X^{(1)}_\bullet,Y^{(1)}_\bullet\rangle, (Y^{(1)}))$ inductively, starting from $\pi_1$.
		We first make the following observation about the finiteness: %, we need the following observation about the finiteness condition:
		\begin{claim}\label{claim for finiteness}
			Let $J=J_\bullet$ be a simplicial finitely generated ideal of a simplicial topologically finitely presented $\mathrm{A_{inf, e}}$-algebra $A_\bullet$.
			Then for each $n\in \NN$, the $n$-th fundamental group $\pi_n(J_\bullet)$ is a finitely generated $A_n$-module.
		\end{claim}
		Granting the Claim, we can apply it onto the unit ideal and the ideal $(Y^{(n)}_\bullet)$ in the simplicial ring $\Ainfe\langle X^{(n)}_\bullet,Y^{(n)}_\bullet\rangle $.
		Then exactly as in Step 2, by using \Cref{Iye const} and taking the term-wise $p$-adic completion, we can produce the simplicial pair $(\Ainfe\langle X^{(n+1)}_\bullet,Y^{(n+1)}_\bullet\rangle, (Y^{(n+1)}_\bullet))$, 
		so that $\pi_n$ of the simplicial pair $\Ainfe\langle X^{(n)}_\bullet,Y^{(n)}_\bullet\rangle $ is killed.
		By Lemma \ref{Iye prop} (ii), this will preserve the finite sets $X_i$ and $Y_i$ for $i\leq n$. %, in particular the finiteness will not 
		So we can take the colimit to get a simplicial finitely presented $\Ainfe$-pair
		\[
		(\Ainfe\langle X_\bullet,Y_\bullet\rangle, (Y_\bullet)):=\underset{n\in \NN}{\colim} (\Ainfe\langle X^{(n)}_\bullet,Y^{(n)}_\bullet\rangle, (Y^{(n)}_\bullet)).
		\]
		Note that by the construction, as $\pi_*(\Ainfe\langle X_\bullet,Y_\bullet\rangle, (Y_\bullet))=\pi_0(\Ainfe\langle X_\bullet,Y_\bullet\rangle, (Y_\bullet))=(A,I)$, the natural map from $	(\Ainfe\langle X_\bullet,Y_\bullet\rangle, (Y_\bullet))$ to its $0$-th fundamental group is a surjective quasi-isomorphism.
		This finishes the construction we need.
		\begin{proof}[Proof of the Claim]
			As in \cite[3.4]{Iy07}, to compute the $i$-th fundamental group, we can take the normalization of $J_\bullet$ by
			\[
			N(J)_n=\cap_{i=1}^n \ker(d_i).
			\]
			Then $\pi_n(J)$ is a quotient of the kernel ideal for the $A_n$-linear map 
			\[
			\ker(d_0|_{N(J)_n}:N(J)_n\ra N(J)_{n-1}).
			\]
			We note that the group $N(J)_n$ can be written as the kernel of the $A_n$-linear map 
			\[
			\xymatrix{J_n\ar[rr]^{\oplus_{i=1}^n d_i~~~~~~~~~~~~~~~~~~} &&\bigoplus_{i=1}^n J_{n-1},}
			\]
			where the $A_n$-linear structure on each component of the right side is defined through the ring homomorphism $d_i:A_n\ra A_{n-1}$.
			But since $J_n$ is finitely generated, while $J_{n-1}$ is $p$-torsion free, by \cite[Tag 0519]{Sta} and \cite[Corollary 5.1.3]{Guo21}, the kernel of the map above is finitely generated over $A_n$.
			This implies that the $p$-torsion free $A_n$-module $N(J)_n$ is finitely presented over $A_n$.
			In this way, the kernel of the $A_n$-linear map
			\[
			d_0:N(J)_n\rra N(J)_{n-1},
			\] 
			with the source being finitely presented and the target being $p$-torsion free, is finitely generated.
			This finishes the proof showing $\pi_n(J_\bullet)$ is finitely generated over $A_n$, and we are done.

		\end{proof}
		
	\end{itemize}
\end{proof}

Moreover, the above allows us to establish a rational version as well.
\begin{proposition}\label{res of pair rational}
	Let $(B,I)$ be a topologically finitely presented algebra over $\Bdre$ with a finitely generated ideal.
	Then there exists a quasi-isomorphism from a simplicial pair of $\Bdre$-algebras with ideals
	\[
	(\Bdre\langle X_\bullet,Y_\bullet\rangle, (Y_\bullet))\rra (B,I),
	\]
	where both $X_i$ and $Y_i$ are finite for each $i\in \NN$.
	Moreover, the same holds if we replace $\Bdre$ by $\Bdr$.
\end{proposition}
\begin{proof}
	The case for $\Bdre$ follows from \Cref{res of pair} by choosing an integral model of the pair over $\Ainfe$, since $\Bdre=\Ainfe[\frac{1}{p}]$.
	To get the result for $\Bdr$-pair $(B,I)$, we need to adjust the proof of \Cref{res of pair}, changing its base from $\Ainfe$ to $\Bdr$.
	The Step 1 in \Cref{res of pair} is identical.
	For Step 2, we need to find a map 
	\[
	f^{(1)}:(\Bdr \langle X_\bullet^{(0)}, Y_\bullet^{(0)}\rangle, (Y_\bullet^{(0)})) \rra (\Bdr\langle X^{(1)}_\bullet,Y^{(1)}_\bullet\rangle, (Y^{(1)}_\bullet)),
	\]
	such that the image of $(\ker(f^{(0)}),\ker(g^{(0)}))$ vanishes.
    Let us choose a finite set of generators $\{h_l\}$ of the kernel ideals so that their mod $\xi$ reduction are power bounded (i.e. inside of $\mathcal{O}_K \langle X_\bullet^{(0)}, Y_\bullet^{(0)} \rangle$).
	Then imitating Step 2 in \Cref{res of pair}, we may choose a compatible set of maps
	\[
	(\Bdre\langle X^{(0)}_\bullet,Y^{(0)}_\bullet\rangle, (Y_\bullet^{(0)}))\rra (\Bdre\langle X^{(1)}_\bullet,Y^{(1)}_\bullet\rangle, (Y_\bullet^{(1)})),
	\]
	such that the image of the mod $\xi^e$ reduction of $\{h_l\}$ vanishes.
	\begin{comment}
	It worth a comment here: to see this, one needs to choose a compatible system of integral rings $A_e\subset \Bdre \langle X^{(0)}_\bullet,Y^{(0)}_\bullet\rangle$, so that each of $A_e$ contains the mod \xi^e reduction of $\{h_l\}$. Then one forms the integral simplicial pairs over $\Ainfe$ kills them, where X_\bullet^{(1)} and Y_\bullet^{(1)} are fixed (parametrizing \{h_l\}.
	\end{comment}
	Thus by taking the inverse limit with respect to $e$, we get Step 2.
	
	At last, to proceed Step 3 of \Cref{res of pair}, we notice that the finiteness in \Cref{claim for finiteness} holds automatically for $\Bdr$-pairs $(\Bdr\langle X^{(n)}_\bullet,Y^{(n)}_\bullet\rangle, (Y^{(n)}_\bullet))$ over $\Bdr$, as each $\Bdr\langle X^{(n)}_i,Y^{(n)}_i\rangle$ is noetherian.
	In this way, similar to the above modification on Step 2, we can mod powers of $\xi$ and kills the reduction of a fixed finite set of generators for $\pi_n(\Bdr\langle X^{(n)}_\bullet,Y^{(n)}_\bullet\rangle, (Y^{(n)}_\bullet))$.
	A further inverse limit of this compatible system of simplicial pairs finishes the proof.
\end{proof}
In the special case when $I$ is the zero ideal, we get the resolution of the algebra itself.
\begin{corollary}\label{res of alg}
	Assume $R$ is either a $p$-torsionfree topologically finitely presented $\Ainfe/\Bdre/\Bdr$-algebra.
	There is a simplicial $\Ainfe/\Bdre/\Bdr$-algebra resolution $A_\bullet \ra R$, such that each $A_n$ is isomorphic to a Tate algebra of finite variables over $\Ainfe/\Bdre/\Bdr$.
\end{corollary}

\subsection{Left Kan extension and Hodge-Tate filtration}
Now we are able to extend the Hodge-Tate filtration to general l.c.i rigid spaces, using simplicial resolution developed last subsection.
The idea is similar to the left Kan extension used in the algebraic setting, but we do it over affinoid algebras.

We start by observing that derived de Rham complex can be computed using simplicial resolution.
\begin{proposition}\label{left kan of ddR}
	Let $R$ be a topologically finite type $\Bdre$-algebra.
	Then for a given simplicial resolution $\Bdre \langle X_\bullet\rangle \ra R$ where each set of variables $X_n$ is finite, we have
	\begin{align*}
		&\underset{\Delta}{\colim} ~\mathbb{L}^\an_{\Bdre \langle X_\bullet\rangle/\Bdre} \cong \mathbb{L}^\an_{R/\Bdre}, \\
		&\underset{\Delta}{\colim} ~\dR_{\Bdre \langle X_\bullet\rangle/\Bdre}/\Fil^m \cong \dR_{R/\Bdre}/\Fil^m, \\
		&\text{filtered completion of }\left( \underset{\Delta}{\colim}~ \dR_{\Bdre \langle X_\bullet\rangle/\Bdre}\right) \cong \dR_{R/\Bdre}.
	\end{align*}
\end{proposition}
\begin{proof}
	We temporarily use the notation $(-)^\wedge_p$ to denote the derived $p$-completion.
	Let $R_0$ be a ring of definition of $R$, which is topologically finite type over $\Ainfe$.
	By induction, we can find a simplicial topologically finite type $\Ainfe$-algebras $A_\bullet \ra R_0$, such that each $A_n$ is a ring of definition of $\Bdre \langle X_n\rangle$.
	The triple of rings $\Ainfe \ra A_\bullet \ra R_0$ then induces the natural triangle of algebraic cotangent complexes
	\[
	\mathbb{L}_{A_\bullet/\Ainfe}\otimes_{A_\bullet}^L R_0  \rra \mathbb{L}_{R_0/\Ainfe} \rra \mathbb{L}_{R_0/A_\bullet},
	\]
	where the final term is equal to the homotopy colimit $\underset{[n]\in \Delta}{\colim}~ \mathbb{L}_{R_0/A_n}$.
	We then apply the derived $p$-completion at the triangle and then invert at $p$.
	By \cite[Tag 091V]{Sta}, the derived completion commutes with colimits.
	In particular we get the equalities
	\begin{align*}
		(\mathbb{L}_{R_0/A_\bullet})^\wedge_p[\frac{1}{p}] &= (\underset{[n]\in \Delta}{\colim} ~\mathbb{L}_{R_0/A_n})^\wedge_p [\frac{1}{p}] \\
		&\cong \underset{[n]\in \Delta}{\colim}~ (\mathbb{L}_{R_0/A_n})^\wedge_p [\frac{1}{p}] \\
		&= \underset{[n]\in \Delta}{\colim}~ \mathbb{L}^\an_{R/\Bdre \langle X_n\rangle}.
	\end{align*}
Note that by assumption the map $\Bdre \langle X_n\rangle \ra R$ is a surjection.
In particular, by \cite[Corollary 5.2.15]{Guo21} we have a natural isomorphism
\[
\mathbb{L}^\an_{R/\Bdre \langle X_n\rangle} \cong \mathbb{L}_{R/\Bdre \langle X_n\rangle},
\]
where the latter is the algebraic cotangent complex.
So by taking the colimit, the equalities above together with the assumption that $\Bdre \langle X_\bullet\rangle$ resolves the ring $R$, imply the following
\[
(\mathbb{L}_{R_0/A_\bullet})^\wedge_p[\frac{1}{p}] \cong \underset{[n]\in \Delta}{\colim}~\mathbb{L}_{R/\Bdre \langle X_n\rangle} \cong 0.
\]
Thus the triangle at the begining leads to the isomorphism
\[
\left(\mathbb{L}_{A_\bullet/\Ainfe}\otimes_{A_\bullet}^L R_0\right)^\wedge_p[\frac{1}{p}]  \cong  (\mathbb{L}_{R_0/\Ainfe})^\wedge_p[\frac{1}{p}],
\]
namely the first formula in the statement.

	For the equality on $\dR/\Fil^m$ and $\dR$, it suffices to notice that each graded piece is given by $L\wedge^i\mathbb{L}^\an[-i]$.
	Thus the last two isomorphisms follows from the first one by applying derived wedge powers.
\end{proof}
Following the above result on analytic cotangent complexes,  we can use the left Kan extension to give another construction of prismatic cohomology as below.
\begin{proposition}\label{left Kan of smooth}
	Let $R_0$ be a topologically finite type algebra over $\mathcal{O}_K$, and let $\epsilon:\mathcal{O}_K \langle X_\bullet\rangle \ra R_0$ be a simplicial resolution such that each set of variables $X_n$ is finite.
	\begin{enumerate}[(i)]
		\item The following colimits are independent of the choice of the resolution
		\[
		\left(\underset{\Delta}{\colim} ~R\Gamma(K \langle X_\bullet\rangle/\Bdr_\Prism, \mathcal{O}_\Prism)\right)^\wedge,~~~~~~~~~~~~\underset{\Delta}{\colim} ~R\Gamma(K\langle X_\bullet\rangle/\Bdr_\Prism, \ol{\mathcal{O}}_\Prism),
		\]
		where $(-)^\wedge$ is the derived $\xi$-completion.
		\item When $R=R_0[\frac{1}{p}]$ is smooth over $K$, the above are naturally isomorphic to $R\Gamma(R/\Bdr_\Prism, \mathcal{O}_\Prism)$ and $R\Gamma(R/\Bdr_\Prism, \ol{\mathcal{O}}_\Prism)$ separately.
	\end{enumerate} 
\end{proposition}
\begin{proof}
	Consider the category $\mathcal{C}$ of $\mathcal{O}_K$-linear maps $(\mathcal{O_K} \langle E \rangle, \beta: \mathcal{O_K} \langle E \rangle \ra R_0)$ for finite sets $E$, and equip its opposite category with the indiscrete topology.
	Then it suffices to show that any such resolution  $\mathcal{O}_K \langle X_\bullet\rangle \ra R_0$ is equivalent to the final object of the associated topos.
	This is shown as in \cite[08PS]{Sta}, with the only difference being that we use the $p$-completeness of $\mathcal{O}_K \langle X_\bullet\rangle$ to get the equality
	\[
	\mathrm{Mor}_C((\mathcal{O}_K \langle E\rangle, \beta), (\mathcal{O}_K \langle X_\bullet\rangle, \epsilon)) = \mathrm{Mor}_\mathrm{Sets}((E, \beta|_E), (\mathcal{O}_K \langle X_\bullet\rangle, \epsilon)).
	\]
	
	Now we consider the presheaf of complexes $L\ol \Prism_{R/\Bdr}$ over $\mathcal{C}$ sending a given  map $\mathcal{O_K} \langle E \rangle \ra R_0$ onto reduced prismatic cohomology $R\Gamma(K \langle E\rangle/\Bdr, \ol{\mathcal{O}}_\Prism)$ together with its Hodge-Tate filtration as in \Cref{HT fil}.
	Then by the above argument,  the global section of $\ol \Prism_{R/\Bdr}$ is isomorphic to 
	\[
	\underset{\Delta}{\colim} ~R\Gamma(K\langle X_\bullet\rangle/\Bdr_\Prism, \ol{\mathcal{O}}_\Prism). \tag{$\ast$}
	\]
	In particular, it is independent of the choice of the simplicial resolution $\mathcal{O}_K \langle X_\bullet\rangle \ra R_0$.
	Moreover, by \Cref{HT fil} the $i$-th graded piece of the Hodge-Tate filtration is equal to the following colimit
	\[
	\underset{\Delta}{\colim} ~\Omega^i_{K\langle X_\bullet\rangle/K}(-1)[-1],
	\]
	which by \Cref{left kan of ddR} is naturally isomorphic to $L\wedge^i\mathbb{L}^\an_{R/K}(-1)[-1]$.
	In particular, when $R$ is smooth over $K$, as the analytic cotangent complex coincides with its continuous differential over $K$, the graded pieces of the colimit of the Hodge-Tate filtrations in $(\ast)$ is isomorphic to the Hodge-Tate filtration of $R\Gamma(R/\Bdr_\Prism, \ol{\mathcal{O}}_\Prism)$.
	So by taking colimits with respect to the exhaustive increasing filtrations of the both, we finish the proof of (ii) for reduced prismatic cohomology.
	At last, the extension of the above to prismatic cohomology of $\mathcal{O}_\Prism$ follows from the derived Nakayama's lemma by the reduction mod $\xi$.
\end{proof}
\begin{remark}
	Here we notice that \Cref{left Kan of smooth} start with a simplicial resolution in the integral level, which is a priori different from the left Kan extension but in the rational level.
	In fact, using the Hodge-Tate resolution as in the second half of the proof and \Cref{left kan of ddR}, one can show that the analogous statement of \Cref{left Kan of smooth} holds true for a simplicial resolution $K \langle X_\bullet\rangle \ra R$ over $K$, without assuming the resolution are defined integrally.
\end{remark}
\begin{definition}\label{left Kan HT}
	Let $R$ be a topologically finite type algebra over $K$, and let $K \langle X_\bullet\rangle \ra R$ be any simplicial resolution such that each $X_n$ is finite.
	\begin{enumerate}[(i)]
		\item We define \emph{derived prismatic cohomology} $L\Prism_{R/\Bdr}$ and its reduction $L\ol\Prism_{R/\Bdr}$ as below
		\begin{align*}
			&L\Prism_{R/\Bdr}:=\left(\underset{\Delta}{\colim} ~R\Gamma(K \langle X_\bullet\rangle/\Bdr_\Prism, \mathcal{O}_\Prism)\right)^\wedge; \\
			&L\ol\Prism_{R/\Bdr}:= underset{\Delta}{\colim} ~R\Gamma(K\langle X_\bullet\rangle/\Bdr_\Prism, \ol{\mathcal{O}}_\Prism),
		\end{align*}
	where $(-)^\wedge$ is the derived $\xi$-completion.
	\item The $i$-th Hodge-Tate filtration on $L\ol\Prism_{R/\Bdr}$ is defined as 
	\[
	\underset{\Delta}{\colim} ~\Fil_i^\mathrm{HT} R\Gamma(K\langle X_\bullet\rangle/\Bdr_\Prism, \ol{\mathcal{O}}_\Prism).
	\]
	\end{enumerate}
\end{definition}
Using the above notations, \Cref{left Kan of smooth} shows that for a smooth rigid space $X$ over $K$, its derived prismatic cohomology is isomorphic to prismatic cohomology $R\Gamma(X/\Bdr_\Prism, \mathcal{O}_\Prism)$, and its Hodge-Tate-filtered reduced derived prismatic cohomology is isomorphic to reduced prismatic cohomology $R\Gamma(X/\Bdr_\Prism, \ol{\mathcal{O}}_\Prism)$ with its Postnikov filtration.

Finally, we are able to show that derived prismatic cohomology coincides with originally defined prismatic cohomology, for general l.c.i rigid spaces.
\begin{theorem}\label{left kan of coh}
	Let $R$ be a topologically finite type algebra over $K$ that has l.c.i singularities.
	Then there are natural isomorphisms
	\begin{align*}
		L\Prism_{R/\Bdr} &\cong R\Gamma(R/\Bdr_\Prism, \mathcal{O}_\Prism);\\
		L\ol\Prism_{R/\Bdr} & \cong R\Gamma(R/\Bdr_\Prism, \mathcal{O}_\Prism).
	\end{align*}
\end{theorem}
\begin{proof}
	By derived Nakayama's lemma, it suffices to show the second isomorphism.
	As both reduced cohomology are \'etale local, we may assume $R$ admits a regular closed immersion as in \Cref{Assump}.
	Namely, we assume that there are surjections $P=\Bdr \langle T_i\rangle \ra P/I \ra P/(I,\xi)\cong R$, where $I$ is a Koszul-regular ideal of $P$ with $P/I$ being flat over $\Bdr$.
	We then use \Cref{res of pair rational} to find a simplicial resolution of pairs 
	\[
	(\Bdr \langle X_\bullet, Y_\bullet\rangle, (Y_\bullet)) \rra (P,I),
	\]
	such that each of $X_n$ and $Y_n$ is a finite set.
	Note that the induced simplicial $K$-algebra $K \langle X_\bullet\rangle$ resolves $P/\xi=R$ under the quotient map.

	To compare two prismatic cohomology, we apply \Cref{left Kan of smooth} at the $K$-resolution $K \langle X_\bullet\rangle \ra R$ to represent derived prismatic cohomology.
	It then suffices to show the following isomorphism
	\[
	\underset{\Delta}{\colim}~ R\Gamma(K\langle X_\bullet\rangle/\Bdr_\Prism, \ol{\mathcal{O}}_\Prism) \rra R\Gamma(R/\Bdr_\Prism, \ol{\mathcal{O}}_\Prism).
	\]
	Notice that the pairs $(\Bdr \langle X_n, Y_n\rangle, (Y_n))$ and $(P,I)$ all satisfy \Cref{Assump}.
	In particular, we can apply Simpson's functor as in \Cref{prism coh} (and in particular \Cref{lift coh}) at the pair $(\Bdr \langle X_n, Y_n\rangle, (Y_n))$ to compute $R\Gamma(K \langle X_n\rangle/\Bdr_\Prism, \mathcal{O}_\Prism)$, which is functorial with respect to $[n]\in \Delta$.
	In particular, by taking the colimit at their reduced prismatic cohomology, we get the following graded morphism
	\begin{align*}
		\underset{\Delta}{\colim} \moplus_i \Omega_{K \langle X_\bullet\rangle/K}^i(-i)[-i] \rra \moplus_i L\wedge^i\mathbb{L}^\an_{R/K}(-i)[-i].
	\end{align*}
In this way, thanks to the simplicial descent for the analytic cotangent complex in \Cref{left kan of ddR}, we get the isomorphism for the graded pieces, thus the one for reduced prismatic cohomology.

\end{proof}
As a consequence, we obtain a naturally defined Hodge-Tate filtration on reduced prismatic cohomology.
\begin{corollary}\label{HT fil}
	Let $X$ be a rigid space over $K$ that is a local complete intersection.
	There exists an ascending exhaustive $\mathbb{N}^\op$-filtration $\Fil^\mathrm{HT}_\bullet$ on $R\Gamma(X/\Bdr_\Prism,\ol{\mathcal{O}}_\Prism)$, and is multiplicative under the canonical $\mathbb{E}_\infty$-ring structure.
	It satisfies the following:
	\begin{enumerate}[(i)]
		\item the $i$-th graded piece is naturally isomorphic to $R\Gamma(X, L\wedge^i\mathbb{L}^\an_{R/K})(-i)[-i]$;
		\item when $X$ is smooth and affinoid, it coincides with the Postnikov filtration as in \Cref{HT comp}.
	\end{enumerate} 
\end{corollary}
Note that by assumption, the $i$-th graded piece of the Hodge-Tate filtration is also isomorphic to $\gr^i R\Gamma(X/K_{\mathrm{inf}}, \ol{\mathcal{O}}_\Prism)(-i)$ for the infnitesimal filtration on infinitesimal cohomology,
which follows from the comparison between the infinitesimal cohomology and the derived de Rham cohomology as in \cite{Guo21}.

\section{Infinitesimal comparison}\label{sec inf compa}
In this section, we prove the infinitesimal comparison theorem, using the decalage functor $L_\eta$.

\begin{theorem}\label{decalage comp}
	Let $X$ be an affinoid rigid space over $K$ that has l.c.i. singularities.
	There is a natural map of complexes as below
	\[\xymatrix{R\Gamma(X/\Bdr_\mathrm{pinf}, \mathcal{O}_{X/\Bdr}) \ar[r]^\phi & L\eta_\xi R\Gamma(X/\Bdr_\Prism, \mathcal{O}_\Prism) \ar[r]^\psi& R\Gamma(X/\Bdr_\Prism, \mathcal{O}_\Prism).}\]
	The map $\phi$ is an isomorphism when $X$ is smooth and affinoid.
\end{theorem}
To simplify the notation, we use $R\Gamma_\mathrm{inf}(X/\Bdr)$ and $R\Gamma_\Prism(X/\Bdr)$ to denote infinitesimal and prismatic cohomology.
\begin{proof}
	To start, we first note that map from infinitesimal cohomology to prismatic cohomology can be constructed without the l.c.i assumption on $X$.
	Recall there is a natural morphism of ringed sites in \Cref{ass prism},
	\[
	(X/\Bdr_\Prism, \mathcal{O}_\Prism) \rra (X/\Bdr_{\mathrm{inf}}, \mathcal{O}_{X/\Bdr}),
	\]
	sending a pro-infinitesimal thickening $(U=\Spa(A),T_e=\Spa(B_e))$ onto the prism
	\[
	(B'=B[\frac{I}{\xi}]^\wedge_\tf  \ra B'/\xi \leftarrow R),
	\]
	where $I$ is the kernel of the surjection $\mathcal{O}_{X/\Bdr}(U,T_e)=B:=\varprojlim B_e \ra A$.
	Notice that there is a natural map of $\Bdr$-algebras $B\ra B'$ by construction.
	Thus by taking the homotopy limit among all pro-infinitesimal thickenings in $X/\Bdr_{\mathrm{inf}}$, we get a map of complexes
	\[
	\psi:R\Gamma_\mathrm{inf}(X/\Bdr) \rra R\Gamma_\Prism(X/\Bdr).
	\]
	
	On the other hand, the decalage functor naturally induces a map $L\eta_\xi R\Gamma_\Prism(X/\Bdr) \ra R\Gamma_\Prism(X/\Bdr)$.
	To see this, by \cite[Lemma 6.10]{BMS}, it suffices to show that under the assumption $\mathrm{H}^0_\Prism(X/\Bdr)$ has no $\xi$-torsion.
	In this case, we consider the natural triangle induced by multiplication by $\xi$ on $\mathcal{O}_\Prism$, and get
	\[
	\xymatrix{R\Gamma_\Prism(X/\Bdr) \ar[r]^{\cdot \xi} & R\Gamma_\Prism(X/\Bdr) \ar[r] & R\Gamma_{\ol\Prism}(X)=R\Gamma(X/\Bdr_\Prism, \ol{\mathcal{O}}_\Prism)}.
	\]
	Notice that everything above has no negative cohomology.
	So by taking the long exact sequence and looking at $\mathrm{H}^0$, we get the $\xi$-torsionfreeness of $\mathrm{H}^0_\Prism(X/\Bdr)$.
	
	To see the map $\phi$ from infinitesimal cohomology to $L\eta_\xi R\Gamma_\Prism(X/\Bdr)$, we assume $X=\Spa(R)$  admits a regular closed immersion into $P=\Bdr \langle T_i\rangle$ as in \Cref{Assump}.
	In this case, by translating both cohomology into forms involving de Rham complexes (\cite[Theorem 4.1.1]{Guo21} for infinitesimal cohomology, and \Cref{prism dR} for prismatic cohomology), the natural map $\psi$ can be rewritten as 
	\[
	\psi: \left(D_{\inf} \ra D_{\inf}\otimes_P \Omega_P^1 \ra \cdots\right) \rra \left( D_\Prism \ra D_\Prism \otimes_P \frac{\Omega_P^1}{\xi} \ra \cdots\right),
	\]
	where $D_\Prism=D_{\inf}[\frac{I}{\xi}]^\wedge_\tf$.
	Notice that the image of each $D_{\inf}\otimes_P \Omega_P^i$ in $D_\Prism\otimes_P \frac{\Omega_P^i}{\xi^i}$ is contained in the submodule $D_\Prism\otimes_P \Omega_P^i=\xi^i \cdot D_\Prism\otimes_P \frac{\Omega_P^i}{\xi^i}	$.
	In particular, by the explicit construction of the $\eta_\xi$ functor in \cite[Section 6]{BMS}, as $\psi^i(\omega^i)$ (for $\omega^i\in D_{\inf}\otimes_P \Omega_P^i$) is contained in $\xi^i \cdot D_\Prism\otimes_P \frac{\Omega_P^i}{\xi^i}$, we see
	\[
	\psi^i(D_{\inf}\otimes_P \Omega_P^i) \subset \left( \eta_\xi (D_\Prism\otimes_P \frac{\Omega_P^\bullet}{\xi^\bullet}) \right)^i.
	\]
	Thus the map from infinitesimal cohomology to prismatic cohomology naturally factors through the map $\psi$.
	%To see the map can be extended to general affinoid rigid spaces that have l.c.i singularities, it suffices to apply \cite[Proposition 6.12]{BMS}
	
	At last, let us further assume $X$ is smooth.
	By the derived $\xi$-completeness of $L\eta_\xi R\Gamma_\Prism(X/\Bdr)$ (\cite[Lemma 6.19]{BMS}) and that of infinitesimal cohomology $R\Gamma_\mathrm{inf}(X/\Bdr)$, it suffices to check the isomorphism after reducing $\phi$ mod $\xi$.
	Notice infinitesimal cohomology satisfies the formula 
	\[
	R\Gamma_\mathrm{inf}(X/\Bdr) \otimes^L_\Bdr K \cong \Omega_{R/K}^\bullet.
	\]
	On the other hand, by the Hodge-Tate comparison (\Cref{HT comp}) and \cite[Proposition 6.12]{BMS}, we have
	\[
	(L\eta_\xi R\Gamma_\Prism(X/\Bdr))\otimes^L_\Bdr K \cong \Omega_{R/K}^\bullet,
	\]
	where Bockstein operator is exactly the continuous differential operator in this form.
	So both the source and the target are abstractly isomorphic to each other.
	Finally, to see the map $\phi$  is indeed an isomorphism, it suffices to do this locally, assuming $R$ admits a smooth lift $\wt R$ over $\Bdr$, where we use the $\xi$-divided de Rham complex in \Cref{coh of envelope} to identify termwise generators.

\end{proof}
It is well-known to experts that the decalage functor does not commute with taking derived global sections (c.f. \cite[discussion above Theorem 1.17]{BMS}).
We give a simple example illustrating the discrepancy.
\begin{example}\label{eg decalage proper}
	Let $X$ be a projective curve over $K$, defined over a discretely valued subfield $k$ that has a perfect residue field.
	By \cite[Proposition 6.12]{BMS} and the Hodge-Tate comparison \Cref{HT comp}, the mod $\xi$-reduction $(L\eta_\xi C)\otimes_\Bdr^L K$ for $C=R\Gamma_\Prism(X/\Bdr)$ is isomorphic to the following explicit complex
	\[
	\left( \mathrm{H}^0(X,\mathcal{O}_X) \rra (\mathrm{H}^1(X, \mathcal{O}_X)(1) \oplus \mathrm{H}^0(X,\Omega_X^1)) \rra \mathrm{H}^1(X,\Omega_X^1)(1) \right).
	\]
	By the Galois equivariance, the complex is isomorphic to the derived sum of $\left( \mathrm{H}^0(X,\mathcal{O}_X) \ra \mathrm{H}^0(X,\Omega_X^1)\right)$ and $\left( \mathrm{H}^1(X,\mathcal{O}_X) \ra \mathrm{H}^1(X,\Omega_X^1) \right)(1)$, where each map is induced from the differential $\mathcal{O}_X\ra \Omega_X^1$, and is zero by the degeneration of the Hodge-de Rham spectral sequence.
	In particular $(L\eta_\xi C)\otimes_\Bdr^L K$ has nonzero direct summands of Hodge-Tate weight $(-1)$.
	On the other hand, the reduction $R\Gamma_{\inf}(X/\Bdr)\otimes_\Bdr K\cong R\Gamma_\mathrm{dR}(X/K)$ only has weight $0$ cohomology.
	Thus it is impossible to have a natural isomorphism between $R\Gamma_{\inf}(X/\Bdr)$ and $L\eta_\xi R\Gamma_\Prism(X/\Bdr)$ in general.
	
\end{example}
\begin{remark}
	The above example indicates that certain components of infinitesimal cohomology and decalaged-prismatic cohomology still coincide.
	In fact, for smooth rigid spaces in general, one may form a map between them, but in the other direction.
	To see this, let $X$ be a smooth rigid space and let $\{U_j\ra X\}$ be an affinoid open covering.
	Then we have
	\begin{align*}
		R\Gamma_\mathrm{inf}(X/\Bdr) \cong R\varprojlim_j R\Gamma_\mathrm{inf}(U_j/\Bdr)  \cong R\varprojlim_j L\eta_\xi R\Gamma_\Prism(U_j/\Bdr)  &
		 \leftarrow  L\eta_\xi  R\varprojlim_j R\Gamma_\Prism(U_j/\Bdr) \\
		&\cong L\eta_\xi R\Gamma(X/\Bdr),
	\end{align*}
where we use the natural isomorphism in \Cref{decalage comp} for smooth affinoid $U_j$ in the second isomorphism above.
When $X$ is smooth and proper, the above identifies the first row of Hodge-Tate spectral sequence with the weight $0$ part of the decalaged-prismatic cohomology after mod $\xi$.
\end{remark}
We also give an example illustrating the discrepancy for l.c.i rigid spaces that are not smooth.
\begin{example}\label{eg decalage aff}
	Let $X=\Spa(R)$ be an affinoid rigid space that has l.c.i singularities and is defined over a discretely valued subfield $k$.
	Assume its embedded dimension into polydiscs $K\langle T_i\rangle$ is $N$ (\cite[Appendix 5.6]{GL20}).
	By \cite[Proposition 6.12]{BMS} again, we can write $(L\eta_\xi C)\otimes_\Bdr^L K$ for $C=R\Gamma_\Prism(X/\Bdr)$ into the following direct sum of complexes
	\begin{align*}
		& \left(R \rra \Omega^1_{R/K} \rra \cdots \rra \Omega^N_{R/K}\right) \\
		&\moplus  \left(\mathrm{H}^{-1}(\mathbb{L}_{R/K}^\an) \rra \mathrm{H}^{-1}(L\wedge^2\mathbb{L}_{R/K}^\an) \rra \cdots \rra \mathrm{H}^{-1}(L\wedge^{N+1}\mathbb{L}_{R/K}^\an)\right)(-1) \\
		&\moplus \cdots \\
		& \moplus \left(\mathrm{H}^{-n}(L\wedge^n \mathbb{L}_{R/K}^\an) \rra \cdots \rra \mathrm{H}^{-n}(L\wedge^{n+N} \mathbb{L}_{R/K}^\an)\right) (-n) \\
		& \moplus \cdots.
	\end{align*}
Note that the weight zero part is the cohomology of the usual continuous de Rham complex $\Omega_{R/K}^\bullet$.
On the other hand, the mod $\xi$ reduction $R\Gamma_{\inf}(R/K)=R\Gamma_{\inf}(R/\Bdr)\otimes_\Bdr^L K$ is computed as the cohomology of the analytic derived de Rham complex $\dR_{R/K}$.
It is known that in this case the latter cohomology is a direct summand of the former one (\cite[Remark 6.2.3]{Guo21}, see also \cite[Section 5]{Bha12} for algebraic case), and different in general (see \cite[Example 4.4]{AK11} for an example of affine variety that has different cohomology).
So by choosing an affinoid open subset around the singularity for the example in loc. cit., we see those two cohomology are different in general.
\end{example}

\begin{remark}
	To globalize \Cref{decalage comp}, one can consider  a (pre)sheaf of complexes $\Prism_{X/\Bdr}$ over a general rigid space $X$, sending each affinoid open subset $\Spa(R)$ to its prismatic cohomology complex $R\Gamma_\Prism(R/\Bdr)$.
	In this case, similar to \cite{BMS}, one can show that for a smooth rigid space $X$ we have a natural isomorphism
	\[
	Ru_{X/\Bdr*} \mathcal{O}_{X/\Bdr} \rra L\eta_\xi \Prism_{X/\Bdr},
	\]
	where $Ru_{X/\Bdr *} \mathcal{O}_{X/\Bdr}$ is the infinitesimal cohomology sheaf as in \cite[Section 7]{Guo21}.
\end{remark}
\section{Pro-\'etale comparison}\label{sec proetale compa}
In this section, we prove the comparison theorem between prismatic cohomology over $\Bdr$ and pro-\'etale cohomology of $\BBdrp$, studied in \cite{Sch13} and \cite[Section 13]{BMS}.

\subsection{K\"unneth formula and \'etale localization}
As a preparation for the pro-\'etale comparison, we first prove the K\"unneth formula and the \'etale localization for reduced prismatic cohomology.

We start by observing that the analytic cotangent complex admits a natural product formula.
\begin{lemma}\label{Kunneth cotangent}
	Let $R_1$ and $R_2$ be two topologically finite type $K$-algebras, and let $p_i$ be the two natural maps from $R_i$ to $R_1\wh{\otimes}_K R_2$.
Then the following natural map is an isomorphism after an integral derived $p$-completion.
\[
Lp_1^*\mathbb{L}^\an_{R_1/K} \moplus Lp_2^*\mathbb{L}^\an_{R_2/K} \rra \mathbb{L}^\an_{R_1\wh{\otimes}_K R_2/K}.
\]
\end{lemma}
\begin{proof}
	Let $R_{1,0}$ and $R_{2,0}$ be two topologically finite type $\mathcal{O}_K$-algebras whose generic fibers are $R_1$ and $R_2$ separately.
	By \cite[Remark 7.2.45]{GR}, the analytic cotangent complex $\mathbb{L}^\an_{R/K}$ is obtained by inverting $p$ at the derived $p$-completion of the algebraic cotangent complex $\mathbb{L}_{R_0/\mathcal{O}_K}$, where $R$ is any of $R_1, R_2$ or $R_1\wh{\otimes}_K R_2$ (respectively for $R_0$).
	Denote by $\mathbb{L}^\an_{R_0/\mathcal{O}_K}$ the derived $p$-completion of the algebraic cotangent complex $\mathbb{L}_{R_0/\mathcal{O}_K}$.
	Then we notice that by \cite[Proposition 7.1.31]{GR}, each of $\mathbb{L}^\an_{R_0/\mathcal{O}_K}$ for $R_0=R_{1,0}, R_{2,0}$ or $R_{1,0}\wh{\otimes} R_{2,0}$ are pseudo-coherent.
	On the other hand, recall the natural isomorphism for the algebraic cotangent complex (\cite[Tag 09DL]{Sta}) as below
	\[
	Lp_1^*\mathbb{L}_{R_{1,0}/\mathcal{O}_K} \moplus Lp_2^*\mathbb{L}_{R_{2,0}/\mathcal{O}_K} \rra \mathbb{L}_{R_{1,0}{\otimes}_{\mathcal{O}_K} R_{2,0}/\mathcal{O}_K}.
	\]
	In this way, by applying the derived $p$-completion and using \cite[Lemma 7.1.25]{GR}, we get a natural isomorphism
	\[
	Lp_1^*\mathbb{L}^\an_{R_{1,0}/\mathcal{O}_K} \moplus Lp_2^*\mathbb{L}^\an_{R_{2,0}/\mathcal{O}_K} \rra \mathbb{L}^\an_{R_{1,0}\wh{\otimes}_{\mathcal{O}_K} R_{2,0}/\mathcal{O}_K}.
	\]
	Thus  the result follows after inverting $p$.
\end{proof}
In the following, to simplify the notation we denote $R\Gamma_{\ol{\Prism}}(R)$ to be reduced prismatic cohomology $R\Gamma(R/\Bdr_\Prism, \ol{\mathcal{O}}_\Prism)$.
\begin{theorem}\label{Kunneth reduced}
	Let $R_1$ and $R_2$ be two topologically finite type algebras over $K$ that have l.c.i singularities, and let $p_i$ be the two natural maps $R_i \ra R_1\wh{\otimes}_K R_2$.
	There is a natural isomorphism of $\mathbb{E}_\infty$-rings over $R_1\wh{\otimes}_K R_2$:
	\[
	\left( Lp_1^*R\Gamma_{\ol{\Prism}}(R_1) \right) \otimes^L_{R_1\wh{\otimes} R_2} \left( Lp_2^* R\Gamma_{\ol{\Prism}}(R_2) \right) \rra R\Gamma_{\ol{\Prism}}(R_1\wh{\otimes}_K R_2).
	\]
\end{theorem}
\begin{proof}
	Let $R$ be a topologically finite type algebra over $K$ that has l.c.i singularities.
	Recall the Hodge-Tate filtration from \Cref{HT fil} that there is an ascending exhaustive $\mathbb{N}^\op$-filtration $\Fil^{\mathrm{HT}}_\bullet$ on $R\Gamma_{\ol{\Prism}}(R)$, such that the $i$-th graded piece is equal to 
	\[
	L\wedge^i\mathbb{L}^\an_{R/K}(-i)[-i].
	\]
	Moreover, it is compatible with the $\mathbb{E}_\infty$-ring structure on $R\Gamma_{\ol{\Prism}}(R)$, and is induced from the infinitesimal filtration on the derived de Rham cohomology $\dR_{R/K}$.
	In particular, this implies that the graded algebra $\gr^{\mathrm{HT}}_\bullet R\Gamma_{\ol{\Prism}}(R)$ obtained by taking the graded pieces of the Hodge-Tate filtration on $R\Gamma_{\ol{\Prism}}(R)$ is generated by the first graded pieces $\mathbb{L}^\an_{R/K}(-1)[-1]$.
	Namely, we have
	\[
	\gr^{\mathrm{HT}}_\bullet R\Gamma_{\ol{\Prism}}(R) \cong \Sym_R (\mathbb{L}^\an_{R/K}(-1)[-1]).
	\]
	
	Back to the K\"unneth formula.
	By the functoriality and lax-monoidal structure, the maps $p_i:R_i \ra R_1\wh{\otimes} R_2$ induces a natural map of $\mathbb{E}_\infty$-algebras
	\[
	R\Gamma_{\ol{\Prism}}(R_1) \otimes_K R\Gamma_{\ol{\Prism}}(R_2) \rra R\Gamma_{\ol{\Prism}}(R_1\wh{\otimes}R_2).
	\]
	By the $R_1\wh{\otimes}R_2$-linearity of the target, we can improve the above to an $R_1\wh{\otimes}R_2$-linear homomorphism
	\[
	\left( Lp_1^*R\Gamma_{\ol{\Prism}}(R_1) \right) \otimes^L_{R_1\wh{\otimes} R_2} \left( Lp_2^* R\Gamma_{\ol{\Prism}}(R_2) \right) \rra R\Gamma_{\ol{\Prism}}(R_1\wh{\otimes}_K R_2).
	\]
	Notice that the Hodge-Tate filtration on both sides are compatible under the map.
	Moreover, by applying the Day convolution (\cite{GP18}) on the Hodge-Tate filtration of $R\Gamma_{\ol{\Prism}}(R_i)$, we can define a product filtration on the left side.
	As the colimit commutes with the tensor product, the product filtration is again ascending and exhaustive.
	The product filtration is further compatible with the Hodge-Tate filtration on $R\Gamma_{\ol{\Prism}}(R_1\wh{\otimes}_K R_2)$ under the map above.
	Thus to show the K\"unneth formula, it amounts to check the isomorphism after applying the graded pieces.
	
	Now by the observation on the Hodeg-Tate graded algebra, the induced map of graded algebras looks like
	\[
	Lp_1^* \Sym_{R_1} (\mathbb{L}^\an_{R_1/K}(-1)[-1]) \underset{R_1\wh{\otimes}R_2}{\motimes} Lp_2^* \Sym_{R_2} (\mathbb{L}^\an_{R_2/K}(-1)[-1]) \rra \Sym_{R_1\wh{\otimes}R_2} (\mathbb{L}^\an_{R_1\wh{\otimes}R_2/K}(-1)[-1]).
	\]
	Moreover, taking the pullback functors inside of the symmetric functor,  the left hand side above can be rewritten as 
	\[
	\Sym_{R_1\wh{\otimes}R_2} \left( Lp_1^*  \mathbb{L}^\an_{R_1/K}(-1)[-1] \moplus Lp_2^* \mathbb{L}^\an_{R_2/K}(-1)[-1] \right).
	\]
	Thus by the compatibility of the Hodge-Tate filtration and the infinitesimal filtration, the isomorphism follows from \Cref{Kunneth cotangent}.
	So we are done.	
\end{proof}
\begin{remark}
	The above uses crucially the pseudo-coherence of the analytic cotangent complex to free us from a $p$-adic completion ``rational''.
Using the condensed mathematics, one might ask if the above can be improved to a K\"unneth formula of non-reduced prismatic cohomology, namely an isomorphism of solid $\mathbb{E}_\infty$-algebras over $\Bdr$ as below
		\[
		R\Gamma(R_1/\Bdr, \mathcal{O}_\Prism)\otimes_\Bdr^L R\Gamma(R_2/\Bdr, \mathcal{O}_\Prism) \rra R\Gamma(R_1\wh{\otimes}_K R_2/\Bdr, \mathcal{O}_\Prism).
		\]
		Notice that the above is taking the tensor product over $\Bdr$-directly in the solid category, while a tensor product in the usual derived category would lose the topological completion.
\end{remark}

Reduced prismatic cohomology also enjoys the \'etale localization formula.
\begin{proposition}\label{etale loc}
	Let $R_1 \ra R_2$ be an \'etale morphism of topologically finite type $K$-algebras that have l.c.i singularities.
	Then the natural map of $\mathbb{E}_\infty$-algebras below is an isomorphism
	\[
	R_2\otimes_{R_1}^L R\Gamma_{\ol{\Prism}}(R_1) \rra R\Gamma_{\ol{\Prism}}(R_2).
	\]
\end{proposition}
\begin{proof}
	We first notice that the analytic cotangent complex satisfies the following natural isomorphism by \cite[Proposition 7.2.39, Theorem 7.2.42]{GR}
	\[
	R_2\otimes_{R_1}^L \mathbb{L}^\an_{R_1/K} \rra \mathbb{L}^\an_{R_2/K}.
	\]
	By taking derived wedge powers of the isomorphism and the commutativity of $L\wedge^i$ and derived tensor product $R_2\otimes_{R_1}^L (-)$, the above tensor product formula extends to all of $L\wedge^i\mathbb{L}^\an_{R/K}$.
	On the other hand, notice that the map in the statement is equivariant under the Hodge-Tate filtration on $R\Gamma_{\ol{\Prism}}(R_1)$ and $R\Gamma_{\ol{\Prism}}(R_2)$.
	Moreover, by construction of the Hodge-Tate filtration as in \Cref{HT fil}, the induced map of graded pieces is isomorphic to the map of wedge powers of analytic cotangent complexes, up to Tate twists and cohomological shifts.
	In particular, the induced map of graded algebras for the Hodge-Tate filtration is an isomorphism, namely
	\[
	R_2\otimes_{R_1}^L \gr^{\mathrm{HT}}_\bullet R\Gamma_{\ol{\Prism}}(R_1) \cong \gr^{\mathrm{HT}}_\bullet R\Gamma_{\ol{\Prism}}(R_2).
	\]
	In this way, by the exhaustiveness of the Hodge-Tate filtration, the isomorphism of graded pieces above implies the isomorphism of $R\Gamma_{\ol{\Prism}}(-)$ as in the statement.
\end{proof}

\subsection{Pro-\'etale comparison}\label{subsec proetale and enlarged}
Now we can prove the pro\'etale comparison.
\begin{theorem}\label{proet comp}
	Let $X$ be a smooth rigid space over $K$.
	There is a natural isomorphism of $\Bdr$-$\mathbb{E}_\infty$-algebras as below
	\[
	R\Gamma(X/\Bdr, \mathcal{O}_\Prism) \rra R\Gamma(X_\pe, \BBdrp).
	\]
\end{theorem}

We start by constructing a natural map.
Recall from \cite{Sch13} that given a rigid space $X$ over $K$, its pro-\'etale site $X_\pe$ has a basis consisting of affinoid perfectoid objects $U$, associated with a complete adic space $\hat U=\Spa(A,A+)$, where $(A,A+)$ is a perfectoid affinoid algebra over $K$.
There is a \emph{deRham sheaf} $\BBdrp$ over $X_\pe$, sending an affinoid perfectoid object $U$ to the following
\[
\BBdrp(U)=(W(A^{+\flat})[1/p])^\wedge, \text{for}~\hat U=\Spa(A,A^+),
\]
where $(-)^\wedge$ is the $\xi$-adic completion.
Here $\BBdrp(U)$ is a complete topological $\Bdr$-algebra with $\BBdrp(U)/\xi=A$.
To simplify the notation, we use $\BBdrp(A)$ to denote $\BBdrp(U)$, where $U$ is an affinoid perfectoid object in $X_\pe$ associated to the perfectoid space $\hat U=\Spa(A,A^+)$.

We then define an enlarged category $X/\Bdr_\Prism^e$ of prisms over $X=\Spa(R)$, whose object is the union of $X/\Bdr_\Prism$ and the collection of triples
\[
(\BBdrp(A) \ra A \leftarrow R),
\]
where $A=(A,A^+)$ is an affinoid perfectoid object in $X_\et$.
A morphism is a commutative diagram of triples that are continuous under the $(p,\xi)$-adic topology, and we equip the category with the indiscrete topology.
Similar to the discussions in the previous sections, we can define the structure sheaves and their cohomology.
The following lemma tells us that this enlargement would not change the cohomology.
\begin{lemma}\label{pm enlarged uni}
	Let $R$ be a topologically finite type $K$-algebra, and let $P=\Bdr \langle T_i\rangle$ admit surjections $P\ra P/I\ra R$.
	Then the prism associated to the envelope $D_\Prism=P[\frac{I}{\xi}]^\wedge_\tf$ covers the final object of $\Sh(R/\Bdr_\Prism^e)$.
\end{lemma}
\begin{proof}
	By \Cref{pm uni}, the prism associated to the envelope $D_\Prism$ is weakly initial when restricted to the subcategory $R/\Bdr_\Prism$.
	So it suffices to show that any prism $(\BBdrp(A) \ra A \leftarrow R)$ for a perfectoid algebra $(A,A^+)$ over $R$ admits a map from $D_\Prism$.
	Let $R_0$ be a ring of definition of $R$ that is topologically finite type over $\mathcal{O}_K$, such that the image of $R_0$ in $A$ is inside of $A^+$.
	By lifting generators of $I$ into $\Ainf \langle T_i\rangle$, we choose surjection $\Ainf \langle T_i\rangle \ra R_0$ compatible with $P \ra R$.
	Then by lifting the image of $T_i$ to $W(A^{+\flat})$, we can form the following commutative diagram
	\[
	\xymatrix{ 
		\Ainf \langle T_i \rangle \ar[d] \ar[r] & R_0 \ar[d]\\
		W(A^{+\flat}) \ar[r] & A^+.}
	\]
	Next we invert $p$ and take the $\xi$-adic completion, we get the following commutative diagram
	\[
	\xymatrix{
		P=\Bdr \langle T_i\rangle \ar[d] \ar[r]& R \ar[d]\\
		\BBdrp(A) \ar[r] & A.}
	\]
	Note that from the diagram,  the image of $I\subset P$ is inside of $\xi\cdot \BBdrp(A)$.
	At last, as $\xi$ is a nonzerodivisor in $\BBdrp(A)$ (\cite[Theorem 6.5]{Sch13}), we thus get a map of prisms by the $\xi$-torsion freeness of $\BBdrp(A)$
	\[
	\xymatrix{ P[\frac{I}{\xi}]^\wedge_\tf \ar[r] \ar[d] & P[\frac{I}{\xi}]_\tf/\xi \ar[d] & R \ar[l] \ar@{=}[d] \\
		\BBdrp(A) \ar[r] & A & R \ar[l].}
	\]
\end{proof}
The inclusion functor naturally defines a map of indiscrete ringed sites 
\[
\iota: (X/\Bdr_\Prism^e, \mathcal{O}_\Prism) \rra (X/\Bdr_\Prism, \mathcal{O}_\Prism).
\]
The above lemma then leads to the following simple comparison.
\begin{corollary}\label{pm enlarged comp} 
	The natural map of cohomology of the structure sheaf of $X/\Bdr_\Prism$ and $X/\Bdr_\Prism^e$ is an isomorphism.
	The same holds for the reduced structure sheaf $\ol{\mathcal{O}}_\Prism$.
\end{corollary}
\begin{proof}
	Using the same proof of \Cref{pm product}, the \v{C}ech nerve for the map from $D_\Prism$ to the final object of $X/\Bdr_\Prism^e$ is the same as the one for $X/\Bdr_\Prism$.
	So by taking the \v{C}ech-Alexander complex as in \Cref{pm Cech}, we get the result.
\end{proof}
\begin{remark}
	The above in particular shows that for any map of sheaves $\iota^{-1} \mathcal{F}_1 \ra \mathcal{F}_2$ over $X/\Bdr_\Prism^e$ that induces an isomorphism when restricted to $X/\Bdr_\Prism$, we have a natural isomorphism
	\[
	R\Gamma(X/\Bdr_\Prism, \mathcal{F}_1) \rra R\Gamma(X/\Bdr_\Prism^e, \mathcal{F}_2).
	\]
\end{remark}
%For our purpose of the pro-\'etale comparison, we could not make difference of cohomology of $\mathcal{O}_\Prism$ and $\ol{\mathcal{O}}_\Prism$ over $X/\Bdr_\Prism$ and $X/\Bdr_\Prism^e$ in the following.

Now we can build the natural map of cohomology.
By construction, each affinoid perfectoid object $U\in X_\pe$ for $\hat U=\Spa(A,A^+)$ over some affinoid open subset $\Spa(R)\subset X$ naturally determines a unique prism $(\BBdrp(A) \ra A \leftarrow R)$.
In particular, as the collection of such affinoid perfectoid objects forms a basis of $X_\pe$, by taking the derived limit over all such $U$ and $R$, we get
\[
R\Gamma(X/\Bdr_\Prism^e, \mathcal{O}_\Prism) \rra R\Gamma(X_\pe, \BBdrp).
\]
Thus by composing with the isomorphism as in \Cref{pm enlarged comp}, we get
\[
\xymatrix{ R\Gamma(X/\Bdr_\Prism, \mathcal{O}_\Prism) \ar[r]^\sim & \ar[r] R\Gamma(X/\Bdr_\Prism^e, \mathcal{O}_\Prism) \ar[r]& R\Gamma(X_\pe, \BBdrp).}
\]

\begin{proof}[Proof of \Cref{proet comp}]
By taking derived limits on an affinoid open covering, it suffices to assume $X=\Spa(R)$ is affinoid.
Moreover, as both sides are derived complete $\xi$-adically, by the derived Nakayama \cite[Tag 0G1U]{Sta}, it suffices to show the isomorphism after mod $\xi$.
Namely the following map of $\mathbb{E}_\infty$-algebras over $K$
\[
R\Gamma_{\ol{\Prism}} (R) \rra R\Gamma(X_\pe, \wh{\mathcal{O}}_X), \tag{$\ast$}
\]
where $\wh{\mathcal{O}}_X$ is the complete structure sheaf over the pro-\'etale site, as in \cite{Sch13}.
Here we notice that the map is also $R$-linear.

To proceed, we notice that a truncation of the $R$-complexes above can be computed using the analytic cotangent complex.
\begin{proposition}\label{proet cotangent}
	Let $R$ be a smooth affinoid algebra over $K$.
	There is a natural commutative diagram as below, compatible with the map in \Cref{proet comp}, such that every arrow becomes an isomorphism after a cohomological truncation at degree $\leq 0$:
	\[
	\xymatrix{ & \mathbb{L}^\an_{R/\Bdr} \ar[ld] \ar[rd] & \\
		R\Gamma_{\ol{\Prism}}(R)(1)[1] \ar[rr] && R\Gamma(X_\pe, \ol{\mathcal{O}}_X)(1)[1].}
	\]
\end{proposition}
We grant this claim at the moment.
By the Hodge-Tate comparison as in \Cref{HT comp}, the above shows that the pro-\'etale comparison map in \Cref{proet comp} is an isomorphism when $X$ is a smooth curve.
On the other hand, as the \'etale localization holds on both sides of $(\ast)$ and everything is $R$-linear (\Cref{etale loc} for reduced prismatic cohomology and \cite{Sch13} for pro-\'etale cohomology of $\wh{\mathcal{O}}_X$), we may assume $X$ is equal to the $n$-dimensional closed disc $\mathbb{D}^n$.
At last, notice that both sides of $(\ast)$ satisfy the K\"unneth formula (\Cref{Kunneth reduced} for reduced prismatic cohomology and \cite[Proposition 8.14]{BMS} for the integral version of pro-\'etale cohomology).
Thus we can reduce the isomorphism to the case when $n=1$ and $X$ is a curve, which finishes the proof.

\end{proof}
\begin{proof}[Proof of \Cref{proet cotangent}]
	We first notice that for any prism $(B \ra B/\xi \leftarrow R)$, there is a natural map of analytic cotangent complex
	\[
	\mathbb{L}^\an_{R/\Bdr} \rra \mathbb{L}^\an_{(B/\xi)/B}\cong B(1)[1],
	\]
	where the isomorphism on the right hand side is because $\xi$ is a nonzerodivisor in $B$.
	This in particular includes the case when $B=\BBdrp(A)$ for a perfectoid $R$-algebra $(A,A^+)$.
	So by applying the derived limit over all prisms in $R/\Bdr_\Prism$, in $R/\Bdr_\Prism^e$, and over all prisms coming from affinoid perfectoid objects in $X_\pe$ separately, we get three maps from cotangent complex in the following diagram
	\[
	\xymatrix{ & \mathbb{L}^\an_{R/\Bdr} \ar[ld] \ar[d] \ar[rd] & \\
	  R\Gamma(R/\Bdr_\Prism, \ol{\mathcal{O}}_\Prism)(1)[1] &	R\Gamma(R/\Bdr_\Prism^e, \ol{\mathcal{O}}_\Prism)(1)[1] \ar[l]_\sim \ar[r] & R\Gamma(X_\pe, \ol{\mathcal{O}}_X)(1)[1],}
	\]
	where the left horizontal arrow is an isomorphism by \Cref{pm enlarged comp}.
	
	To show the isomorphism after the truncation, we first notice that by applying the distinguished triangle at the triple $\Bdr \ra K \ra R$ and the smoothness of $R$ over $K$, we see $\mathbb{L}^\an_{R/\Bdr}$ lives in degree $[-1,0]$ and fits in the triangle
	\[
	R(1)[1] \rra \mathbb{L}^\an_{R/\Bdr} \rra \Omega_{K/R}^1.
	\]
	Moreover, it is showed in \cite[Theorem 7.2.3]{Guo19} that the map below towards pro-\'etale cohomology is an isomorphism
	\[
	\mathbb{L}^\an_{R/\Bdr} \rra \tau^{\leq 0} \left( R\Gamma(X_\pe, \wh{\mathcal{O}}_X)(1)[1] \right).
	\]
	So it is left to show the isomorphism of $\mathbb{L}^\an_{R/\Bdr}$ with the truncated prismatic cohomology.
	
	Using the \'etale localization of both sides, we may assume $R=K\langle T_i\rangle$.
	As in the proof of \cite[Theorem 5.5.1]{Guo21}, by the distinguished triangle of the analytic cotangent complex for $\Bdr \ra P^{\wh\otimes_\Bdr \bullet +1} \ra R$, there is a cosimplicial diagram computing $\mathbb{L}^\an_{R/\Bdr}$ as follows
	\[
	\mathbb{L}^\an_{R/P^{\wh\otimes \bullet+1}}=\varprojlim_{[n]\in \Delta^\op} \mathbb{L}^\an_{R/P^{\wh\otimes n+1}},
	\]
	where $P^{\wh\otimes n+1}=P\langle \Delta(n) \rangle$ is the inverse limit (with respect to $e$) of the $(n+1)$-th complete self product of $\Bdre \langle T_i\rangle$.
	As each surjection $P\langle \Delta(n) \rangle\ra R$ is regular, the above cosimpliicial diagram is isomorphic to
	\[
	\varprojlim_{[n]\in \Delta^\op} (\xi, \Delta(n))/(\xi,\Delta(n))^2 [1].
	\]
	On the other hand, by \Cref{pm product} the above induces the \v{C}ech nerve for the weakly final prism $(P \ra P/\xi \leftarrow R)$, and we can form the \v{C}ech-Alexander complex for reduced prismatic cohomology in \Cref{pm Cech} as the following
	\[
	\varprojlim_{[n]\in \Delta^\op} P[\frac{\Delta(n)}{\xi}]^\wedge_\tf/\xi=\varprojlim_{[n]\in \Delta^\op} R[\frac{\Delta(n)}{\xi}].
	\]
	So the map $\mathbb{L}^\an_{R/\Bdr} \ra R\Gamma_{\ol{\Prism}}(R)(1)[1]$ is equivalent to the following map of homotopy limits
	\[
	\left( \varprojlim_{[n]\in \Delta^\op} (\xi, \Delta(n))/(\xi,\Delta(n))^2 [1] \right) \rra \left( \varprojlim_{[n]\in \Delta^\op} R[\frac{\Delta(n)}{\xi}](1)[1]\right) .
	\]
	Namely we reduce the question to show that the following map of cosimplicial diagrams induces an isomorphism of $\mathrm{H}^{\leq1}$
	\[
	\xymatrix{\xi\cdot R \ar@<2pt>[r]  \ar@<-2pt>[r]\ar[d] & (\xi,\Delta(1))/(\xi,\Delta(1))^2 \ar@<3pt>[r] \ar[r] \ar@<-3pt>[r] \ar[d] & (\xi,\Delta(2))/(\xi,\Delta(2))^2  \ar[d]  &\cdots \\
		\xi\cdot R \ar@<2pt>[r]  \ar@<-2pt>[r] & \xi\cdot R[\frac{\Delta(1)}{\xi}] \ar@<3pt>[r] \ar[r] \ar@<-3pt>[r]   & \xi\cdot R[\frac{\Delta(2)}{\xi}]  & \cdots.}
	\]
	Finally, to finish the calculation,  it suffices to notice that the maps above are all compatible with the original construction of the cosimplicial diagram $P^{\wh{\otimes}_\Bdr \bullet +1}$.
	So simply tracking the generators $T_i\in R=K \langle T_i\rangle$ and $\delta_i=1\otimes T_i -T_i \otimes 1 \in \Delta(1)$ for the first two degrees, we get the isomorphism of $\mathrm{H}^{\leq 1}$.
\end{proof}
\begin{corollary}
	Let $X$ be a smooth rigid space over $K$.
	Then there is a natural isomorphism for complexes of $\mathbb{E}_\infty$-$\mathrm{B_{dR}}$-algebras
	\[
	R\Gamma_{\inf}(X/\Bdr)[\frac{1}{\xi}] \rra R\Gamma_\Prism(X/\Bdr)[\frac{1}{\xi}].
	\]
\end{corollary}
\begin{proof}
	This follows from \cite[Theorem 13.1]{BMS}.
\end{proof}

\section{Galois invariant}\label{sec Galois inv}
In this section, we show that Galois invariant of prismatic cohomology is isomorphic to infinitesimal cohomology, when the rigid space is defined over a discretely valued subfield.
We will use the condensed mathematics developed by Clausen-Scholze.
Our references are the lecture notes \cite{Sch19} by Scholze and the recent preprint \cite{Bos21} by Bosco.

Throughout the section, we fix a  discretely valued subfield $k$ of $K$ that has perfect residue field.
We denote by $G_k$ the Galois group for the extension $K/k$.
We also fix an embedding $k\ra \Bdr$ compatible with the surjection $\Bdr\ra K$.

\begin{comment}
Our main theorem of this section is the following.
\begin{theorem}\label{galois}
	Let $Y$ be a rigid space over $k$ with l.c.i singularities.
	Then for each $i\in \mathbb{N}$, there is an isomorphism of finite dimensional $k$-vector spaces
	\[
	\mathrm{H}^i(Y/k_{\inf}, \mathcal{O}_{Y/k}) \rra \mathrm{H}^i(Y_K/\Bdr, \mathcal{O}_\Prism)^{G_k}.
	\]
\end{theorem}
\begin{remark}
	Notice that there is a natural base extension isomorphism
	\[
	R\Gamma(Y/k_{\inf}, \mathcal{O}_{Y/k})\otimes_k \Bdr \rra R\Gamma(Y_K/\Bdr, \mathcal{O}_{Y_K/\Bdr}).
	\]
	Thus the above theorem in particular gives an way to recover the infinitesimal cohomology, when $X$ is defined over a discretely valued subfield but not necessarily smooth.
\end{remark}
\end{comment}
\subsection{Condensed mathematics and condensed group cohomology}\label{subsec condensed}
We recall the basics about the condensed mathematics and condensed group cohomology, following \cite{Sch19}, \cite[Appendix]{Bos21} and \cite[Section 4.3]{BS15}.

\begin{convention}
	As in \cite{Sch19}, we fix a uncountable strong limit cardinality $\kappa$ throught the section.
	Every construction is defined to be $\kappa$-small unless otherwise specified.
\end{convention}

\begin{definition}\label{cond def}
	The category of \emph{condensed sets/groups/rings}, is defined as a sheaf of sets/groups/rings over the pro-\'etale site of a point $\ast_\pe$.
\end{definition}
To relate to the usual notion of the topological structure, we have the following results.
\begin{theorem}\label{cond and top}\cite[Example 1.5, Proposition 1.7, Theorem 2.16]{Sch19}
	There is a natural functor from the category of topological spaces/groups/rings to the category of condensed sets/groups/rings, denoted as $M\mapsto \ul M$, where the latter sends a profinite set $S\in \ast_\pe$ to the set of continuous maps from $S$ to $M$.
	The functor satisfies the following:
	\begin{enumerate}[(i)]
		\item the functor is faithful, and is fully faithful when restricted to those $T$ whose underlying topological space is $\kappa$-compactly generated;
		\item it induces an equivalence between compact Hausdorff spaces $X$ and qcqs condensed sets;
		\item it admits a left adjoint $X\mapsto X(\ast)_\mathrm{top}$, such that for a topological space $T$, the composition $\ul T(\ast)_\mathrm{top}$ is the same set but with the compact generated topology.
	\end{enumerate}
\end{theorem}
For a set/group/ring $M$, we say a condensed set/group/ring $A$ \emph{underlies} $M$ if we have $A(\ast)\cong M$.

Denote $\mathrm{LCA}$ to be the category of locally compact abelian groups, and $\Cd$ to be the category of condensed abelian group.
For two condensed abelian groups $M$ and $N$, we let the condensed group $\ul\Hom_\Cd(M,N)\in \Cd$ be the internal group of homomorphisms between them (\cite[Section 2]{Sch19}).

When restricted to abelian groups, the functor $M \mapsto \ul M$ induces a fully faithful embedding from $\mathrm{LCA}$.
\begin{theorem}\label{cond and lca}\cite[Proposition 4.2, Corollary 4.9]{Sch19}
	The natural functor $M \mapsto \ul M$ induces a fully faithful embedding
	\[
	D^b(\mathrm{LCA}) \rra D(\Cd),
	\]
	such that
	\[
	\ul\Ext^i_\Cd(\ul M, \ul N) \cong \ul{\Ext^i_\mathrm{LCA} (M,N)}.
	\] 
\end{theorem}

On the other hand, the category of condensed abelian groups behave exactly like the category of abelian groups.
\begin{theorem}\cite[Theorem 2.2]{Sch19}
	The category $\Cd$ is an abelian category, satisfying Grothendieck's axioms (AB3), (AB4), (AB5), (AB3*), (AB4*) and (AB6).
	Moreover, it is generated by compact projective objects $\mathbb{Z}[S]$, for $S\in \ast_\pe$ being an extremally disconnected set.
\end{theorem}
Inside of the category $\Cd$, there is abelian subcategory called $\Sd$.
We recall its definition as below.
\begin{definition}\label{solid def}\cite[Definition 5.1]{Sch19}
	\begin{enumerate}
		\item For a profinite set $S$, we define the condensed abelian group $\mathbb{Z}[S]^\blacksquare:=\varprojlim_i \mathbb{Z}[S_i]$, where each $S_i$ is finite and $S=\varprojlim_i S_i$;
		\item A condensed group $A$ is called \emph{solid} if the natural map below is an isomorphism
		\[
		\Hom_\Cd(\mathbb{Z}[S], A) \rra \Hom_\Cd(\mathbb{Z}[S]^\blacksquare, A).
		\]
	\end{enumerate}
\end{definition}
Similarly one can define a complex $C\in D(\Cd)$ to be solid by replacing $\Hom$ by $R\Hom$, and it can be showed that those two notions are compatible (\cite[Theorem 5.8]{Sch19}).
We list the important properties of $\Sd$ below.
\begin{theorem}\cite[Theorem 5.8, Corollary 5.5]{Sch19}\label{solid prop}
	\begin{enumerate}[(i)]
		\item The category $\Sd$ is an abelian subcategory of $\Cd$ that is stable under limits, colimits and extensions.
		\item The objects $\prod_I \mathbb{Z}\in \Sd$ for any set of $I$ forms a family of compact projective generators of $\Sd$. In particular $\mathbb{Z}[S]^\blacksquare$ is solid for any profinite set $S$.
		\item The functor $D(\Sd)\subset D(\Cd)$ is fully faithful whose essential image are solid complexes, and a complex is solid if and only if each of its cohomology is solid.
		\item The inclusion $\Sd\subset \Cd$ admits a left adjoint $A\mapsto A^\blacksquare$, which is the unique colimit preserving functor extending $\mathbb{Z}[S] \mapsto \mathbb{Z}[S]^\blacksquare$.
		The left derived functor associated to $A\mapsto A^\blacksquare$ is also the left adjoint to the inclusion $D(\Sd)\subset D(\Cd)$.
	\end{enumerate}

\end{theorem}

We then turn to the notion of condensed group cohomology.
We first recall its definition as in \cite[Appendix B.1]{Bos21}.
\begin{definition}\label{group coh def}
	Let $G$ be a condensed group, and let $A$ be a $G$-module in $\Cd$.
	The $i$-th \emph{condensed group cohomology} of $G$ in $A$, as an object in $\Cd$, is 
	\[
	\ul\Ext^i_{\mathbb{Z}[G]}(\mathbb{Z}, A).
	\]
\end{definition}
We use the notation $R\Gamma(G,-)$ for the composition functor $R\Hom_{\mathbb{Z}[G]}(\mathbb{Z}, -)=R\Gamma(\ast, R\ul\Hom_{\mathbb{Z}[G]}(\mathbb{Z}, -))$, evaluating at condensed $G$-modules.

The following useful proposition says condensed group cohomology recovers the usual notion of continuous group cohomology in many cases.
\begin{proposition}\label{group coh prop}\cite[Appendix B.2]{Bos21}
	Let $G$ be a profinite group, and let $A$ be a solid $\ul G$-module.
	\begin{enumerate}
		\item The complex $R\ul\Hom_{\mathbb{Z}[\ul G]}(\mathbb{Z}, A)$ is isomorphic to the complex of solid abelian groups
		\[
		A \rra \ul\Hom_{\mathbb{Z}}(\mathbb{Z}[\ul G], A) \rra \ul\Hom_{\mathbb{Z}}(\mathbb{Z}[\ul G \times \ul G], A) \rra \cdots.
		\]
		\item If $A=\ul M$ for some $T_1$ topological $G$-module $M$, then the global section functor induces a natural isomorphism below
		\[
		\Ext_{\mathbb{Z}[\ul G]}^i (\mathbb{Z}, \ul M) \rra \mathrm{H}^i_\ct(G,M).
		\]
	\end{enumerate}
\end{proposition}
\begin{proof}
	Here we note that though the item (i) is slightly different from the statement in \cite[Appendix B.2.(i)]{Bos21} where $A$ is assumed to be $\ul M$ already, the proof applies to the more general case.
\end{proof}
\begin{remark}
	The condensed group cohomology functor $R\ul\Hom_{\mathbb{Z}[G]}(\mathbb{Z}, -)$ for a condensed group $G$ is the right derived functor of $\ul\Hom_{\mathbb{Z}[G]}(\mathbb{Z}, -)$.
	In particular, since the latter is lax-symmetric monoidal, it naturally induces a functor from $\mathbb{E}_\infty$-algebras in $\Cd$ (resp. $\Sd$) that admits a $G$-action to $\mathbb{E}_\infty$-algebras in $\Cd$ (resp. $\Sd$).
\end{remark}

We finish this subsection by collecting some results on affinoid algebras and its associated condensed rings.
We refer the reader to \cite[Appendix A]{Bos21} for a discussion of solid $k$-modules, together with the symmetric tensor product $-\otimes_k^\bs -$.
\begin{lemma}\label{cond of affinoid}
	Let $R_0=k \langle T_i\rangle/ I_0$ be a topologically finite type $k$-algebra, let $R=R_0\wh\otimes_k K\cong K\langle T_i\rangle/ I_0$, let $R_e= \Bdre \langle T_i\rangle/I_0$, and let $\wt R=\Bdr \langle T_i\rangle/I_0=\varprojlim R_e$.
	Then each of $\ul R_0, \ul R, \ul R_e, \ul{\wt R}$ is solid, and satisfies the formulae
	\[
	\ul R=\ul R_0\otimes_{\ul k}^\bs \ul K,~~~~~~\ul R_e =\ul R_0\otimes_{\ul k}^\bs \ul\Bdre,~~~~~~\ul{\wt R} = \ul R_0 \otimes_{\ul k}^\bs \ul \Bdr.
	\]
\end{lemma}
To save us from complicated notations, we use $k$, $K$, $\Bdre$, and $\Bdr$ instead of $\ul k$, $\ul K$, $\ul \Bdre$ or $\ul\Bdr$ to denote their associated solid rings.
\begin{proof}
	For the solidity of the rings above, this can be checked using the definition of the functor $M\mapsto \ul M$ in \Cref{cond and top}, and the fact that solid modules are preserved under limits and colimits as in \Cref{solid prop}.
	For example the ring $R$ is computed as 
	\[
	R\cong \colim_{n\in\NN} \left(\varprojlim_{m\in\NN} p^{-n}(R'/p^mR') \right),
	\]
	where $R'$ is a ring of definition of $R$, and thus each $R'/p^mR'$ is a discrete abelian group.
	By \Cref{cond and top}, it then suffices to notice that a continuous map from a profinite set $S$ to $R$ can be written as a of limit of continuous maps from $S$ to $p^{-n}(R'/p^mR')$, for some $n\in \mathbb{N}$.
	
	For the tensor product formulae, the first two with $\ul K$ and $\Bdrn$ follows  from the observation that the functor taking the associated condensed group commutes with the complete tensor product, over the category of Banach $k$-algebras (\cite[Appendix A.26]{Bos21}).
	For the last one, we apply the limit-tensor formula in \cite[Appendix A.24.(i)]{Bos21} at $V_n=\Bdrn$ and $W=\ul R_0$, to get
	\[
	(\varprojlim_n \Bdrn)\otimes_k^\bs \ul R_0 \cong \varprojlim_n (\Bdrn \otimes_k^\bs \ul R_0) \cong \varprojlim_n (\ul{\Bdrn \langle T_i\rangle/I_0}),
	\]
	where the last formula follows from \cite[Appendix A.26]{Bos21} (recall from \cite{Fon94} that $\Bdr$ is a $k$-Fr\'echet algebra).
	To finish the proof, we notice that by definition of the functor $M\mapsto \ul M$ in \Cref{cond and top} we have $\varprojlim_n  (\ul{\Bdrn \langle T_i\rangle/I_0})=\ul{\Bdr \langle T_i\rangle/I_0}=\ul{\wt R}.$	
\end{proof}
\subsection{Galois invariant of cohomology}
We now compute Galois invariants of cohomology.

We recall the following fact computing continuous Galois cohomology of $\Bdr$-modules.
\begin{lemma}\label{Galois of Bdr}
	We have the following equalities:
\[
\mathrm{H}_{\mathrm{cont}}^n(G_k, K(i))= \begin{cases}
	k,~n=i=0;\\
	k\cup \log\chi,~n=1,~i=0;\\
	0,~n\geq 2,~or~i\neq 0.
\end{cases}
\]
The same results hold if we replace $K(0)$ by $\Bdr$.
Here $\log \chi\in \mathrm{H^1}_{\mathrm{cont}}(G_k, \mathbb{Q}_p)=\mathrm{Hom_{cont}}(G_k,\mathbb{Q}_p)$ is the logarithm of the cyclotomic character.
\end{lemma}\label{Galois of P}
\begin{proof}
	The calculation of Galois cohomology for $K(i)=\xi^i\Bdr/\xi^{i+1}\Bdr$ can be found in [\cite{Tat67}].
	For $\Bdr$ this follows from an induction and an inverse limit process.
\end{proof}
\begin{lemma}\label{Galois of diff}
	Let $P$ be the $\Bdr$-algebra $\Bdr \langle T_j\rangle$, and let $\Omega_{P/\Bdr}^i=\varprojlim \Omega_{\Bdre \langle T_j\rangle /\Bdre}^i$ be the inverse limit of $p$-adic continuous differentials.
	Let $P_0=k \langle T_j\rangle$ be the affinoid $k$-subalgebra inside of $P$, let $I_0$ be an ideal in $P_0$, and let $I=I_0\otimes_{P_0} P$ be the ideal in $P$.
	Then for $n,i \in \mathbb{N}$ we have 
	\[
	\mathrm{H}_{\mathrm{cont}}^n(G_k, \Omega_{P/\Bdr}^i\otimes_P I)= \begin{cases}
			\Omega_{P_0/k}^i\otimes_{P_0} I_0,~n=0;\\
			\Omega_{P_0/k}^i\otimes_{P_0} I_0\cup \log\chi,~n=1;\\
			0,~n\geq 2,
		\end{cases}
		\]
	where $\Omega_{P_0/k}^i$ is the $i$-th continuous differential of $P_0$ over $k$.
	The same holds if we replace $P$ by $K \langle T_i\rangle$ or $\Bdre \langle T_i\rangle$.
\end{lemma}
\begin{proof}
%	We first notice that each finite free $P$-module $\Omega_{P/\Bdr}^i$ is equal to the tensor product $\Omega_{P_0/k}^i\otimes_{P_0} P$, where $\Omega_{P_0/k}^i$ is finite free over $P_0$.
%	So it suffices to do the calculation for $\Omega_{P/\Bdr}^0=P$.
	Similarly to \Cref{Galois of Bdr}, by a limit argument and the induction, one can reduce to the calculation of continuous group cohomology for $P/\xi(i)=K \langle T_i\rangle (i)$, which follows from Tate-Sen formalism (see \cite[Lemma 3.10]{LZ17} for example).
\end{proof}

Now we compute Galois invariant of infinitesimal cohomology.
\begin{theorem}\label{Galois of inf}
	Let $Y$ be a rigid space over $k$, and let $Y_K$ be its complete base extension to $K$.
	\begin{enumerate}[(i)]
		\item There is a natural complex of condensed abelian groups  $\ul{R\Gamma}_{\inf}(Y/k)$ and a complex of $\ul G_k$-condensed modules $\ul{R\Gamma}_{\inf}(Y_K/\Bdr)$ underlying infinitesimal cohomology $R\Gamma(Y/k_{\inf}, \mathcal{O}_{Y/k})$ and $R\Gamma(Y_K/\Bdr_\mathrm{pinf}, \mathcal{O}_{Y_K/\Bdr})$ separately;
		\item the above complexes satisfy a solid base change formula
		\[
		\ul{R\Gamma}_{\inf}(Y/k) \otimes_k^{L\bs} \Bdr \cong \ul{R\Gamma}_{\inf}(Y_K/\Bdr);
		\]
		\item the $\ul G_k$-action induces the equality
		\[
		R\Gamma(\ul G_k, \Bdr)\otimes_k R\Gamma(Y/k_{\inf}, \mathcal{O}_{Y/k}) \rra R\Gamma(\ul G_k, \ul{R\Gamma}_{\inf}(Y_K/\Bdr)).
		\]
	\end{enumerate}

\end{theorem}
\begin{remark}\label{Galois fail}
	Before we jump to the proof, one might attempt to ask if we can replace the condensed infinitesimal cohomology $\ul{R\Gamma}_{\inf}(Y_K/\Bdr)$ in (iii) by the actual infinitesimal cohomology $R\Gamma(Y_K/\Bdr_\mathrm{pinf}, \mathcal{O}_{Y_K/\Bdr})$ and compute continuous group cohomology in the classical sense.
	However, even though each object in $Y_K/\Bdr_{\mathrm{pinf}}$ admits a natural $k$-Fr\'echet structure, one cannot conclude that $R\Gamma(Y_K/\Bdr_\mathrm{pinf}, \mathcal{O}_{Y_K/\Bdr})$ is also topological.
	This is due to the lack of the notion of a derived/homotopy limit for topological groups.
\end{remark}
\begin{proof}
We first give the definition of condensed infinitesimal cohomology as in part (i).
We define the condensed structure sheaf $\ul{\mathcal{O}}_{\inf}$ over $Y/k_{\inf}$, by sending an affinoid infinitesimal thickening $(R_0, S_0)$ onto the condensed ring $\ul S_0$.
Here recall that both $R_0$ and $S_0$ are topologically finite type algebras over $k$, and $S_0\ra R_0$ is a surjection with the kernel being nilpotent.
On the other hand, recall for a pro-infinitesimal thicking $(\Spa(R), \Spa(B_e))\in Y_K/\Bdr_{\mathrm{pinf}}$ with $B_e$ being a topologically finite type algebra over $\Bdre$, we can consider the induced condensed ring $\ul B=\varprojlim \ul B_e$.
	This allows us to define the condensed structure sheaf $\ul{\mathcal{O}}_{Y_K/\Bdr}$ over $Y_K/\Bdr_\mathrm{pinf}$ separately, sending any such pro-infinitesimal thickening onto $\ul B$.
	\footnote{A priori this is only a presheaf in condensed ring.
		To get the sheafifiness, one can use for example \cite[Appendix A.15]{Bos21}.}
Our condensed infinitesimal cohomology are then defined as infinitesimal cohomology of condensed structure sheaves, as objects in $D(\Cd)$
\begin{align*}
	\ul{R\Gamma}_{\inf}(Y/k) & := R\Gamma(Y/k_{\inf}, \ul{\mathcal{O}}_{Y/k});\\
	\ul{R\Gamma}_{\inf}(Y_K/\Bdr) & := R\Gamma(Y_K/\Bdr_\mathrm{pinf}, \ul{\mathcal{O}}_{Y_K/\Bdr}).
\end{align*}

	We now temporarily let $Y=\Spa(R_0)$ for a topological finite type algebra $R$ over $k$, and let $R=R_0\wh\otimes_k K$ be its complete field extension.
	Suppose we have a surjection of affinoid $k$-algebras $P_0=k \langle T_i \rangle \ra R$ with $R=P_0/I_0$ for an ideal $I_0\subset P_0$, and let $P$ be the $\Bdr$-algebra $\Bdr \langle T_i\rangle$.
	We then obtain a $\ul G_k$-condensed algebra underlying the infinitesimal envelope $D_{\inf}=\varprojlim \Bdre \langle T_i\rangle/I_0^e$ of the pro-infinitesimal thickening $(Y,\Spa(P/(\xi^e,I_0^e)))$ as below
	\[
	\ul D_{\inf}:=\ul{\mathcal{O}}_{Y_K/\Bdr}(Y,\Spa(P/(\xi^e,I_0^e))) \cong \varprojlim \ul{\Bdre \langle T_i\rangle/I_0^e}.
	\]
	Here we note that by rewriting this limit, the above is also equal to $\varprojlim \ul{\Bdr \langle T_i\rangle /I_0^e}$, where  by \Cref{cond of affinoid} each $\ul{\Bdr \langle T_i\rangle/I_0^e}$ is isomorphic to the solid tensor product $\Bdr \otimes^\bs_k \ul{P_0/I_0^e}$.
	So thanks to the limit-tensor product formula \cite[Appendix A.24]{Bos21}, since $\{P_0/I_0\}$ is an inverse system of $K$-Banach algebras, we get
	\[
	\ul D_{\inf} \cong  \Bdr\otimes^\bs_k \varprojlim_e \ul{P_0/I_0} =\Bdr \otimes^\bs_k \ul{D}_{0,{\inf}},
	\]
	where $D_{0,{\inf}}$ is the infinitesimal envelope of $R_0$ inside of $P_0$, as a pro-infinitesimal thickening in $R_0/k_{\inf}$, and $\ul{D}_{0,{\inf}}$ is the section of $\ul{\mathcal{O}}_{Y/k}$ at it.
	
	Now we apply the above computation to the \v{C}ech-Alexander complex for the infinitesimal sites.
	Recall from \cite[Proposition 2.27]{Guo21} that infinitesimal cohomology for $Y/k_{\inf}$ and $Y_K/\Bdr_{\inf}$ are computed by the cosimplicial complexes $D_{0,{\inf}}(\bullet)$ and $D_{\inf}(\bullet)$, where $D_{0,{\inf}}(n)$ and $D_{\inf}(n)$ are formal completion for the surjections $P_0(n):=P_0^{\wh\otimes_k n+1} \ra R_0$ and $P(n):=P^{\wh\otimes_\Bdr n+1} \ra R$ separately.
	We apply these to the condensed structure sheaves to get the cosimplicial condensed algebras 
	\begin{align*}
			\ul{R\Gamma}_{\inf}(Y/k) &\cong R\varprojlim_{[n]\in \Delta^\op} \ul{D}_{0,{\inf}}(n) ;\\
			\ul{R\Gamma}_{\inf}(Y_K/\Bdr) & \cong R\varprojlim_{[n]\in \Delta^\op} \ul{D}_{{\inf}}(n).
	\end{align*}
Here we note that since a condensed group is a sheaf over the pro-\'etale site $\ast_\pe$, by applying the derived global section functor $R\Gamma(\ast,-)$,  we see from below that the complex of condensed groups $\ul{R\Gamma}_{\inf}(Y/k)$ underlie the corresponding usual infinitesimal cohomology
\begin{align*}
	R\Gamma(\ast, \ul{R\Gamma}_{\inf}(Y/k)) &\cong R\varprojlim_{[n]\in \Delta^\op} R\Gamma(\ast, \ul{D}_{0,{\inf}}(n)) \\
	&\cong R\varprojlim_{[n]\in \Delta^\op} D_{0,{\inf}}(n) \\
	&\cong R\Gamma_{\inf}(Y/k_{\inf}, \mathcal{O}_{Y/k}).
\end{align*}
Similarly for $\ul{R\Gamma}_{\inf}(Y_K/\Bdr)$.
Here we use \Cref{cond and top} together with the exactness of $R\Gamma(\ast, -)$, as each profinite set admits a map from $\ast$.

Moreover, notice that each $\ul{D}_{{\inf}}(n)$ satisfies the solid tensor product formula 
\[
\ul{D}_{{\inf}}(n) \cong \Bdr \otimes^\bs_k \ul{D}_{0,{\inf}}(n).
\]
So combining with the flatness of $\ul{\Bdr}$ over $\ul k$ as in \cite[Appendix A.6, A.10, A.12]{Bos21}, we get the solid tensor product formula for condensed infinitesimal cohomology as in item (ii)
\[
R\varprojlim_{[n]\in \Delta^\op} \ul{D}_{{\inf}}(n) \cong \Bdr \otimes^{L\bs}_k R\varprojlim_{[n]\in \Delta^\op} \ul{D}_{0,{\inf}}(n).
\]

To compute Galois cohomology, we notice the following formula at each $[n]\in \Delta^\op$ \footnote{We apologize for the abuse of notation here: $\ul D_{\inf}(n)$  (similarly for other $(n)$ in the rest of proof) means the $n$-th term in the cosimplicial diagram, not the $n$-th Tate twist.}
\begin{align*}
	R\Hom_{\mathbb{Z}[\ul G_k]}(\mathbb{Z}, \ul{D}_{{\inf}}(n)) & =  R\Hom_{\mathbb{Z}[\ul G_k]}\left(\mathbb{Z}, \varprojlim_m \ul{P(n)/(I_0, \Delta(n))^mP(n)}\right),
\end{align*}
where $(I_0, \Delta(n))P(n)$ is the ideal of $P(n)$ generated by the ideal $(I_0,\Delta(n))\subset P_0(n)$.
By commuting the order of limit and Galois cohomology, the above is further equal to 
\[
R\varprojlim_m R\Hom_{\mathbb{Z}[\ul G_k]}\left(\mathbb{Z}, \ul{P(n)/(I_0, \Delta(n))^mP(n)}\right).
\]
Here we use explicitly the equality  $\varprojlim_m \ul{P(n)/(I_0, \Delta(n))^mP(n)} \cong R\varprojlim_m \ul{P(n)/(I_0, \Delta(n))^mP(n)}$, by the repleteness of $\Sh(\ast_\pe)$ (\cite[3.1.10, 3.2.3, 4.2.8]{BS15}) together with the surjectivity of condensed groups $\ul{P(n)/(I_0, \Delta(n))^{m+1}P(n)}\ra \ul{P(n)/(I_0, \Delta(n))^mP(n)}$ (\cite[Appenxid A.15]{Bos21}).
Notice that by \Cref{Galois of diff} and \Cref{group coh prop}, applying $R\Hom_{\mathbb{Z}[\ul G_k]}(\mathbb{Z}, - )$ at the short exact sequence 
\[
\ul{(I_0, \Delta(n))^mP(n)} \ra \ul{P(n)}  \ra \ul{P(n)/(I_0, \Delta(n))^mP(n)},
\] we get 
\[
R\Hom_{\mathbb{Z}[\ul G_k]}\left(\mathbb{Z}, \ul{P(n)/(I_0, \Delta(n))^mP(n)}\right) \cong R\Hom_{\mathbb{Z}[\ul G_k]}(\mathbb{Z}, \Bdr) \otimes_k \left( P_0(n)/(I_0, \Delta(n))^mP(n) \right).
\]
After a further inverse limit with respect to $m$, and with a help of the finiteness of Galois cohomology of $\Bdr$ as in \Cref{Galois of Bdr}, we have
\[
R\Hom_{\mathbb{Z}[\ul G_k]}(\mathbb{Z}, \ul{D}_{{\inf}}(n)) \cong  R\Hom_{\mathbb{Z}[\ul G_k]}(\mathbb{Z}, \Bdr)\otimes_k D_{0,{\inf}}(n).
\]	
In this way, applying the above formula at the cosimplicial diagram, we get
\begin{align*}
		R\Hom_{\mathbb{Z}[\ul G_k]}(\mathbb{Z}, \ul{R\Gamma}_{\inf}(Y_K/\Bdr)) & \cong  R\Hom_{\mathbb{Z}[\ul G_k]}(\mathbb{Z}, R\varprojlim_{[n]\in \Delta^\op} \ul{D}_{{\inf}}(n)) \\
		& \cong R\varprojlim_{[n]\in \Delta^\op}  R\Hom_{\mathbb{Z}[\ul G_k]}(\mathbb{Z}, \ul{D}_{{\inf}}(n)) \\
		&  \cong R\varprojlim_{[n]\in \Delta^\op} \left( R\Hom_{\mathbb{Z}[\ul G_k]}(\mathbb{Z}, \Bdr)\otimes_k D_{0,{\inf}}(n) \right)\\
		& \cong R\Hom_{\mathbb{Z}[\ul G_k]}(\mathbb{Z}, \Bdr) \otimes_k \left( R\varprojlim_{[n]\in \Delta^\op}D_{0,{\inf}}(n) \right)\\
		& \cong R\Hom_{\mathbb{Z}[\ul G_k]}(\mathbb{Z}, \Bdr)  \otimes_k R\Gamma_{\inf}(Y/k, \mathcal{O}_{Y/k}).
\end{align*}
Here the second to the last isomorphism again follows from the finiteness of Galois cohomology of $\Bdr$ in \Cref{Galois of Bdr}.

At last, to get the result for general rigid spaces $Y$ (which are assumed to be qcqs), we can form a homotopy limit over a finite affinoid \v{C}ech covering, and apply the equalities above.
Here we use the finiteness of $R\Hom_{\mathbb{Z}[\ul G_k]}(\mathbb{Z}, \Bdr)$ and the flatness of $\ul\Bdr$ over $\ul k$ again to guarantee that the homotopy limits commute with the tensor products.
\end{proof}
Here we notice that using the same method, we can improve the above into a filtered version as below.
\begin{corollary}\label{Galois filtered}
		Let $Y$ be a rigid space over $k$, and let $Y_K$ be its complete base extension to $K$.
		\begin{enumerate}[(i)]
			\item There is a natural complex of condensed abelian groups  $\ul{\Fil}^i_{\inf}(Y/k)$ and a complex of $\ul G_k$-condensed modules $\ul{\Fil}^i_{\inf}(Y_K/\Bdr)$ underlying the $i$-th infinitesimal filtration $R\Gamma(Y/k_{\inf}, \mathcal{I}^i_{Y/k})$ and the $i$-th lifted filtration $\wt\Fil^iR\Gamma(X/\Bdr_{\mathrm{pinf}}, \mathcal{O}_{X/\Bdr})$ associated to $\{Y_{\Bdre}\}$  (c.f. \Cref{lifted inf}) separately;
			\item the above complexes satisfy a solid base change formula
			\[
			\ul{\Fil}^i_{\inf}(Y/k) \otimes_k^{L\bs} \Bdr \cong \ul{\Fil}^i_{\inf}(Y_K/\Bdr);
			\]
			\item the $\ul G_k$-action induces the equality
			\[
			R\Gamma(\ul G_k, \Bdr)\otimes_k R\Gamma(Y/k_{\inf}, \mathcal{I}^i_{Y/k}) \rra R\Gamma(\ul G_k, \ul{\Fil}^i_{\inf}(Y_K/\Bdr)).
			\]
		\end{enumerate}
	The same holds if we replace $\Bdr$ by $\Bdre$.
\end{corollary}
\begin{remark}\label{Galois tensor product}
	It is worth mentioning that \ref{Galois filtered} (ii) is \emph{not} the same as the filtered tensor product formula considered in the usual sense; instead we endow $\Bdr$ here with the trivial filtration at the degree $0$.
	To get the usual filtered tensor product formula for infinitesimal cohomology, we may regard $\Bdr$ as a filtered solid algebra with $\Fil^i\Bdr= \xi^i\Bdr$.
	In this way, using the similar idea of proof as in \Cref{Galois of inf} (with the difference that we take the filtered completion of the filtered tensor product $\ul D_{0, {\inf}}(n)\otimes_k^\bs \Bdr$ as $\ul D_{\inf}(n)$), one can also prove the following tensor product formula in the filtered derived category $\DF(k)$:
	\[
	(\ul{R\Gamma}_{\inf}(Y/k)\otimes^L_k \Bdr)^\wedge \cong \ul{R\Gamma}_{\inf} (Y_K/\Bdr),
	\]
	where the right hand side is equipped with the usual infinitesimal filtration for $Y_K/\Bdr_{\inf}$ (instead of the lifted filtration), and $(-)^\wedge$ is the filtered completion.
\end{remark}

Notice that the proof of \Cref{Galois of inf} above also implies the following formula of Galois invariant.
\begin{corollary}\label{Galois inv}
	Let $Y$ be a rigid space over $k$, and let $Y_K$ be its complete base extension to $K$.
	Then we have an isomorphism of $k$-vector spaces
	\[
	\mathrm{H}^0(\ul G_k, \ul{\mathrm{H}}^i_{\inf}(Y_K/\Bdr) \cong \mathrm{H}^i(Y/k_{\inf}, \mathcal{O}_{Y/k}).
	\]
\end{corollary}
\begin{proof}
	Thanks to \Cref{Galois of Bdr} and \Cref{Galois of P}, we have 
	\[
	\mathrm{H}^i(\ul G_k, \ul D_{\inf}(n))= D_{0, {\inf}}(n),~i=0,1;
	\]
	and vanishes for other $i$.
	Thus  the statement follows by inductively applying the equality at the cosimplicial diagram $\ul D_{\inf}(n)$ and Mayer-Vietoris sequences for \v{C}ech covering by affinoid open subsets.
\end{proof}
\begin{remark}
	When $Y$ is affinoid, by choosing a closed immersion into a polydisc $k\langle T_i\rangle$, one can compute its infinitesimal cohomology using the completed de Rham complex of $k\langle T_i\rangle^\wedge_{I_0}$ over $k$, where $I_0$ is the defining ideal of $Y$.
	In the special case when $Y$ is smooth, we can even use the continuous de Rham complex of $Y/k$.
	Pointed out as in \cite[Remark 5.14, 5.22]{Bos21}, the underlying topological structures of the condensed infinitesimal cohomology and the usual infinitesimal cohomology are not the same.
	Thus the underlying topological group of condensed Galois invariant in \Cref{Galois of inf} is different from the usual continuous Galois invariant of infinitesimal cohomology.
	However, one may still use criteria in \cite[Lemma 4.3.9]{BS15} to check that the underlying group structures, forgetting their topology, would coincide.
\end{remark}
We are now ready to compute Galois invariant of prismatic cohomology.
\begin{theorem}\label{Galois of prism}
	Let $Y$ be a rigid space over $k$ that has l.c.i singularities, and let $Y_K$ be its complete field extension to $K$.
	\begin{enumerate}[(i)]
		\item There is a natural complex of condensed groups $\ul{R\Gamma}_\Prism(Y_K/\Bdr)$ underlying prismatic cohomology $R\Gamma_\Prism(Y_K/\Bdr, \mathcal{O}_\Prism)$;
		\item Galois invariant of prismatic cohomology is equal to infiniteismal cohomology of $Y/k_{\inf}$, namely
		\[
		R\Gamma(\ul G_k, \Bdr)\otimes_k R\Gamma(Y/k_{\inf}, \mathcal{O}_{Y/k}) \rra R\Gamma(\ul G_k, \ul{R\Gamma}_\Prism(Y_K/\Bdr)).
		\]
	\end{enumerate}
\end{theorem}
\begin{proof}
Similar to \Cref{prism coh}, we apply Simpson's functor at the condensed complex $\ul{R\Gamma}_{\inf}(Y_K/\Bdr)$ and its lifted filtration (associated to the canonical lifts $\{Y_\Bdre\}$).
Namely, we consider the condensed complex below
\begin{align*}
	\ul{R\Gamma}(Y_K/\Bdr_\Prism) &= \Psi(\ul{R\Gamma}_{\inf}(Y_K/\Bdr), \ul\Fil^\bullet_{\inf}) \\
	&:=\left(\underset{n\in \mathbb{N}}{\colim} ~ \xi^{-n}\cdot\ul{\Fil}^n_{\inf}(Y_K/\Bdr) \right)^\wedge, \tag{$\ast$}
\end{align*}
where $(-)^\wedge$ is the condensed version of derived $\xi$-adic completion, i.e. the functor $R\varprojlim_e (-\otimes_\Bdr^{L\bs }\Bdre)$.
Then to see $\ul{R\Gamma}(Y_K/\Bdr_\Prism)$ underlies the usual prismatic cohomology, we apply $R\Gamma(\ast, -)$ at $\ul{R\Gamma}(Y_K/\Bdr_\Prism)$, and use a global version of \Cref{prism coh} and \Cref{Galois filtered}.
Here we notice that thanks to the compactness of the final object $\ast$, the functor $R\Gamma(\ast, -)$ commutes with the colimit in the formula above.

For (ii), we first consider Galois cohomology of the complex 
\[
C_e:=\left( \underset{n\in \mathbb{N}}{\colim} ~ \xi^{-n}\cdot\ul{\Fil}^n_{\inf}(Y_K/\Bdr) \right) \otimes_\Bdr^\bs \Bdre.
\]
By \Cref{Galois filtered}, applying $R\Hom_{\mathbb{Z}[\ul G_k]}(\mathbb{Z}, -)$ at this colimit diagram, we get the following derived tensor product $R\Gamma(\ul G_k, \Bdr)\otimes_k C'$, where $C'$ is the following colimit
\[
\xymatrix{R\Gamma(Y/k_{\inf}, \mathcal{O}_{Y/k}) & R\Gamma(Y/k_{\inf}, \mathcal{I}^1_{Y/k}) \ar[l] \ar[d]^{\id} & & \\
	& R\Gamma(Y/k_{\inf}, \mathcal{I}^1_{Y/k})  & R\Gamma(Y/k_{\inf}, \mathcal{I}^2_{Y/k}) \ar[d]^{\id} \ar[l] &\\
	& & R\Gamma(Y/k_{\inf}, \mathcal{I}^2_{Y/k}) & R\Gamma(Y/k_{\inf}, \mathcal{I}^3_{Y/k}) \ar[l] \ar[d]^{\id} \\
	& & & \vdots }.
\]
Notice that from the diagram it is clear that $C'\cong R\Gamma(Y/k_{\inf}, \mathcal{O}_{Y/k})$, so we have
\[
R\Hom_{\mathbb{Z}[\ul G_k]}(\mathbb{Z}, C_e)\cong R\Gamma(\ul G_k, \Bdr)\otimes_k  R\Gamma(Y/k_{\inf}, \mathcal{O}_{Y/k}).
\]

At last, by construction the condensed complex $\ul{R\Gamma}(Y_K/\Bdr_\Prism)$ is equal to $R\varprojlim_e C_e$.
Thus by switching the order of right derived functors, we get Galois invariant of condensed prismatic cohomology as below
\begin{align*}
	 R\Gamma(\ul G_k, \ul{R\Gamma}_\Prism(Y_K/\Bdr)) & \cong R\Hom_{\mathbb{Z}[\ul G_k]}(\mathbb{Z}, R\varprojlim_e C_e) \\
	 & \cong R\varprojlim_e R\Hom_{\mathbb{Z}[\ul G_k]}(\mathbb{Z}, C_e)\\
	 & \cong R\Gamma(\ul G_k, \Bdr)\otimes_k  R\Gamma(Y/k_{\inf}, \mathcal{O}_{Y/k}).
\end{align*}
\end{proof}
\begin{corollary}
	Let $Y$ be a proper rigid space over $k$ that has l.c.i singularities.
	Then Galois invariant of condensed prismatic cohomology $\ul{\mathrm{H}}^i_\Prism(Y_K/\Bdr)$ is finite dimensional over $k$.
\end{corollary}
\begin{proof}
	This follows from \Cref{Galois inv}, the finiteness of $R\Gamma(\ul G_k, \Bdr)$ in \Cref{Galois of Bdr}, and finiteness of infinitesimal cohomology $R\Gamma(Y/k_{\inf}, \mathcal{O}_{Y/k})$ for proper rigid space as in \cite{Guo21}.
\end{proof}

\end{document}